\documentclass{article}
\usepackage[utf8]{inputenc}
\usepackage{amsmath}
\usepackage{amsfonts, amsmath, amssymb, amscd, amsthm, color, graphicx, mathrsfs, wasysym, setspace, mdwlist, calc, float, mathtools, tikz-cd,enumitem}

\newcommand{\Conv}{\mathop{\scalebox{1.5}{\raisebox{-0.2ex}{$\ast$}}}}%

\usepackage{libertine,libertinust1math} 
\usepackage{tocloft}
\usepackage{comment}

\usepackage{xcolor}

\usepackage{hyperref}

\hypersetup{colorlinks,
    linkcolor={red!50!black},
    citecolor={blue!80!black},
    urlcolor={blue!80!black}}
\usepackage{float}

\numberwithin{equation}{section}

\renewcommand{\leq}{\leqslant}
\renewcommand{\geq}{\geqslant}

\renewcommand{\phi}{\varphi}
\newcommand{\WR}{\mathcal{WR}}

\newcommand{\bh}{\overline H}
\renewcommand{\d }{{\rm d} }
\renewcommand{\dh}{\widehat \d}

\newcommand{\e }{\varepsilon }
\renewcommand{\P }{\mathcal P}

\newcommand{\lab}{{\bf Lab}}
\renewcommand{\ll }{\langle\hspace{-.7mm}\langle }
\newcommand{\rr }{\rangle\hspace{-.7mm}\rangle }

\newcommand{\QN}{\mathcal {QN}}
\newcommand{\M}{\mathcal M}
\newcommand{\Nn}{\mathcal N}
\newcommand{\ra}{\rightarrow}
\newcommand{\ca}{\curvearrowright}

\newcommand{\C}{\mathcal C}
\newcommand{\D}{\mathcal D}

\newcommand{\K}{\mathcal K}
\renewcommand{\L}{\mathcal L}
\newcommand{\Q}{\mathcal Q}
\newcommand{\R}{\mathcal R}
\renewcommand{\S}{\mathcal S}

\newcommand{\sU}{\mathscr U}

\newcommand{\cM}{\mathcal M}

\newcommand{\oo}{\bar\otimes}
\newcommand{\cN}{\mathcal N}
\newcommand{\cQ}{\mathcal Q}
\newcommand{\cL}{\mathcal L}

\newcommand{\sH}{\mathscr H}

\newcommand{\liftS}{\varsigma}
\newcommand{\FD}{\mathsf D}
\newcommand{\FE}{\mathsf E}
\newcommand{\FF}{\mathsf F}
\newcommand{\ma}{\mathfrak{a}}
\newcommand{\mc}{\mathfrak{c}}
\newcommand{\ms}{\mathfrak{s}}
\newcommand{\mq}{\mathfrak{q}}
\newcommand{\hatrmd}{\widehat{\mathrm{d}}}

\newcommand{\Ng}{\overline}

\newcommand{\ZZ}{\mathbb Z}
\newcommand{\NN}{\mathbb N}
\usepackage{cleveref}

\theoremstyle{plain}
\newtheorem{thm}{Theorem}[section]

\newtheorem*{thm*}{Theorem}
\newtheorem{cor}[thm]{Corollary}
\newtheorem{lem}[thm]{Lemma}
\newtheorem{claim}[thm]{Claim}
\newtheorem{prop}[thm]{Proposition}

\theoremstyle{definition}
\newtheorem{defn}[thm]{Definition}

\newtheorem{ex}[thm]{Example}

\theoremstyle{remark}
\newtheorem{rem}[thm]{Remark}

\AddToHook{env/lem/begin}{\crefalias{thm}{lemma}}
\crefname{lem}{lemma}{lemmas}

\AddToHook{env/cor/begin}{\crefalias{thm}{corollary}}
\crefname{cor}{corollary}{corollaries}

\AddToHook{env/notation/begin}{\crefalias{thm}{notation}}
\crefname{notation}{notation}{notations}

\AddToHook{env/prop/begin}{\crefalias{thm}{proposition}}
\crefname{prop}{proposition}{propositions}

\AddToHook{env/claim/begin}{\crefalias{thm}{claim}}
\crefname{claim}{claim}{claims}

\AddToHook{env/rem/begin}{\crefalias{thm}{remark}}
\crefname{rem}{remark}{remarks}

\AddToHook{env/conv/begin}{\crefalias{thm}{convention}}
\crefname{cov}{convention}{conventions}

\AddToHook{env/assumption/begin}{\crefalias{thm}{assumption}}
\crefname{assumption}{assumption}{assumptions}

\newcommand{\innpr}[1]{\langle #1 \rangle}
\newcommand{\Norm}[1]{\lVert #1 \rVert} 

\newcommand{\acts}{\curvearrowright}

\DeclareMathOperator{\supp}{supp}
\DeclareMathOperator{\Ad}{Ad}
\DeclareMathOperator{\Aut}{Aut}

\renewcommand{\subset}{\subseteq}

\newcommand{\Path}{\mathfrak{P}}

\newcommand{\bg}{\overline{G}}
\newcommand{\pint}{\partial_{\mathrm{int}}}
\newcommand{\pext}{\partial_{\mathrm{ext}}}
\newcommand{\al}{\mathsf}
\newcommand{\family}{\mathfrak}

\DeclareMathOperator{\stab}{Stab}

\title{McDuff superrigidity for group II$_1$ factors}
\author{Juan Felipe Ariza Mej{\'i}a, Ionu\c t Chifan, Denis Osin, Bin Sun}
\date{}

\begin{document}
\maketitle

\begin{abstract} \noindent Developing new techniques at the interface of geometric group theory and von Neumann algebras, we identify the first examples of ICC groups
$G$ whose von Neumann algebras are McDuff and exhibit a new rigidity phenomenon, termed \emph{McDuff superrigidity}: an arbitrary group $H$ satisfying  $\L(H)\cong \L(G)$ decomposes as  $H \cong G\times A$ for an ICC amenable group $A$. 
    
\end{abstract}

\section{Introduction}

A II$_1$ factor $\M$ is called \emph{McDuff} if and only if it tensorially absorbs the hyperfinite factor $\mathcal R$ of Murray and von Neumann \cite{MvN43}, that is, $$\M \cong \M \bar\otimes \mathcal R.$$ This property first appeared in McDuff's work on central sequences \cite{MD69a}; in particular, she proved that $\M$ is McDuff if and only if its central sequence algebra contains two noncommuting unitaries \cite{MD69b}.

In this paper we study rigidity phenomena for McDuff factors associated to countable groups. Recall that the von Neumann algebra $\mathcal L(G)$ of a countable discrete group $G$ is defined as the weak operator closure of the image of the left regular representation $\lambda \colon \mathbb CG\to B(\ell^2(G))$ \cite{MvN36}. As observed by Murray and von Neumann, the von Neumann algebra $\mathcal L(G)$ is a II$_1$ factor precisely when $G$ satisfies the infinite conjugacy classes condition (abbreviated \emph{ICC}): the conjugacy class of every non-trivial element of $G$ is infinite. We say that a group $G$ is \emph{$W^*$-McDuff} if $\mathcal L(G)$ is a McDuff factor.

The class of $W^*$-McDuff groups  includes McDuff’s original family of examples \cite{MD69b}, Thompson’s group $F$ \cite{Jol98} and its generalizations $F_n$ \cite{Pic04}, Thompson-like groups from cloning systems \cite{BZ21}, the pure braid group on infinitely many strands $P_\infty$ \cite{CS20}, any ICC group of the form $G\times A$ with $A$ nontrivial amenable, and any infinite direct sum $\oplus_{n\in \mathbb N} G_n$ of nontrivial ICC groups.

By definition, for every $W^*$-McDuff group $G$ one has 
\begin{equation}\label{Eq:McDuffsr}
\L(G)\cong \L(G)\bar\otimes \mathcal R \cong \L(G\times A),
\end{equation}  
where $A$ is any ICC amenable group.
In particular, $W^*$-McDuff groups are never $W^*$-superrigid in the sense of Popa \cite{Po07}. However, it is natural to ask whether there exist $W^*$-McDuff groups for which (\ref{Eq:McDuffsr}) is the \emph{only} obstruction to $W^*$-superrigidity.  More formally, we introduce the following.

\begin{defn}
A group $G$ is said to be \emph{McDuff superrigid} if for any countable group $H$, the isomorphism $\L(G)\cong \L(H)$ implies the existence of an ICC amenable group $A$ such that $H\cong G\times A$.
\end{defn}

Instances of this phenomenon are expected to be quite rare and it is easy to construct $W^*$-McDuff groups that are not McDuff superrigid. 

\begin{ex}
Consider $G = \bigoplus_\mathbb{N} \mathbb{F}_m$ for some integer $m\ge 2$. Further, we fix some $k\in \mathbb N$ and let $n=(m-1)k^2+1$ and $H=\bigoplus_\mathbb{N} \mathbb{F}_n$. Then $G$ and $H$ are $W^*$-McDuff and we have $$\cL(\mathbb{F}_m) \cong \cL(\mathbb{F}_n)\otimes M_k(\mathbb{C}) $$ by the Voiculescu's formula for free group factors \cite[Theorem 3.3]{Voi89}. Taking tensor product over countably many copies of these algebras, we obtain $$\cL(G) \cong \cL(H) \bar{\otimes} \R \cong \cL(H).$$  However, an elementary group theoretic argument shows that we cannot have $H\cong G \times A$ for any amenable group $A$ if $m\ne n$. In particular, $G$  is not McDuff superrigid.
\end{ex}

The main goal of our paper is to construct the first examples of McDuff superrigid groups. More precisely, we prove the following. 

\begin{thm}\label{Thm:main}
There exists a family of countable $W^\ast$-McDuff groups $\{ G_i\}_{i\in I}$ such that $|I|=2^{\aleph_0}$, $G_i\not\cong G_j$ whenever $i\ne j$, and every $G_i$ is McDuff superrigid. 
\end{thm}

Our groups $G_i$ are infinite direct sums of $W^\ast$-superrigid groups with property $(T)$ constructed in \cite{CIOS21}. The proof of Theorem \ref{Thm:main} incorporates several new ideas on both the group-theoretic and analytic sides. We discuss these in more detail below.

\paragraph{Acknowledgments.} The first author was supported by NSF Grants FRG-DMS-1854194 and DMS-2154637, the Graduate College Post-Comprehensive Research Fellowship and the Erwin and Peggy Kleinfeld  Scholar Fellowship; the second author was supported by the NSF Grants DMS-2154637 and DMS-2452247; the third author is supported by the NSF grant DMS-2405032 and the Simons Fellowship in Mathematics MP-SFM-00005907. The fourth author is supported by the NSF grant DMS-2507047.

\section{Outline of the paper}

In this section, we discuss the main ideas leading to Theorem \ref{Thm:main} in more detail. We begin with the group-theoretic component of the proof, which provides new insight into the algebraic structure of group theoretic Dehn fillings and appears to be of independent interest. To set the stage, we first briefly review the motivation and relevant background.

\paragraph{2.1. Group theoretic Dehn filling and wreath-like products with bounded cocycle.} The classical Dehn surgery on a $3$-dimensional manifold $M$ consists of cutting off a solid torus -- intuitively, ``drilling" along an embedded knot in $M$ -- and then gluing it back in a ``twisted" way. The systematic study of this operation dates back to the foundational work of Dehn in the 1910s. The theory was further developed in the 1960s, culminating in the Lickorish–Wallace theorem: \emph{every closed, connected, orientable $3$-manifold can be obtained from the
$3$-sphere by performing finitely many Dehn surgeries. }

The second step of the surgery, known as \emph{Dehn filling}, admits the following formalization.
Let $M$ be a compact orientable $3$-manifold with toric boundary. Topologically distinct ways of attaching a solid torus to $\partial M$ can be classified by free homotopy classes of unoriented essential simple closed curves in $\partial M$, called {\it slopes}. Given a slope $s$, the corresponding   Dehn filling  $M(s)$ of $M$ is the manifold obtained from $M$ by attaching a solid torus to $\partial M$ so that the meridian of the torus goes to a simple closed curve of the slope $s$. A breakthrough result of Thurston \cite[Theorem~1.6]{Th} asserts that \emph{if $M\setminus\partial M$ admits a finite volume hyperbolic structure, then $M(s)$ is hyperbolic for all but finitely many slopes.} 

In group-theoretic settings, the role of the pair $\partial M\subset M$ is played by a pair of groups $H\leqslant G$ and the existence of a finite volume hyperbolic structure on $M\setminus\partial M$ translates to relative hyperbolicity of $G$ with respect to $H$. Under the assumptions of Thurston's theorem, we can think of $s$ as an element of $\pi_1(\partial M)\leqslant \pi_1(M)$, and by the Seifert--van Kampen theorem we have
\begin{equation}\label{Eq:pi1}
\pi_1(M(s))=\pi_1(M)/\ll s\rr.
\end{equation}
This suggests that the process of attaching a solid torus to $M$ must correspond to quotienting the group $G$ by normal subgroups generated by elements of $H$.

In what follows, we say that a certain property $P$ \emph{holds for all sufficiently deep normal subgroups of a group} $H$ if $P$ holds for every subgroup $N\lhd H$ avoiding a fixed set of non-trivial elements. The following algebraic analogue of Thurston's theorem was first obtained in \cite{Osi07} (an independent proof for torsion-free groups was given by Groves and Manning in \cite{GM}), and then extended to the more general case of weakly hyperbolic groups in \cite[Theorem 2.27]{DGO11}.

\begin{thm}\label{Thm:DF}
Suppose that a group $G$ is hyperbolic relative to a subgroup $H$. Then for any sufficiently deep subgroup $N\lhd H$, the natural map $H/N\to G/\ll N\rr$ (where $\ll N\rr$ denotes the normal closure of $N$ in $G$) is injective and $G/\ll N\rr$ is hyperbolic relative to $H/N$. 
\end{thm}

In \cite{DGO11}, the authors proved that the subgroup $\ll N\rr$ is a free product of conjugates of $N$ under the assumptions of Theorem \ref{Thm:DF}. In this paper, we further contribute to the study of the algebraic structure of Dehn fillings by showing that the extension 
$$
1\to \ll N\rr \to G\to G/\ll N\rr \to 1
$$ 
``almost splits". More precisely, the main group-theoretic result of this paper asserts that, under mild technical assumptions, there exists a ``canonical section" $G/\ll N\rr \to G$ such that every element in the image of the associated cocycle $G/\ll N\rr \times G/\ll N\rr \to \ll N\rr$ is a product of at most $6$ conjugates of elements of $N$ (see Theorem \ref{thm. general main}). Although this theorem is reminiscent in spirit of the Rips–-Sela technology of canonical representatives \cite{RS}, our proof is entirely different and relies on several novel ideas.

This result has an immediate application to the structure of wreath-like products naturally associated with Dehn fillings.  This class of groups was introduced in \cite{CIOS21} and used there and in subsequent papers \cite{CIOS23, CIOS24} to construct examples exhibiting various rigidity properties. 

\begin{defn}
    Recall that a group $G$ is a \emph{wreath-like product of groups $A$ and $B$ associated to an action $B \ca I$ on a certain set $I$} if there exists a short exact sequence
\begin{equation}\label{wrlext}
1 \ra \bigoplus_{i\in I} A_i \hookrightarrow G \overset{\varepsilon}{\twoheadrightarrow} B \ra 1,
\end{equation}
such that for all $g \in G$ and all $i \in I$ we have $A_i\cong A$ and $g A_i g^{-1} = A_{\varepsilon(g)i}$. We denote by $\mathcal{WR}(A, B \ca I)$ the collection of all wreath-like products as above.
\end{defn}  

 If the extension \ref{wrlext} splits, the group $G$ is simply the (permutational) wreath product $A{\, \rm wr}_I\, B$. In this paper, we introduce a subclass of $\mathcal{WR}(A, B \ca I)$ consisting of groups $G$ such that the extension (\ref{wrlext}) ``splits up to a bounded error".

\begin{defn}\label{def. cl}
We say that a wreath-like product $G\in \mathcal{WR}(A, B \ca I)$ has \textit{bounded cocycle} and write $G\in \mathcal{WR}_b(A, B \ca I)$ if there is a set-theoretic map $\liftS\colon B \rightarrow G$ such that 
$\varepsilon\circ\liftS=\mathrm{id}_B$ and the elements in the image of the associated cocycle $\alpha \colon B\times B\rightarrow \bigoplus_{i\in I}A_i$ defined by $$\alpha (x,y)=\liftS(x)\liftS(y)\liftS(xy)^{-1},\;\; \text{for all }\, x,y\in B$$ have uniformly bounded supports (considered as functions $I\to \bigcup_{i\in I}A_i$). That is, there is a constant $C\in\mathbb N$ such that, for any $x,y\in B$, we have 
$$
\left |\{ i\in I \mid \alpha(x,y) (i)\ne 1\}\right |\leqslant C.
$$
\end{defn}

In \cite{CIOS21}, we showed that certain groups naturally associated with group theoretic Dehn fillings have wreath-like product structure. In particular, this is the case when $G$ is a torsion-free hyperbolic group, $H=\langle g\rangle $ is a maximal cyclic subgroup of $G$, and $N=\langle g^n\rangle $ for some large $n\in \mathbb N$. In this paper, we show that these groups actually have bounded cocycle whenever $n$ is prime. More precisely, we prove the following. 

\begin{thm}\label{Thm:main-gt}
Let $G$ be a torsion-free hyperbolic group. Suppose that $g$ is a non-trivial element of $G$ that is not a proper power. Then for any sufficiently large prime $n\in \NN$, we have $$ G/[\ll g^n\rr, \ll g^n\rr] \in \mathcal{WR}_b(\ZZ, G/\ll g^n\rr \ca I),$$ where the action of $G/\ll g^n\rr $ on $I$ is transitive with stabilizers isomorphic to $\ZZ/n\ZZ$. \\ In addition, $G/\ll g^n\rr $ is a non-elementary, ICC, hyperbolic group.\end{thm}
\noindent For a more comprehensive version of this result, the reader is referred to Theorem~\ref{Thm:main-gt-full} in Section~\ref{sec. group main}.

Informally, Definition \ref{def. cl} asserts that groups from $\mathcal{WR}_b(A, B \ca I)$ are algebraically close to the corresponding (permutational) wreath products $A{\, \rm wr}_I\, B$. Note that $A{\, \rm wr}_I\, B$ never has property (T) of Kazhdan if $A$ is non-trivial and $I$ is infinite. In contrast, applying Theorem \ref{Thm:main-gt} to a property (T) group $G$, we obtain plenty of wreath-like products with bounded cocycle that inherit property (T) from $G$. The existence of groups combining these seemingly incompatible features plays a crucial role in our proof of Theorem~\ref{Thm:main}.

\paragraph{2.2. Arrays and $W^*$-rigidity of infinite products.}
Another key step in our approach to establishing McDuff superrigidity is a rigidity theorem in the spirit of~\cite[Theorem A]{CU18} for infinite direct sums of property (T) groups $$G_n\in \mathcal{WR}_b(A_n, B_n\ca I_n), \;\; n\in \mathbb N,$$ where the groups $A_n$, $B_n$, and the actions $B_n\ca I_n$ satisfy certain additional conditions.

Recall that \cite[Theorem A]{CU18} established the following \emph{infinite product rigidity} phenomenon: if $G=\bigoplus_{n\in \mathbb N} G_n$ and each $G_n$ is an ICC, property~(T), biexact, weakly amenable group, then every group $H$ satisfying $\mathcal{L}(G)\cong \mathcal{L}(H)$ decomposes as $H=\Big(\bigoplus_{n} H_n\Big)\oplus A$,
where $A$ is an ICC amenable group, and for each $n$ we have $\mathcal{L}(G_n)\cong \mathcal{L}(H_n)$ up to amplification. Moreover, by \cite{CJ85, DP24, CH89}, each $H_n$ is again ICC, property~(T), biexact, and weakly amenable.

In the recent work \cite{DD25}, Ding and Drimbe strengthened this theorem by removing the weak amenability assumption on the groups $G_n$. Their proof also avoids relying on Popa--Vaes’s classification of normalizers \cite{PV12}, replacing it with a suitable form of relative biexactness  in the context of infinite direct sum groups.

However, none of these prior results apply to the case when the direct summands are our wreath-like product groups. It is currently unknown whether these groups are biexact, and they are very likely not weakly amenable. As a result, establishing the above infinite product rigidity requires substantial new ideas beyond the existing technology. Although we follow the general strategy of \cite{CU18}, and to some extent the product-reconstruction framework of \cite{CdSS15}, at the technical level we must introduce several new methods to compensate for the lack of biexactness in our setting. Below we outline some of the principal novel features of our approach.

Whereas in \cite{CU18} the groups $G_n$ were assumed to be weakly amenable and bi-exact in order to apply the relative strong solidity theorem of \cite[Theorem 1.4]{PV12}, in our setting only the acting groups $B_n$ are required to satisfy these two properties. This essential distinction manifests itself in a key step of the proof in Section~\ref{sec. vNa main} (corresponding to \cite[Section 3]{CU18}): there, we need an algebra to intertwine into
$$
\cL(G)\,\bar{\otimes}\,\cL\!\left(\bigoplus_{n\neq j} G_n\right)
$$
inside $\cL(G)\,\bar{\otimes}\,\cL(G)$. However, the relative strong solidity result only yields an intertwining into $$
\cL(G)\,\bar{\otimes}\,\cL\!\left(\big(\oplus_{n\neq j} G_n\big) \oplus A_j^{(I_j)}\right).$$ To overcome this difficulty, we exploit the fact that the $2$-cocycle with uniformly bounded support naturally gives rise to an array on $G_n$ into a weakly--$\ell^2$ representation that, in a precise sense, detects the support of the $2$-cocycle (see Definition~\ref{arraydef} and Theorem~\ref{supportarray}). We then combine this \emph{support array} with a proper array into a weakly--$\ell^2$ representation of $B_n$ (coming from bi-exactness of $B_n$) to construct arrays on groups in $\mathcal{WR}_b(A_n, B_n \curvearrowright I_n)$
(where the action $B_n\curvearrowright I_n$ has amenable stabilizers). These arrays enjoy the property that balls of finite radius lie inside a finite union of translates of elements in $A_n^{(I_n)}$ whose supports have uniformly bounded size (with the bound depending on the radius). A further novel ingredient of our approach is a method for tensoring infinitely many such arrays while maintaining control over finite-radius balls. We emphasize that while finite products of proper arrays had been constructed in \cite{CS11} to obtain relatively proper arrays, no infinite product construction was previously available. In our case, it is precisely the special structure of the groups in $\mathcal{WR}_b(A_n, B_n \curvearrowright I_n)$ that enables this construction. 

The second key role of the arrays constructed for groups in $\mathcal{WR}_b(A,B\ca I)$ comes in Theorem \ref{subrel} where a correspondence between the rate of convergence of the Gaussian deformation (as introduced in \cite{CS11}) and the tensor length deformation from \cite{Io06} is obtained.  Thus, by combining some of the methods from \cite{CU18} with new developments in the infinitesimal analysis of arrays, initiated in \cite{CS11, Io06}, we are able to show the following result, which parallels \cite[Theorem A]{CU18} and \cite[Theorem 6.2]{DD25}.

\begin{thm}\label{main1} For every $j\in \mathbb{N}$, let $A_j$ be nontrivial amenable group, $B_j$ an ICC subgroup of a hyperbolic group, and $B_j \ca I_j$ an action on a set $I_j$ with amenable stabilizers. Further, let $G_j \in \mathcal {WR}_b(A_j, B_j\ca I_j)$ be any groups with property (T), and let $G=\bigoplus_{j\in \mathbb{N}}G_j$. 

Suppose that $H$ is any group satisfying $\L(G)=\L(H)$. Then $H$ admits a direct sum decomposition $H=(\bigoplus_{j\in \mathbb{N}} H_j) \oplus A$, where $H_j$ is an ICC property (T) group for all $j\in \mathbb{N}$ and $A$ is an ICC amenable group. Moreover, there is a sequence $(t_j)_{j\in \mathbb{N}}\subset \mathbb R_+^*$ such that for every $k\in \mathbb{N}$ one can find a unitary $u_k\in \L(H)$ such that
\begin{equation}\label{intertwiningtheparts-main1}\begin{split}& \L(G_j)^{t_j}=u_k \L(H_j)u^*_k\quad \text{ for all } j\leqslant  k \text{; and }\\& \L(\oplus_{j> k} G_j)=u_k\L((\oplus_{j> k} H_j) \oplus A)u^*_k.\end{split}
\end{equation}
\end{thm}

\noindent Further combining Theorem \ref{main1} with the W$^*$-superrigidity results for the property (T) wreath-like product groups with abelian base from \cite[Theorem 1.3]{CIOS21} we obtain many examples of groups satisfying the aforementioned McDuff rigidity phenomenon.

\begin{cor}\label{main2} Under the assumptions of Theorem \ref{main1}, suppose that all groups $A_j$ are abelian. Then for any group $H$ such that $\L(G)\cong \L(H)$, one can find an ICC amenable group A such that $H \cong G \times A$.
\end{cor}

Our main result, Theorem \ref{Thm:main}, is essentially a combination of Theorem \ref{Thm:main-gt}, Corollary \ref{main2}, and elementary additional arguments. 

\section{Group-theoretic preliminaries}\label{sec. preliminary}

\subsection{Notation}

Recall that, for a subset $S$ of a group $G$, we denote by $\ll S \rr$ the normal subgroup of $G$ generated by $S$. If $T$ is another subset of $G$, we let
\[[S,T]=\{sts^{-1}t^{-1}\mid s\in S, t\in T\}.\]

If $H\leqslant G$ is a subgroup, we denote the centralizer of $H$ in $G$ by $C_G(H)$. For an action $G\curvearrowright I$ and $i\in I$, we denote by $\stab_G(i)$ the stabilizer of $i$ in $G$.

\subsection{Van Kampen diagrams}\label{subsec.vk}
We will prove boundedness of certain cocycles by constructing specific van Kampen diagrams, which will be recalled in this subsection. Our main reference is \cite{ol2012geometry}.

For the rest of the current section, let $G$ be a group. Fix a presentation
\begin{equation}\label{eq. pre G}
G=\langle\al{A}\mid\al{R}\rangle,
\end{equation}
where $\al{R}$ is a symmetric set of words in $\al{A}$ (i.e., for every $w\in\al{R}$, every cyclic shift of $w$ or $w^{-1}$ belongs to $\al{R}$).

\begin{defn}\label{def. vk diagram}
    A \textit{van Kampen diagram} $\Delta$ over \eqref{eq. pre G} is a finite oriented connected planar $2$-complex with labels on its oriented edges such that the following items \ref{item. diagram 1}, \ref{item. diagram 2}, \ref{item. diagram 3} and \ref{item. diagram 4} hold.
\end{defn}

\begin{enumerate}[label=(\alph*)]
\item\label{item. diagram 1} Each oriented edge of $\Delta$ is labeled by a letter in $\al{A}\sqcup\{1\}$.
\item\label{item. diagram 2} If an oriented edge $e$ of $\Delta$ has label $a\in\al{A}\sqcup\{1\}$, then $e^{-1}$ has label $a^{-1}$, where $e^{-1}$ (resp. $a^{-1}$) is the inverse of $e$ (resp. $a$).
\end{enumerate}

Here, we adopt the convention $1=1^{-1}$.

Let $p=e_1\cdots e_k$ be a path in a van Kampen diagram over \eqref{eq. pre G}. The initial (resp. terminal) vertex of $p$ is denoted as $p^-$ (resp. $p^+$). The \textit{label} of $p$, denoted as $\lab(p)$, is obtained by concatenating the labels of the edges of $p$ and then omit all $1$. So $\lab(p)\in \al A^\ast$. The length of $p$ is $\ell(p)=\|\lab(p)\|$.

Edges labeled by letters from $\al{A}$ are called \textit{essential edges}, while edges labeled by the letter $1$ are called \textit{non-essential edges}. A \textit{face} of $\Delta$ is a $2$-cell of $\Delta$. Let $\Pi$ be a face of $\Delta$, the boundary of $\Pi$ is denoted as $\partial\Pi$. Likewise, the boundary of $\Delta$ is denoted by $\partial\Delta$. Note that if we choose a base point for $\partial\Pi$ (resp. $\partial\Delta$), then $\partial\Pi$ (resp. $\partial\Delta$) becomes a path in $\Delta$. For a word $w$ over $\al{A}$, we use the notation $\lab(\partial\Pi)\equiv w$ (resp. $\lab(\partial \Delta)\equiv w$) to indicate that one can pick a base point to turn $\partial\Pi$ (resp. $\partial\Delta$) into a path $p$ so that $\lab(p)\equiv w$.

\begin{enumerate}[label=(\alph*)]
\setcounter{enumi}{2}
\item\label{item. diagram 3} For every face $\Pi$ of a van Kampen diagram $\Delta$ over the presentation $\eqref{eq. pre G}$, at least one of the following conditions (c1) and (c2) holds:
\item[(c1)] $\lab(\partial\Pi)$ is equal to an element of $\al{R}$.
\item[(c2)] $\partial\Pi$ either consists entirely of non-essential edges or consists of exactly two essential edges with mutually inverse labels (in addition to non-essential edges).
\end{enumerate}

A face satisfying (c2) is called a \textit{non-essential face}. All other faces are called \textit{essential faces}. The process of adding non-essential faces to a van Kampen diagram is called a \textit{$0$-refinement}. The interested readers are referred to \cite{ol2012geometry} for a formal discussion. 

\begin{enumerate}[label=(\alph*)]
\setcounter{enumi}{3}
\item\label{item. diagram 4} Each face is homeomorphic to a disc, i.e., its boundary has no self-intersection.
\end{enumerate}

The well-known van Kampen lemma states that a word $w$ over $\al{A}$ represents $1$ in $G$ if and only if there is a van Kampen diagram $\Delta$ over \eqref{eq. pre G} such that $\Delta$ is homeomorphic to a disc (such diagrams are called \textit{disk diagrams}), and that $\lab(\partial\Delta)\equiv w$.

\begin{rem}\label{mapofvk}
For a disk diagram $\Delta$ and a vertex $O$ of $\Delta$, there exists a unique continuous map $\mu$ from the $1$-skeleton of $\Delta$ to $\mathrm{Cay}(G,\al{A})$ sending $O$ to the identity vertex, preserving the labels of the essential edges and collapsing non-essential edges to vertices.
\end{rem}

\subsection{Hyperbolically embedded subgroups}

We begin by recalling some definitions and notation used throughout the rest of the paper.  

By a \emph{generating alphabet} $\al A$ of a group $G$, we mean an abstract set given together with a map $\al A\to G$ whose image generates $G$; to simplify our notation, we do not distinguish between elements of $\al A$ and their images in $G$ whenever no confusion is possible. Moreover, we will denote by $\al A^\ast$ the free monoid on $\al A$, i.e., the set of all words (including the empty word) over the alphabet $\al A$.

By the \emph{Cayley graph} of $G$ with respect to a generating alphabet $\al A$, denoted $\mathrm{Cay}(G, \al A)$, we mean a graph with the vertex set $G$ and the set of edges defined as follows. For each $a\in \al A$ and each $g\in G$, there is an oriented edge $e$ going from $g$ to $ga$ in $\mathrm{Cay}(G, \al A)$ and labelled by $a$. 

Given a combinatorial path $p$ of $\mathrm{Cay}(G,\al A)$, we will denote the combinatorial inverse of $p$ as $p^{-1}$, and let $p^-$ (resp. $p^+$) be the initial (resp. terminal) vertex of $p$. Note that, being vertices of $\mathrm{Cay}(G,\al A)$, $p^+$ and $p^-$ are in fact elements of $G$, and their inverses are denoted by $(p^+)^{-1}$ and $(p^-)^{-1}$, respectively (not to be confused with $p^{-1}$). The \textit{label} of $p$, denoted $\lab(p)$, is obtained by concatenating the labels of the edges in $p$. The \textit{length} of $p$, denoted $\ell(p)$, is the number of letters in $\lab(p)$. If $w$ is a word over $\al{A}$, then $\|w\|$ will denote the length of $w$, i.e., the number of letters in $w$. Therefore, $\ell(p)=\|\lab(p)\|$. Moreover, if $v$ is another word over $\al{A}$ (resp. $g$ is an element of $G$), then we write $w\equiv v$ to indicate a letter-by-letter equality, and $w=_G v$ (resp. $w=_G g$) to indicate that $w$ and $v$ represent the same element in $G$ (resp. $w$ represents $g$ in $G$). For a subset $S\subset G$, we will write $w\in_G S$ (resp. $w\not\in_G S$) to indicate that $w$ represents an element of $G$ that is in (resp. not in) $S$. In contrast, for a subset $T\subset \al A^\ast$, we will write $w\in T$ to indicate that $w$ is an element of $T$.

For the rest of the current section, let $H\leqslant G$ be a subgroup, and let $X\subseteq G$ be a subset such that $X\cup H$ generates $G$. We think of $X$ and the subgroup $H$ as abstract sets and consider  the alphabet 
\begin{equation}\label{eq. generating alphabet}
\al A= X \sqcup H.
\end{equation}

We think of the Cayley graph $\mathrm{Cay}(H, H)$ as the subgraph of $\mathrm{Cay}(G,\al{A)}$.

\begin{defn}\label{def. relative metric}
For any $g,h\in H$, let $\widehat{\mathrm{d}} (g,h)$ be the length of a shortest path in $\mathrm{Cay}(G,\al{A}) $ that connects $g$ to $h$ and contains no edges of $\mathrm{Cay}(H, H)$. If no such path exists, we set $\widehat{\mathrm{d}} (g,h)=\infty $. Further, if $p$ is a path between two vertices $g$ and $gh$, we define $\widehat{\ell}(p)=\widehat{\mathrm{d}}(1,h)$.
\end{defn}

Clearly $\widehat{\mathrm{d}}$ satisfies the triangle inequality (with addition extended to $[0, \infty]$ in the natural way). The metric $\dh$ is called the \textit{relative metric} with respect to $X$.

We are now ready to define hyperbolically embedded subgroups introduced in \cite{DGO11}.
 
\begin{defn}\label{def. he}
The subgroup $H$ is \emph{hyperbolically embedded  in $G$ with respect to a subset $X\subseteq G$}, denoted $H \hookrightarrow_h (G,X)$, if the group $G$ is generated by the alphabet $\al A$ defined by \eqref{eq. generating alphabet} and the following conditions hold.
\begin{enumerate}
\item[(a)] The Cayley graph $\mathrm{Cay}(G,\al{A}) $ is a hyperbolic metric space.
\item[(b)] For each $n\in \mathbb N$, the set $\left\{ h\in H\mid \widehat{\mathrm{d}}(1,h)\leqslant n\right\}$ is finite.
\end{enumerate}
\end{defn}

\subsection{Isolated components and geodesic polygons}\label{sec. isolated components}

\begin{defn}\label{def. component}
Let $p$ be a path in $\mathrm{Cay}(G,\al{A})$. An $H$\textit{-subpath} $q$ of $p$ is a subpath of $p$ such that $\lab(q)\in H^\ast$ (if $p$ is a \textit{loop}, i.e., $p^-=p^+$, we allow $q$ to be a subpath of some cyclic shift of $p$). An $H$-subpath $q$ of $p$ is called an $H$-\textit{component} if $q$ is not properly contained in any other $H$-subpath. 

We say that two $H$-components $q_1,q_2$ of $p$ are \textit{connected} if there exists a path $c$ in $\mathrm{Cay}(G,\al{A})$ such that $c$ connects a vertex of $q_1$ to a vertex of $q_2$, and that $\lab(c)$ is a letter from $H$. An $H$-component $q$ of $p$ is called \textit{isolated} if it is not connected to any other $H$-component of $p$.
\end{defn}

\begin{defn}
Suppose that $p$ is a path in $\mathrm{Cay}(G,\al{A})$ that is not a loop and $q_1,q_2$ are two $H$-components of $p$. We write $$q_1<_p q_2$$ if $p$ first visits $q_1$ and then $q_2$.
\end{defn}

The lemma below follows immediately from Definition \ref{def. component} and will be used routinely in our paper.

\begin{lem}\label{lem. component and geodesic}
    If $p$ is a geodesic in $\mathrm{Cay}(G, \al A)$, then any $H$-component of $p$ is isolated in $p$ and consists of a single edge labeled by an element of  $H\smallsetminus\{1\}$.
\end{lem}

The next proposition was first proved for relatively hyperbolic groups in \cite{Osi07},  and then generalized in \cite{DGO11}.

\begin{prop}[{\cite[Proposition 4.14]{DGO11}}]\label{prop. geodesic polygon}
Suppose that $\mathrm{Cay}(G,\al A)$ is a hyperbolic metric space. Then there exists a constant $D>0$ satisfying the following property. Let $p$ be an $n$-gon in $\mathrm{Cay}(G,\al{A})$ with geodesic sides $p_1,...,p_n$. Let $I\subset\{1,2,...,n\}$ be a subset such that $p_i$ is an isolated $H$-component of $p$ for all $i\in I$. Then
\[\sum_{i\in I}\widehat{\ell}(p_i)\leqslant Dn.\]
\end{prop}

To stress the fundamental role played by this proposition in our paper, we introduce the following. 

\begin{defn}\label{Def:pc}
We say that $D\geqslant 0$ is a \textit{polygonal constant} for the Cayley graph $Cay (G, \al A)$ if it satisfies the conclusion of  Proposition \ref{prop. geodesic polygon}.
\end{defn}

\subsection{Separating cosets}\label{sec. separating coset}

We recall the notion and basic properties of separating cosets in this subsection. Our main reference is \cite{hull2013induced}. For the rest of this section, we will let $D$ be a polygonal constant given by \Cref{prop. geodesic polygon}, and fix $C>6D$.

\begin{defn}
    We say that a path $p$ in $\mathrm{Cay}(G,\al A)$ \textit{penetrates} an $H$-coset $gH$ if $p$ has an $H$-component $q$ whose endpoints lie in $gH$. If in addition, $\widehat\ell(q)>C$, then we say $p$ \textit{essentially penetrates} $gH$.

    If $p$ is a path that every left $H$-coset at most once (for example, this is the case if $p$ is a geodesic, by \Cref{lem. component and geodesic}), we will say that $q^-$ (resp. $q^+$) is the vertex where $p$ \textit{enters} (resp. exits) $gH$.
\end{defn}

\begin{defn}\label{def. essentially penetrate}
    For $f, g\in G$, let $S'(f,g)$ be the set of left $H$-cosets $kH$ such that there exists a geodesic $p$ in $\mathrm{Cay}(G,\al A)$ such that $p^-=f,p^+=g$ and $p$ essentially penetrates $kH$. Let also $S(f,g)=S'(f,g)\cup fH\cup gH$.
\end{defn}

\begin{rem}
    For technical reasons, our definition of $S(f,g)$ is slightly different from that of \cite{hull2013induced}. But the following results still hold true.
\end{rem}

\begin{lem}\label{lem. basic of separating cosets}
    \begin{enumerate}[label=(\roman*)]
        \item\label{item. keep order of separating cosets} (\cite[Lemma 3.7]{hull2013induced}) For $f, g\in G$, there is a linear order $<_{f,g}$ on $S(f,g)$ such that the following hold. Let $k_1H<_{f,g}k_2H\in S(f,g)$. Then for every geodesic $p$ in $\mathrm{Cay}(G,\al A)$ with $p^-=f,p^+=g$, we have that $p$ penetrates both $k_1H$ and $k_2H$, and $k_1H<_p k_2H$.
        
        \item\label{item. separating cosets flip} (\cite[Lemma 3.2 (a)]{hull2013induced}) $S(f,g)=S(g,f)$. Moreover, for $k_1H\neq k_2H\in S(f,g)$, we have $k_1H<_{f,g} k_2H$ if and only if $k_1H>_{g,f} k_2H$.
        \item\label{item. separating cosets equivariant} (\cite[Lemma 3.3 (b)]{hull2013induced}) $S(kf,kg)=\{kxH\mid xH\in S(f,g)\}$. Moreover, $k_1H<_{f,g} k_2H$ if and only if $kk_1H<_{kf,kg} kk_2H$

        \item\label{item. 2 exception} (\cite[Lemma 3.9]{hull2013induced}) Let $f,g,k\in G$. Then $S(f,k)$ decomposes as

        \[S(f,k)=S_1\sqcup S_2\sqcup S_3,\]
        where $S_1\subset S(f,g)\smallsetminus S(g,k),S_3\subset S(g,k)\smallsetminus S(f,g)$ and $|S_2|\leqslant 2$.

        \item\label{item. decomposition and order} Under the setting of item \ref{item. 2 exception}, $S_1$ (resp. $S_3$) is an initial (resp. a terminal) segment of $<_{f,k}$.
    \end{enumerate}
\end{lem}

\begin{rem}
    In \cite[Lemma 3.7]{hull2013induced}, the order $<_{f,g}$ is only defined on $S'(f,g)$. We extend this order by letting $fH$ (resp. $gH$) be the least (resp. greatest) element in $S(f,g)$.
\end{rem}

\begin{proof}[Proof of \Cref{lem. basic of separating cosets}]
    It remains to prove item \ref{item. decomposition and order}. We prove that $xH<_{f,k}yH$ for any $xH\in S_1$ and $yH\in S_2$. The proof for the other cases are the same.

    So suppose, to the contrary, that $xH>_{f,k}yH$. We claim that $yH\in S(f,g)$: Let $p$ be a geodesic from $f$ to $k$ that essentially penetrates $yH$, let $p_1$ be the subpath of $p$ between $f$ and the vertex $e^-$ where $p$ enters $xH$, and let $q$ be a geodesic from $f$ to $g$. By item \ref{item. keep order of separating cosets}, $q$ penetrates $xH$. Let $q_1$ be the subpath of $q$ between $g$ and the vertex $e^+$ where $q$ exits $xH$. Let $e$ be an edge connecting $e^-$ and $e^+$. 
    
    We claim that the concatenation $p_1eq_1$ is a geodesic that connects $f,g$: Let $s_1$ be the subpath of $q$ between $f$ and $t^-$, the vertex where $q$ enters $xH$, let $p_2$ be the subpath of $p$ between $k$ and $t^+$, the vertex where $p$ exits $xH$, and let $t$ be an edge connecting $t^-$ and $t^+$. As $p$ is geodesic, we have $\ell(s_1)+\ell(t)+\ell(p_2)\geqslant \ell(p)$, which yields $\ell(s_1)\geqslant \ell(p_1)$. Thus, $\ell(p_1eq_1)\leqslant \ell(s_1)+1+\ell(q_1)=\ell(q)$.
    
    By construction, the geodesic $p_1eq_1$ essentially penetrates $yH$. As $xH>_{f,k}yH$, the proof of \cite[Lemma 3.9]{hull2013induced} would include $yH$ in $S_1$, a contradiction.
\end{proof}

\subsection{Diagram surgery}\label{sec. surgery}

Let $\al R\subset \al A^\ast$ be a symmetric subset so that $G$ is given by the presentation \eqref{eq. pre G} and $\al R$ contains all words over the subalphabet $H\subset \al A$ that represent $1$ in $G$. For the rest of the current section, let $N\lhd H$, let $\bg=G/\ll N \rr$, and let $\al S$ be the set consisting of all words over $H$ representing elements of $N$ in $G$. 

We recall the techniques for diagram surgery in this subsection. Our main reference is \cite{DGO11}. Let $\family{D}$ be the set of all van Kampen diagrams $\Delta$ over \eqref{eq. pre G} satisfying the following:

\begin{enumerate}[label=(D\arabic*)]
\item\label{item. D1} Topologically $\Delta$ is a disc with $k\geqslant 0$ holes. The boundary of $\Delta$ can be decomposed as $\partial \Delta=\pext\Delta\cup\pint\Delta$, where $\pext\Delta$ is the boundary of the disc, and $\pint\Delta$ consists of disjoint cycles (connected components) $c_1,...,c_k$ that bound the holes.
\item\label{item. D2} For $i=1,...,k$, $c_i$ is labeled by a word from $\al S$.
\item\label{item. D3} $\Delta$ is equipped with a \textit{cut system} that is a collection $T=\{t_1,...,t_k\}$ of disjoint paths (\textit{cuts}) $t_1,...,t_k$ in $\Delta$ without self-intersections such that, for $i=1,...,k$, the two endpoints of $t_i$ belong to different components of $\partial\Delta$, and that after cutting $\Delta$ along $t_i$ for all $i = 1,...,k$, one gets a disc van Kampen diagram $\widetilde{\Delta}$ over \eqref{eq. pre G}.
\end{enumerate}

\begin{lem}[{\cite[Lemma 3.17]{sun2018cohomologyi}}]\label{Lem:diag}
    A word $w\in \al A^\ast$ satisfies $w=_{\bg} 1$ if and only if there exists a van Kampen diagram $\Delta\in\family{D}$ such that $\lab(\pext\Delta)\equiv w$. 
\end{lem}

\begin{defn}
    For a word $w\in\al A^\ast$ such that $w=_{\bg}1$, we will denote the set of all diagrams $\Delta\in\family D$ such that $\lab(\pext\Delta)\equiv w$ by $\family{D}(w)$.
\end{defn}

\begin{lem}[{\cite[Lemma 7.11]{DGO11}}]\label{lem. adding path}
Let $\Delta\in\family{D}$, let $\widetilde{\Delta}$ be the disc van Kampen diagram resulting from cutting $\Delta$ along its set of cuts. Define $\kappa\colon \widetilde{\Delta}\rightarrow\Delta$ to be the map that ``sews'' the cuts.  Fix an arbitrary vertex $O$ in $\widetilde{\Delta}$ and let $\mu$ be the map sending $\widetilde\Delta^{(1)}$ to $\mathrm{Cay}(G,\al{A})$ as described by \Cref{mapofvk}.

Let $a,b$ be two vertices on $\partial\Delta$ and let $\widetilde{a},\widetilde{b}$ be two vertices on $\partial\widetilde{\Delta}$ such that $\kappa(\widetilde{a})=a,\kappa(\widetilde{b})=b$. Then for any path $p$ in $\mathrm{Cay}(G,\al A)$ connecting $\mu(\widetilde{a})$ to $\mu(\widetilde{b})$, there is a diagram $\Delta_1\in\family{D}$ with the following properties:

\begin{enumerate}
\item[(a)] $\Delta$ and $\Delta_1$ have the same boundary and cut system. By this we mean the following: Let $K$ (resp. $K_1$) be the subgraph of $\Delta^{(1)}$ (resp. $\Delta_1^{(1)}$), the $1$-skeleton of $\Delta$ (resp. $\Delta$), consisting of $\partial\Delta$ (resp. $\partial\Delta_1$) and all cuts of $\Delta$ (resp. $\Delta_1$). Then there is a graph isomorphism $K_1\rightarrow K$ that preserves orientation and labels and maps the cuts of $\Delta_1$ to the cuts of $\Delta$ and $\partial\Delta_1$ to $\partial\Delta$.
\item[(b)] There is a path $q$ in $\Delta_1$ without self-intersections and does not intersect $\pext\Delta_1$ such that (1) $q$ connects $a$ and $b$, (2) $q$ has no common vertices with the cuts of $\Delta_1$ except possibly for $a,b$, and (3) $\lab(q)\equiv \lab(p)$.
\end{enumerate}
\end{lem}

\begin{defn}
By an $H$\textit{-subpath} of $\Delta$ or $\pext\Delta$ we mean any path of $\Delta$ whose label belongs to $H^\ast\smallsetminus\{1\}^\ast$. By an \textit{$H$-component} of $\pext\Delta$ we mean an $H$-subpath of $\pext\Delta$ that is not contained in any other $H$-subpath of $\pext\Delta$.

We say two $H$-subpaths $p,q$ of $\Delta$ are \textit{connected} if there exist $H$-components $a,b$ in $\widetilde{\Delta}$ such that $\kappa(a)$ (resp. $\kappa(b)$) is a subpath of $p$ (resp. $q$), and that $\mu(a),\mu(b)$ are connected in $\mathrm{Cay}(G,\al{A})$ (in the sense of \Cref{def. component}).
\end{defn}

\begin{rem}
The definitions of $H$-subpaths and connected $H$-subpaths in $\partial\Delta$ for a van Kampen diagram $\Delta\in\family{D}$ do not depend on the pre-chosen vertex $O$.
\end{rem}

\begin{defn}\label{type of diagrams}
For a diagram $\Delta\in \family{D}$, the \textit{type} of $\Delta$ is defined by the formula
\[\tau(\Delta)=(g(\Delta),\sum_{i=1}^{g(\Delta)}\ell(t_i)),\]
where $g(\Delta)$ is the number of holes in $\Delta$ and $t_1,t_2,...,t_{g(\Delta)}$ form the cut system of $\Delta$.

We order the types of diagrams in $\family{D}$ lexicographically: $(k_1,\ell_1)<(k_2,\ell_2)$ if and only if either $k_1<k_2$ or $k_1=k_2$ and $\ell_1<\ell_2$.
\end{defn}

The proof of the lemma below is the same as the one for \cite[Lemma 7.17 (b)]{DGO11}.

\begin{lem}\label{lem. surgery}
Suppose that $H\hookrightarrow_h (G,X)$ and $\widehat{\mathrm{d}}(1,n)>(4+k)D$ (recall that $D$ is a polygonal constant constant given by \Cref{prop. geodesic polygon}), and $w\in\al{A}^{\ast}$ is the label of a path $p$ in $\mathrm{Cay}(G,\al{A})$ that is the concatenation of at most $k$ geodesics. Also suppose that $w\neq_G 1$ and $w=_{\bg}1$. Let $\Delta$ be a diagram in $\family{D}(w)$ of minimal type. Then there exists a connected component $c$ of $\pint\Delta$ such that $c$ is connected (as an $H$-subpath of $\Delta$) to an $H$-subpath of $\pext\Delta$.
\end{lem}

\begin{lem}\label{lem. good diagram}
Suppose that $\rm{Cay}(G,\al A)$ is hyperbolic, $\widehat{\mathrm{d}}(1,n)>(4+k)D$, and $w\in\al{A}^{\ast}$ is the label of a $k$-gon $p$ in $\mathrm{Cay}(\bg,\al{A})$ with geodesic sides. Then there exists a van Kampen diagram $\Delta\in\family D(w)$ such that the following hold.

\begin{enumerate}
    \item[(i)] Each component $c$ of $\pint\Delta$ (if there is any) is a loop with a single edge labeled by a letter in $N\smallsetminus\{1\}$. Moreover, there exists a unique cut $t_c$ in the cut system of $\Delta$ such that $t^-_c\in c, t^+_c=q^-_c$ and $t_c$ is a non-essential edge, where $q_c$ is an $H$-subpath of $\pext\Delta$.
    \item[(ii)] $\Delta$ is of minimal type. In particular, if $\pint\Delta$ contains two components $c_1\neq c_2$, then $c_1$ is not connected to $c_2$.
\end{enumerate}
\end{lem}

\begin{defn}
    We say a van Kampen diagram $\Delta$ \emph{good} if it satisfies the conclusions (i) and (ii) of \Cref{lem. good diagram}.
\end{defn}

\begin{proof}[Proof of \Cref{lem. good diagram}]
    We induct on $g(w)$. The base case $g(w)=0$ is trivial.

    Suppose that the lemma holds for all $u$ with $g(u)<g(w)$. Let $\Delta_1\in\mathcal{D}(w)$ of minimal type. By \Cref{lem. surgery}, we may assume that there is a component $c$ of $\pint\Delta_1$ such that $\Delta_1$ contains a path $e$ connecting $c$ and an $H$-component $q$ of $\pext\Delta_1$ such that $e$ is a single edge with $\lab(e)\in H, e^+=q^-$ and $e$ does not self-intersect or intersect the cut system $T$ of $\Delta_1$. By cutting $\Delta_1$ along $e$, we obtain a van Kampen diagram $\Delta_2$ with one less hole. 

    We have $\lab(e^{-1}ceq)=_G h$ for some $h\in H$. Let $\Delta_3$ be a disk diagram over \eqref{eq. pre G} with $\partial\Delta_2=e^{-1}ceqe_1$, where $e_1$ is an edge with $\lab(e_1)\equiv h$. By gluing $\Delta_2$ to $\Delta_3$ we obtain a van Kampen diagram $\Delta_4$ with one less hole than $\Delta_1$. Let $u\equiv \lab(\pext\Delta_4)$. We then have $g(u)<g(w)$.

    To apply the induction hypothesis to $u$, we need to prove that $u$ labels a $k$-gon with geodesic sides. Let $r$ be the subpath of $\pext\Delta_1$ such that $r$ contains $q$ and $\lab(r)$ labels a geodesic side of $p$. The side $r$ gives rise to a path $s$ in $\pext\Delta_4$. Then $\lab(s)$ labels a geodesic in $\rm{Cay}(\bg,\al A)$, as this path has the same length and endpoints as the geodesic labeled by $\lab(r)$. Note also that each of the other geodesic sides of $p$ gives rise to a subpath of $\pext\Delta_4$ whose label labels a geodesic in $\rm{Cay}(\bg,\al A)$. So $u$ labels a $k$-gon in $\rm{Cay}(\bg,\al A)$ with geodesic sides.

    The induction hypothesis then gives us a good diagram $\Delta_5\in\family D(u)$. Note that $\lab(e^{-1}ce)=_G n_1$ for some $n_1\in N$. Let $\Delta_6$ be the disk van Kampen diagram over \eqref{eq. pre G} such that $\partial\Delta_6=e_2e_3e_4qe_1$, where $e_3$ is an edge with $\lab(e_3)\equiv n_1$ and $e_2,e_4$ are non-essential edges. By gluing $\Delta_6$ with $\Delta_5$, and then gluing $e_2$ with $e_4$, we obtain a diagram $\Delta$ with $\pext\Delta\equiv w$. By an abuse of notation, we denote the edge in $\Delta$ resulting from $e_2$ also by $e_2$. Let $T_4$ be the cut system of $\Delta_4$. Then $T:=T_4\cup\{e_4\}$ is a cut system for $\Delta$. As 

    \[g(\Delta)=1+g(\Delta_5)\leqslant 1+g(\Delta_2)=g(\Delta_1)\]
    and each cut in the cut system of $\Delta$ has zero length, $\Delta$ is a good diagram in $\family{D}(w)$.
\end{proof}

\begin{defn}
    Let $\Delta$ be a good diagram such that $\lab(\pext\Delta)$ labels a $k$-gon with geodesic sides in $\rm{Cay}(\bg,\al A)$. By cutting $\Delta$ along all cuts, we obtain a disk van Kampen diagram $\widetilde\Delta$. Each component $c$ of $\pint\Delta$ gives rise to a subpath of $\partial\widetilde\Delta$ with label $\lab(cq_c)$, where $q_c$ is the $H$-component of $\pext\Delta$ connected to $c$. Note that $\lab(cq_c)$ represents an element $h_c\in H$. By replacing each $\lab(cq_c)$ by $h_c$ in $\lab(\partial\widetilde \Delta)$, we obtain a word $w_\Delta$ that labels a $k$-gon $p_\Delta$ with geodesic sides in $\rm{Cay}(G,\al A)$. For simplicity, we will call $p_\Delta$ \textit{a polygon obtained by applying a cutting and replacing process} to $\Delta$.
\end{defn}

\subsection{Sufficiently long Dehn filling}

Assume now that $G$, $H$, $X$ are as above and $N\lhd H$. Note that $\al A=X\sqcup H$ is a generating alphabet of $\bg=G/\ll N\rr$ via the composition $\al A\to G\to \bg$. So it makes sense to talk about $\rm{Cay}(\bg,\al A)$, the Cayley graph of $\bg$ with respect to $\al A$. The following result can be extracted from Theorem 7.15 in \cite{DGO11} and its proof. For the definition of a polygonal constant, see Definition \ref{Def:pc}.

\begin{thm}\label{thm. dehn filling}
    Suppose that $\rm{Cay}(G,\al A)$ is hyperbolic. There exist constants $D_0\leqslant \overline D$ and $\delta$ such that, for any $N\lhd H$ satisfying $\hatrmd(1,n)>D_0$ for all $n\in N\smallsetminus\{1\}$, the following hold.
\begin{enumerate}
\item[(a)] $H/N$ naturally embeds in $\bg $ and the Cayley graph $\rm{Cay}(\bg,\al A)$ $\delta$-hyperbolic. 

\item[(b)] $\overline D$ is a polygonal constant for $\rm{Cay}(\bg,\al A)$.
\end{enumerate}
\end{thm}

\begin{proof}[On the proof]
 Compared to \cite[Theorem 7.15]{DGO11}, we added the uniformness of the constants $\delta $ and $\overline D$. Tracing the constants in the proofs of Lemma 4.9 and Proposition 4.14 in \cite{DGO11}, the reader can verify that these constants only depend on the set $S$ of $H$-subwords that occur in relations from the set $\e_0(R)$ in the relative presentation (82) on p. 124 of \cite{DGO11} and the corresponding isoperimetric constant. It is easy to see that the set $S$ is independent of the choice of $N$. Further, the proof of Theorem 7.15 in \cite{DGO11} provides a uniform bound on the isoperimetric constant (see the inequality two lines above (82) on p. 124). Thus, the desired uniformness follows.
\end{proof}
    
Note that the constant $\overline D$ in the theorem above works, in particular, for the trivial filling, when $N=\{ 1\}$; thus, $\overline D$ is also a polygonal constant for $\rm{Cay}(G,\al A)$. 

\Cref{def. component}, applied to the group pair $(\bg,\bh)$ and the relative generating set $\overline X$, gives us the notion of an (isolated) $\bh$-component. For \Cref{lem. estimate of hat length} below, we emphasize that $\hatrmd\colon H\times H \to [0,\infty]$ stands for the relative metric on $H$ for the Cayley graph $\rm{Cay}(G,\al A)$ (and not for the Cayley graph $\rm{Cay}(\bg,\al A)$).

\begin{lem}\label{lem. estimate of hat length}
    Let $\Delta\in\family D$ be a good diagram such that $w:=\lab(\pext\Delta)$ labels a $k$-gon $p$ in $\rm{Cay}(G,\al A)$ with geodesic sides. Let also $q$ be an $H$-subpath of $\pext\Delta$ that is not connected to any component of $\pint \Delta$ and gives rise to an isolated $\bh$-component in $p$. Then

    \[\hatrmd(1,\lab(q))\leqslant (k+2)D.\]
\end{lem}

\begin{proof}
    By applying a cutting and replacing process to $\Delta$, we obtain a cycle $\widetilde p$ in $\rm{Cay}(G,\al A)$. We claim that $\widetilde p$ is a polygon with at most $k+2$ geodesic sides and $\lab(q)$ labels a side of $\widetilde p$ which is an isolated $H$-component.

    Let $\overline\mu\colon \Delta^{(1)}\to \rm{Cay}(\bg,\al A)$ be the map defined by \Cref{mapofvk}, and let $\overline r_1,...,\overline r_k$ be the subpaths of $\pext\Delta$ such that $\overline\mu(r_i)$ is a geodesic side of $\overline \mu(\pext\Delta)$ for each $i$. Assume, without loss of generality, that $q\subseteq r_1$. Under the cutting and replacing process, each of $r_2,...,r_k$ gives rise to a geodesic side of $\widetilde p$, as each $\lab(r_i)$ labels a geodesic in $\rm{Cay}(\bg,\al A)$. The path $r_1$ also gives rise to a geodesic side $s$ of $\widetilde p$, but in order to apply \Cref{prop. geodesic polygon}, we will think of $s$ as a concatenation $s_1es_2$, where $e$ is the edge resulting from $q$ and so $\lab(e)\equiv \lab(q)$, and $s_1,s_2$ are geodesics. Our assumption implies that $e$ is an isolated $H$-component of $\widetilde p$. Therefore, \Cref{prop. geodesic polygon} implies $\widehat\ell(e)\leqslant (k+2)D$.
\end{proof}

\subsection{Cohen--Lyndon triples}\label{sec. cl}

Cohen--Lyndon triples appear naturally in the context of Dehn filling and are key to our construction of wreath-like products. They were first studied by Cohen--Lyndon in \cite{cohen1963free}, hence the name. These triples serve as a rich source of wreath-like products. We will recall the definition and some properties in this subsection.

\begin{defn}\label{def. cohen-lyndon}
    Let $G$ be a group, $H\leqslant G$, $N\lhd H$. We say that $(G,H,N)$ is a \textit{Cohen--Lyndon triple} if there exists a left transversal $T$ of $H\ll N\rr$ in $G$ (recall that $\ll N\rr$ is the normal closure of $N$ in $G$) such that $\ll N \rr$ decomposes as a free product:
    \[\ll N \rr=\Conv_{t\in T}t N t^{-1}.\]
\end{defn}

\begin{prop}[{\cite[Theorem 5.1]{sun2018cohomologyi}}]\label{prop. cl}
    Suppose that $N$ is a normal subgroup of $H$ such that $\widehat{\mathrm{d}}(1,n)>24D$ for all $1\neq n\in N$. Then $(G,H,N)$ is a Cohen--Lyndon triple.
\end{prop}

Under the assumptions of \Cref{prop. cl}, let
\begin{equation}\label{eq. def S}
S=\big\{ [g_1n_1g^{-1}_1, g_2n_2g^{-1}_2] \mid n_1,n_2\in N \text{ and } g_1,g_2 \in G \text{ and } g_1H\ll N\rr \ne g_2H\ll N\rr\} .
\end{equation} 
Then $S$ is a normal subgroup of $G$ (see the discussion after \cite[Definition 4.12]{CIOS23}). Let $T$ be the transversal given by \Cref{def. cohen-lyndon}. Note that $T$ can be identified with the set of left cosets of $H\ll N \rr/\ll N \rr$ in $G/\ll N \rr$. So there is an action $G/\ll N \rr\curvearrowright T$ induced by left multiplication.

\begin{prop}[{\cite[Proposition 4.14]{CIOS23}}]\label{prop. wreath product}
If $\widehat{\mathrm{d}}(1,n)>24D$ for all $1\neq n\in N$, then $G/S\in\WR(N,G/\ll N \rr\curvearrowright T)$. More precisely, there is a short exact sequence
\[1 \rightarrow \bigoplus_{t\in T}tNt^{-1} \rightarrow G/S \xrightarrow{\epsilon} G/\ll N \rr \rightarrow 1\]
such that $g\cdot tNt^{-1} \cdot g^{-1}=(\epsilon(g)\cdot t)N(\epsilon(g)\cdot t)^{-1}$ for all $g\in G/S$.
\end{prop}

\begin{rem}\label{rem. stabilizer}
    In \Cref{prop. wreath product}, as the action $G/\ll N \rr\curvearrowright T$ is induced by left multiplication, for every $t\in T$, we have $\stab_{G/\ll N \rr}(t)\cong H\ll N \rr/\ll N \rr\cong H/N$, where the last isomorphism follows from \cite[Theorem 7.15]{DGO11}.
\end{rem}

\section{Wreath-like product groups with $2$-cocycles of uniformly bounded support}\label{sec. group main}

In this section, we prove the main group-theoretic result of our paper. 

\begin{thm}\label{thm. general main}
    Suppose that $G$ is a group, $H\leqslant G$ is a subgroup, and $X\subset G$ is a relative generating set with respect to $H$ such that $\mathrm{Cay}(G,X\sqcup H)$ is a hyperbolic metric space. Also let $\dh\colon H\times H\rightarrow [0,\infty]$ be the relative metric corresponding to $X$. Suppose that $N$ is a normal subgroup of $H$ and:

    \begin{enumerate}[label=(\roman*)]
        \item\label{item. no involution} $\bg:=G/\ll N \rr$ has no elements of order $2$;
        \item\label{item. malnormal} the image of $H$ is malnormal in $\bg$; and
        \item\label{item. long} for all $n\in N\smallsetminus\{1\}$, we have 
    \begin{equation}\label{eq. very long}
        \dh(1,n)>63\overline D, 
    \end{equation}
    where $\overline D$ is the constant provided by \Cref{thm. dehn filling}.
    \end{enumerate}
    Then there exists a set-theoretic section $\liftS\colon \bg\rightarrow G$ such that for all $\Ng f,\Ng g\in \bg$, the element $\liftS (\Ng f)\liftS(\Ng g)(\liftS(\Ng f \Ng g))^{-1}$ can be written as a product of at most six conjugates of elements of $N$, i.e., there exist $\{g_i\}^9_{i=1}\subset G, \{n_i\}^9_{i=1}\subset N$ such that

    \[\liftS(\Ng f)\liftS(\Ng g)(\liftS(\Ng f \Ng g))^{-1}=\prod^9_{i=1}g_in_ig^{-1}_i.\]
\end{thm}

In what follows, we work under the assumptions of Theorem \ref{thm. general main}.

\subsection{Equivariant systems of paths between cosets}

For $\Ng f,\Ng g\in\bg$, let $S(\Ng f,\Ng g)$ be defined by \Cref{def. essentially penetrate} with respect to the Cayley graph $\rm{Cay}(\bg,\al A)$ and the constant $C=7\overline D$. Note that our definition of $S(\Ng f,\Ng g)$ here aligns with our definition in \Cref{sec. separating coset}.

We will construct a path system $\Path$, which consists of paths in $\rm{Cay}(\bg,\al A)$ between left $\bh$-cosets of $\bg$ that are ``close to each other'' under the following metric:

\begin{defn}
    Let $\Ng f \bh \neq \Ng g\bh$ be left $\bh$-cosets. We will denote by $D(\Ng f\bh,\Ng g\bh)$ the set of pairs $(\Ng f_1,\Ng g_1)$ such that $\Ng f_1\in \Ng f\bh, \Ng g_1\in \Ng g\bh$ and 
    
    \[\overline \d(\Ng f_1,\Ng g_1)\leqslant \overline\d(\Ng f_2,\Ng g_2) \text{ for all }\Ng f_2\in \Ng f\bh, \Ng g_2\in \Ng g\bh.\]
    We define the \textit{coset distance} between $\Ng f\bh$ and $\Ng g\bh$ by 

    \[\overline\d_C(\Ng f\bh,\Ng g\bh)=\min_{(\Ng f_1,\Ng g_1)\in D(\Ng f\bh,\Ng g\bh)}|S(\Ng f_1,\Ng g_1)|-1.\]
\end{defn}

\begin{rem}
    The coset distance is symmetric.
\end{rem}

For the rest of this section, choose a subset $R\subset H$ of coset representatives of $N$ such that $R^{-1}=R$ and

\[\hatrmd(1,r)\leqslant \hatrmd(1,rn) \text{ for all } r\in R.\]
Note that such a set $R$ exists: In general, it might not be possible to have $R^{-1}=R$, which happens only under the existence of an element $h\in H\smallsetminus N$ such that $h^2\in N\smallsetminus\{1\}$. In our situation, \cite[Theorem 7.15]{DGO11} and item \ref{item. long} imply that the group $H/N$ embeds into $\bg$ as a subgroup. So $H/N$ has no order-$2$ elements as the same holds for $\bg$.

\begin{prop}\label{prop. equivariant paths}
There exists a set $\Path$ of geodesics of $\mathrm{Cay}(\bg,\al A)$ such that the following hold.
\begin{enumerate}[label=(\roman*)]
    \item\label{item. p exist and unique} For every ordered pair $(\Ng f\bh, \Ng g\bh)$ of left $\bh$-cosets such that $\Ng f\bh\neq \Ng g\bh$ and $\overline\d_C(\Ng f\bh,\Ng g\bh)=1$, there exists a unique element $p\in \Path$ such that $p$ is a geodesic in $\mathrm{Cay}(\bg,\al A)$, $p^-\in \Ng f\bh,p^+\in \Ng g\bh$ and $p$ does not essentially penetrate any left $\bh$-coset.
    \item\label{item. P flip} For every $p_1\in \Path$, let $p_2$ be the unique element of $\Path$ such that $p^-_2\in p^+_1\bh,p^+_2\in p^-_1\bh$, then $p_2=p^{-1}_1$.
    \item\label{item. P equivariant} For every $p\in \Path$ and $\Ng g\in \bg$, we have $\Ng g\cdot p\in \Path$.
    \item\label{item. R} For each $H$-subpath $q$ of each path $p\in \Path$, we have $\lab(q)\in R$.
\end{enumerate}
\end{prop}

To prove \Cref{prop. equivariant paths}, we start with an auxiliary result. Let
\[\overline{\FD}=\{\bh\Ng g\bh\mid \Ng g\in \bg\smallsetminus \bh\}\]
be the set of double cosets of $\bh$. Let $\FD$ be the subset of $\overline{\FD}$ consisting of double cosets $\bh\Ng g\bh$ such that there exists a geodesic $p$ in $\mathrm{Cay}(\bg,\al A)$ satisfying the following properties:

\begin{enumerate}[label=(P\arabic*)]
    \item\label{item. Domain 1} $p^-=1,p^+\in \bh\Ng g\bh$;
    \item\label{item. Domain 2} $p$ does not essentially penetrate any left $\bh$-coset; and
    \item\label{item. Domain 3} $\ell(p)\leqslant \d(1,\Ng f)$ for all $\Ng f\in \bh\Ng g\bh$.
\end{enumerate}
Note that $\FD\neq \emptyset$. Indeed, as $\bh\neq \bg$, there exists $\Ng x\in \overline X\smallsetminus \bh$, and we have $\bh \Ng x \bh\in \FD$.

\begin{lem}\label{lem. fundamental domain}
There exists a disjoint-union decomposition
\[\FD=\FD_+\sqcup \FD_-\]
such that for all $\bh\Ng g\bh\in\FD$, exactly one of $\bh\Ng g\bh$ and $\bh\Ng g^{-1}\bh$ belongs to $\FD_+$.
\end{lem}

\begin{proof}
Consider the family $\FF$ of pairs $(\FE_+,\FE_-)$ such that
\begin{enumerate}[label=(\arabic*)]
    \item $\FD=\FE_+\sqcup \FE_-$;
    \item if $\bh\Ng g\bh\in \FE_+$, then $\bh\Ng g^{-1}\bh\not\in\FE_+$;
\end{enumerate}

Note that $\FF$ is non-empty as $(\emptyset,\FD)\in \FF$. Note also that $\FF$ is partially ordered as follows: $(\FE_+,\FE_-)\prec (\FE'_+,\FE'_-)$ if and only if $\FE_+\subseteq \FE'_+$. Note that every totally ordered subset of $\FF$ has an upper bound, i.e., the one obtained by taking the union of the first factors and intersection of the second factors. So Zorn's lemma implies that $\FF$ has a maximal element $(\FD_+,\FD_-)$.

We claim that $(\FD_+,\FD_-)$ satisfies the requirement of the lemma. Suppose that there is $\bh\Ng g\bh\in \FD$ such that $\bh\Ng g\bh,\bh\Ng g^{-1}\bh\in \FD_-$. If $\bh\Ng g\bh\neq \bh\Ng g^{-1}\bh$, then by adding $\bh\Ng g\bh$ to $\FD_+$ and deleting $\bh\Ng g\bh$ from $\FD_-$ we get an element of $\FF$ greater than $(\FD_+,\FD_-)$, a contradiction. Therefore, we have $\bh\Ng g\bh=\bh\Ng g^{-1}\bh$, or $\bh\Ng g=\Ng g^{-1}\bh$. So there exist $\Ng h_1,\Ng h_2\in \bh$ such that $\Ng h_1\Ng g=\Ng g^{-1}\Ng h_2$ and thus
\[(\Ng h_1\Ng g)^2=\Ng h_1\Ng g\Ng g^{-1}\Ng h_2=\Ng h_1\Ng h_2\in \bh,\]
and
\[(\Ng h_1\Ng g)^2=\Ng g^{-1}\Ng h_2\Ng h_1\Ng g\in \Ng g^{-1}\bh \Ng g.\]
Therefore,
\[(\Ng h_1\Ng g)^2\in \bh\cap \Ng g^{-1}\bh \Ng g.\]
As $\bg$ does not have order-$2$ elements, $(\Ng h_1\Ng g)^2\neq 1$. The malnormality of $\bh$ then implies that $\Ng g\in \bh$, which contradicts the definition of $\FD$.
\end{proof}

Let $\Path_0$ be a set of paths in $\mathrm{Cay}(\bg,\al{A})$ such that for every $\bh\Ng g\bh\in \FD_+$, there is a unique path $p\in\Path_0$ that satisfies \ref{item. Domain 1}, \ref{item. Domain 2}, \ref{item. Domain 3}, and additionally

\begin{enumerate}[label=(P\arabic*)]
\setcounter{enumi}{3}
    \item\label{item. R 2} For each $H$-subpath $q$ of each path $p\in \Path_0$, we have $\lab(q)\in R$.
\end{enumerate}

Let
\[\Path=\{\Ng g\cdot p\mid p\in \Path_0\}\cup \{\Ng g\cdot p^{-1}\mid p\in \Path_0\}.\]

\begin{proof}[Proof of \Cref{prop. equivariant paths}]
By construction, $\Path$ satisfies items \ref{item. P equivariant} and \ref{item. R}. We prove that $\Path$ also satisfies items \ref{item. p exist and unique} and \ref{item. P flip}. Let $\Ng f\bh\neq \Ng g\bh$ be two distinct cosets of $\bh$ such that $\d_C(\Ng f\bh,\Ng g\bh)=1$. We first find a path $p_1\in\Path$ such that $p_1^-\in \Ng f\bh,p_1^+\in \Ng g\bh$.

Note that either $\bh\Ng f^{-1}\Ng g\bh$ or $\bh\Ng g^{-1}\Ng f\bh$ belongs to $\FD_+$. Without loss of generality we may assume $\bh\Ng g^{-1}\Ng f\bh\in\FD_+$. Then $\Path_0$ contains a geodesic $p_0$ with $p^-=1$ and $p^+\in \bh\Ng g^{-1}\Ng f\bh$. By multiplying $p^{-1}_0$ by some $\Ng h_0\in \bh$, we get a path $\Ng h_0\cdot p^{-1}_0\in\Path$ going from $\Ng g^{-1}\Ng f\bh$ to $\bh$. Then $p_1=\Ng g\Ng h_0\cdot p^{-1}_0\in \Path$ is a geodesic with $p^-_1\in\Ng f\bh$ and $p^+_1\in\Ng g\bh$.

Now, suppose that there is another path $p_2\in \Path$ with $p^-_2\in \Ng f\Ng H$ and $p^+_2\in \Ng g\Ng H$. There are two cases to consider:

\textbf{Case 1.} There exists a path $p_3\in \Path_0$ and an element $\Ng k_1\in \bg$ such that $p_2=\Ng k_1\cdot p_3$.

Then we have $\Ng k_1=\Ng k_1\cdot p^-_3=p^-_2\in \Ng f\bh$. So there exists some $\Ng h_1\in \bh$ such that $\Ng k_1=\Ng f\Ng h_1$. By construction, $p^+_3\in \bh\Ng k_2\bh$ for some $\bh\Ng k_2\bh\in \FD_+$. So there exist $\Ng h_2,\Ng h_3\in \bh$ such that $p^+_3=\Ng h_2\Ng k_2\Ng h_3$. We have 

\[\Ng f\Ng h_1\Ng h_2\Ng k_2\Ng h_3=\Ng k_1\Ng h_2\Ng k_2\Ng h_3=\Ng k_1\cdot p^+_3=p^+_2\in \Ng g\bh.\]

So $\Ng h_2\Ng k_2\Ng h_3\in \Ng h^{-1}_1\Ng f^{-1}\Ng g\bh\subset \bh\Ng f^{-1}\Ng g\bh$. It follows that $\bh\Ng f^{-1}\Ng g\bh=\bh\Ng k_2\bh\in \FD_+$. But then both $\bh\Ng g^{-1}\Ng f\bh$ and $\bh\Ng f^{-1}\Ng g\bh$ belong to $\FD_+$, which violates \Cref{lem. fundamental domain}.

\textbf{Case 2.} There exists a path $p_3\in \Path_0$ and an element $\Ng k_1\in G$ such that $p_2=\Ng k_1\cdot p^{-1}_3$.

Then we have $\Ng k_1=\Ng k_1\cdot p^-_3=p^+_2\in \Ng g\bh$. So there exists some $\Ng h_1\in \bh$ such that $\Ng k_1=\Ng g\Ng h_1$. By construction, $p^+_3\in \bh\Ng k_2\bh$ for some $\bh \Ng k_2\bh\in \FD$. So there exist $\Ng h_2,\Ng h_3\in \bh$ such that $p^+_3=\Ng h_2\Ng k_2\Ng h_3$. We have 

\[\Ng g\Ng h_1\Ng h_2\Ng k_2\Ng h_3=\Ng k_1\Ng h_2\Ng k_2\Ng h_3=\Ng k_1\cdot p^+_3=p^-_2\in \Ng f\bh.\]

So $\Ng h_2\Ng k_2\Ng h_3\in \Ng h^{-1}_1\Ng g^{-1}\Ng f\bh\subset \bh\Ng g^{-1}\Ng f\bh$. Therefore, $p_3$ is a path in $\Path_0$ that goes from $1$ to $\bh\Ng g^{-1}\Ng f\bh$. Note that $p_0$ is another such path. By the construction of $\Path_0$, we have $p_3=p_0$. So 

\[\lab(p_1)\equiv (\lab(p_0))^{-1}\equiv \lab(p_2).\]

Let $e_1$ (resp. $e_2$) be the edge of $\mathrm{Cay}(\bg,\al A)$ labeled by a letter $\Ng h'_1\in \bh$ (resp. $\Ng h'_2\in \bh$) such that $e^-_1=p^-_1$ (resp. $e^-_2=p^+_2$) and $e^+_1=p^-_2$ (resp. $e^+_2=p^+_1$). Then $e_1p_2e_2p^{-1}_1$ is a quadrilateral in $\mathrm{Cay}(\bg,\al A)$. So its label represents the identity in $\bg$, i.e.,

\[\Ng h'_1\lab(p_0)\Ng h'_2(\lab(p_0))^{-1}=_{\bg} 1.\]

Therefore, $\Ng h'_1\in \lab(p_0)\bh(\lab(p_0))^{-1}$. As $p_0\in \Path_0$ and $\bh\not\in \FD_+$, we have $\lab(p_0)\not\in_{\bg} \bh$. So the malnormality of $\bh$ implies that $\Ng h'_1=1$, which in turn implies that $\Ng h'_2=1$. So $p_1=p_2$.
\end{proof}

\subsection{Path representative}

\begin{prop}
    There exists a map $\phi$ from $\bg$ to the set of paths in $\rm{Cay}(\bg,\al A)$ satisfying the following. For all $\Ng g\in\bg$, the path $\phi(\Ng g)$ starts at $1$ and decomposes as

    \begin{equation}\label{Eq:phig}
    \phi(\Ng g)=e_0p_1e_1p_2\cdots p_{k+1}e_{k+1},
    \end{equation}
   where

    \begin{enumerate}
        \item[(i)] $k=|S(1,\Ng g)|-2$;
        
        \item[(ii)] for $i=0,...,k+1$, there is a unique $\bh$-coset $\Ng f\bh\in S(1,\Ng g)$ such that $e_i$ is an edge in $\Ng f\cdot\rm{Cay}(\bh,\bh)\subseteq \rm{Cay}(\bg,\al A)$;

        \item[(iii)] for $i=1,...,k$, $e_i$ is a non-trivial edge, although $e_0$ and $e_{k+1}$ might possibly be trivial;

        \item[(iv)] for $i=1,...,k+1$, $p_i$ is a path in $\Path$ connecting the elements of $S(1,\Ng g)$ containing $e_{i-1}$ and $e_i$, respectively;

        \item[(v)] the edges $e_0,...,e_{k+1}$ are all labeled by elements of $R$.
    \end{enumerate}
\end{prop}

\begin{proof}
    Fix any $\Ng g\in \bg$. Let

    \[\bh<_{1,\Ng g}\Ng g_1\bh <_{1,\Ng g}\cdots<_{1,\Ng g}\Ng g_k\bh <_{1,\Ng g} \Ng g\bh\]
    be the elements of $S(1,\Ng g)$. For $i=1,...,k+1$, let $p_i$ be the unique path in $\Path$ with $p^-_i\in \Ng g_{i-1}\bh$ and $p^+_i\in \Ng g_i$. For $i=1,...,k$, we let $e_i$ be the edge in $\Ng g_i\cdot\rm{Cay}(\bh,\bh)$ with $e^-_i=p^+_i,e^+_i=p^-_{i+1}$ and $\lab(e_i)\in R$. Let also $e_0$ (resp. $e_{k+1}$) be the possibly trivial edge in $\rm{Cay}(\bh,\bh)$ (resp. $\Ng g\cdot\rm{Cay}(\bh,\bh)$) with $e^-_0=1$ (resp. $e^-_{k+1}=p^+_{k+1}$), $e^+_0=p^-_1$ (resp. $e^+_{k+1}=\Ng g$) and $\lab(e_0)\in R$ (resp. $\lab(e_{k+1})\in R$). Finally, we define $\phi(\bg)$ by \eqref{Eq:phig}.
\end{proof}

\subsection{Some estimates}

Recall that $\family D$ denotes the set of the van Kampen diagrams over \eqref{eq. pre G} satisfying items \ref{item. D1}, \ref{item. D2}, and \ref{item. D3}. For this subsection, let $\Delta\in\family D$ be a good diagram that satisfies the following additional properties:

\begin{enumerate}
    \item[(i)] $\pext\Delta$ decomposes as a concatenation of six paths:

    \[\pext\Delta=e_1p_1e_2p_2e_3p_3.\]
    For $i=1,2,3$, $e_i$ is a (possibly non-essential) edge labeled by an element of $H$. $\lab(p_i)$ labels a path in $\rm{Cay}(\bg,\al A)$ that is a concatenation of at most five geodesics. The initial and terminal letters of $\lab(p_i)$ do not belong to $H$.

    \item[(ii)] $\pint\Delta\neq \emptyset$ and each component of $\pint\Delta$ is connected to an $H$-component of $\Delta$ lying in one of $p_1,p_2,p_3$.
\end{enumerate}

\begin{lem}\label{lem. connect means long}
    For any component $c$ of $\pint\Delta$, there is an $H$-component $q$ of $\pext\Delta$ connected to $c$ with $\widehat\ell(q)>7\overline D$.
\end{lem}

\begin{proof}
    We split our argument into three cases.

    \textbf{Case 1.} $c$ is connected to a unique $H$-component $q$ of $\pext\Delta$.

    The cut system of $\Delta$ contains an edge $e$ with $\lab(e)=1$ connecting $q^-$ with the vertex of $c$. We have $\lab(cq)=_G h$ for some $h\in H$. By applying a cutting and replacing process to $\Delta$, we obtain a polygon in $\rm{Cay}(G,\al A)$ with at most eighteen geodesic sides and with $h$ labeling an isolated $H$-component. By \Cref{lem. estimate of hat length}, we have $\hatrmd(1,h)\leqslant 20\overline D$.
    Thus,

    \[\widehat\ell(q)\geqslant \widehat\ell(c)-\hatrmd(1,h)> 43\overline D.\]

    \textbf{Case 2.} $c$ is connected to exactly two $H$-components of $\pext\Delta$.

    Without loss of generality, we may assume that $c$ is connected to $H$-components $q_1\subset p_1$ and $q_2\subset p_2$. The case where $q_1$ and $q_2$ lie on the same $p_i$ can be treated similarly. We may also assume that the cut corresponding to $c$ is an edge $e$ with $\lab(e)=1$ connecting $c$ with $q^-_1$. By \Cref{lem. surgery}, we may assume that $\Delta$ contains edges $t_1,t_2$ with $\lab(t_1),\lab(t_2)\in H$ and $t^-_1=q^+_1,t^+_1=q^-_2,t^-_2=q^+_2,t^+_2=q^-_1$, as shown by \Cref{fig. case 2}.

    \begin{figure}[ht!]
        \centering

\tikzset{every picture/.style={line width=0.75pt}} 

\begin{tikzpicture}[x=0.75pt,y=0.75pt,yscale=-1,xscale=1]

\draw    (324,237) -- (424,137) ;

\draw    (121,244) -- (84,176) ;
\draw    (121,244) -- (324,237) ;
\draw    (294,20) -- (366,20) ;
\draw    (366,20) -- (424,137) ;
\draw    (84,176) -- (127,142) ;
\draw [color={rgb, 255:red, 255; green, 0; blue, 0 }  ,draw opacity=1 ]   (127,142) -- (158,116) ;
\draw    (158,116) -- (196,86) ;
\draw [color={rgb, 255:red, 255; green, 0; blue, 0 }  ,draw opacity=1 ]   (196,86) -- (238,55) ;
\draw    (238,55) -- (294,20) ;
\draw [color={rgb, 255:red, 255; green, 0; blue, 0 }  ,draw opacity=1 ]   (158,116) .. controls (211,181) and (266,156) .. (196,86) ;
\draw [color={rgb, 255:red, 255; green, 0; blue, 0 }  ,draw opacity=1 ]   (127,142) .. controls (230,309) and (421,219) .. (238,55) ;
\draw   (147,152.5) .. controls (147,145.04) and (153.04,139) .. (160.5,139) .. controls (167.96,139) and (174,145.04) .. (174,152.5) .. controls (174,159.96) and (167.96,166) .. (160.5,166) .. controls (153.04,166) and (147,159.96) .. (147,152.5) -- cycle ;
\draw    (127,142) -- (148,146) ;

\draw (155.5,140) node [anchor=north west][inner sep=0.75pt]    {$c$};
\draw (346,139) node [anchor=north west][inner sep=0.75pt]    {$\Delta _{2}$};
\draw (184,107) node [anchor=north west][inner sep=0.75pt]    {$\Delta _{1}$};
\draw (126,106) node [anchor=north west][inner sep=0.75pt]    {$q_{1}$};
\draw (207,43) node [anchor=north west][inner sep=0.75pt]    {$q_{2}$};

\draw (264,187) node [anchor=north west][inner sep=0.75pt]    {$\Delta_3$};

\end{tikzpicture}
        \caption{Case 2}
        \label{fig. case 2}
    \end{figure}

    In \Cref{fig. case 2}, $\Delta_1$ is a subdiagram of $\Delta$. So $\Delta_1$ is a good diagram with $\lab(\pext\Delta_1)$ labels a polygon in $\rm{Cay}(\bg,\al A)$ with at most six geodesic sides. Moreover, $t_1$ cannot be connected to any component $c_1$ of $\pint\Delta_1$, as in that case $c_1$ would be connected to $c$, which contradicts our assumption that $\Delta$ is a good diagram. So $\widehat\ell(t_1)\leqslant 8\overline D$ by \Cref{lem. estimate of hat length}. Similarly, by considering the subdiagram $\Delta_2$, we get that $\widehat\ell(e_2)\leqslant 25\overline D$. 
    
    Note that the region cannot $\Delta_3$ cannot contain any component $c_1$ of $\pint\Delta$ other than $c$, as otherwise $c_1$ would be connected to $\pext\Delta_3$ by \Cref{lem. surgery}, and thus connected to $c$, violating our definition of a good diagram. Therefore, the word $\lab(q_1t_1q_2t_2c)$ labels a cycle in $\rm{Cay}(G,\al A)$. Assuming $\widehat\ell(p_1),\widehat\ell(p_2)\leqslant 7\overline D$, we get

    \[\widehat\ell(c)\leqslant 47\overline D,\]
    a contradiction with \eqref{eq. very long}.

    \textbf{Case 3.} $c$ is connected to three distinct $H$-components of $\pext\Delta$.

    This time, let us assume that $c$ is connected to three $H$-components of the same side $p_1$. The other case can be treated similarly. Let these $H$-components be $q_1,q_2,q_3$. By \Cref{lem. surgery}, we may assume that $\Delta$ contains edges $t_1,t_2,t_3$ such that $t^-_1=p^+_1,t^+_1=q^-_2,t^-_2=q^+_2,t^+_2=q^-_3,t^-_3=q^+_3,t^+_3=q^-_1$, as shown by \Cref{fig. case 3}.

    \begin{figure}[ht!]
        \centering

\tikzset{every picture/.style={line width=0.75pt}} 

\begin{tikzpicture}[x=0.75pt,y=0.75pt,yscale=-1,xscale=1]

\draw    (344,257) -- (444,157) ;

\draw    (141,264) -- (104,196) ;
\draw    (141,264) -- (344,257) ;
\draw    (314,40) -- (386,40) ;
\draw    (386,40) -- (444,157) ;
\draw    (104,196) -- (128,176) ;
\draw [color={rgb, 255:red, 255; green, 0; blue, 0 }  ,draw opacity=1 ]   (128,176) -- (164,148) ;
\draw    (164,148) -- (183,133) ;
\draw [color={rgb, 255:red, 255; green, 0; blue, 0 }  ,draw opacity=1 ]   (183,133) -- (223,101) ;
\draw    (223,101) -- (244,84) ;
\draw [color={rgb, 255:red, 255; green, 0; blue, 0 }  ,draw opacity=1 ]   (244,84) -- (285,53) ;
\draw    (285,53) -- (314,40) ;
\draw [color={rgb, 255:red, 255; green, 0; blue, 0 }  ,draw opacity=1 ]   (164,148) .. controls (219,227) and (239,180) .. (183,133) ;
\draw [color={rgb, 255:red, 255; green, 0; blue, 0 }  ,draw opacity=1 ]   (223,101) .. controls (271,179) and (305,155) .. (244,84) ;
\draw [color={rgb, 255:red, 255; green, 0; blue, 0 }  ,draw opacity=1 ]   (128,176) .. controls (363,353) and (402,170) .. (285,53) ;
\draw   (153,176) .. controls (153,171.03) and (157.03,167) .. (162,167) .. controls (166.97,167) and (171,171.03) .. (171,176) .. controls (171,180.97) and (166.97,185) .. (162,185) .. controls (157.03,185) and (153,180.97) .. (153,176) -- cycle ;
\draw    (128,176) -- (153,176) ;

\draw (154,170) node [anchor=north west][inner sep=0.75pt]    {$c$};
\draw (129,136) node [anchor=north west][inner sep=0.75pt]    {$q_{1}$};
\draw (185,91) node [anchor=north west][inner sep=0.75pt]    {$q_{2}$};
\draw (250,43) node [anchor=north west][inner sep=0.75pt]    {$q_{3}$};

\end{tikzpicture}
        \caption{Case 3}
        \label{fig. case 3}
    \end{figure}
    
    Similar to Case 2, by applying \Cref{lem. estimate of hat length} and assuming $\widehat\ell(p_1),\widehat\ell(p_2),\widehat\ell(p_3)\leqslant 7\overline D$, we get that $\widehat\ell(c)\leqslant 62\overline D$,
    which contradicts \eqref{eq. very long}.
\end{proof}

\subsection{Proof of \Cref{thm. general main}}

For every $\Ng g\in\bg$, let $\liftS(\Ng g)\in G$ be the element represented by $w_{\Ng g}:=\lab(\phi(\Ng g))$. We will prove that $\liftS(\Ng g)$ satisfies the conclusion of the theorem. Below, we fix $\Ng f,\Ng g\in \bg$.

\begin{claim}
    \Cref{thm. general main} will follow once we have constructed a good diagram $\Delta\in \family D(w_{\Ng f}w_{\Ng g}w^{-1}_{\Ng f\Ng g})$ with

\begin{equation}\label{eq. genus leq 9}
    g(\overline\Delta)\leqslant 9
\end{equation}
(Recall that $g(\overline\Delta)$ denotes the number of holes in $\overline\Delta$).
\end{claim}

\begin{proof}[Proof of the claim]
    Let $\widetilde\Delta$ be the disk diagram obtained by cutting $\overline\Delta$ along all cuts. Then $\lab(\partial \widetilde\Delta)$ can be decomposed as

    \[\lab(\partial \widetilde\Delta)\equiv w_0\cdot \prod^9_{k=1}(n_kw_k),\]
    where for $k=1,...,9$, $n_k$ is either $1$ or the label of some component of $\pint\overline\Delta$, and $w_0,...,w_9$ satisfies

    \begin{equation}\label{eq. product of w}
        \prod^9_{k=0}w_k\equiv \lab(\pext\overline\Delta)\equiv w_{\Ng f}w_{\Ng g}w^{-1}_{\Ng f\Ng g}.
    \end{equation}

    Note that $\lab(\partial \widetilde\Delta)$ can be rewritten as

    \begin{equation}\label{eq. rewritten}
        \lab(\partial \widetilde\Delta)\equiv \left(\prod^9_{k=1}\left((\prod^{k-1}_{\ell=0}w_{\ell}) n_k(\prod^{k-1}_{\ell=1}w_{\ell})^{-1}\right)\right)\cdot (\prod^9_{k=0}w_k).
    \end{equation}

    For $k=1,...,9$, let $g_k$ be the element of $G$ represented by $\prod^{k-1}_{\ell=1}w_{\ell}$. Using \eqref{eq. product of w}, equation \eqref{eq. rewritten} can be rewritten as

    \begin{equation*}
        \lab(\partial \widetilde\Delta)\equiv (\prod^9_{k=1}g_k n_k g^{-1}_k)\cdot w_{\Ng f}w_{\Ng g}w^{-1}_{\Ng f\Ng g}.
    \end{equation*}

    As $\widetilde\Delta$ is a disk diagram, we have $\lab(\partial\widetilde\Delta)=_G 1$. Thus,

    \[w_{\Ng f}w_{\Ng g}w^{-1}_{\Ng f\Ng g}=_G \prod^9_{k=1}(g_k n^{-1}_k g^{-1}_k)\]
    as desired.
\end{proof}

It remains to construct a good diagram $\Delta\in \family D( w_{\Ng f}w_{\Ng g}w^{-1}_{\Ng f\Ng g})$ such that $g(\Delta)\leq 9$. The word $w_{\Ng f}w_{\Ng g}w^{-1}_{\Ng f\Ng g}$ labels a triangle in $\rm{Cay}(\bg,\al A)$ with three sides $p_1,p_2,p_3$ labeled by $w_{\Ng f},w_{\Ng g},w^{-1}_{\Ng f\Ng g}$, respectively, and with $p^-_1=1$. Recall that \Cref{lem. basic of separating cosets} gives us sets

\[S_1\subset S(1,\Ng f)\smallsetminus S(1,\Ng f\Ng g),~~S_2\subset S(\Ng f,\Ng g)\smallsetminus S(1,\Ng f),~~S_3=S(1,\Ng f)\smallsetminus S(\Ng f,\Ng g)\]
such that

By \Cref{lem. basic of separating cosets}, we have

\begin{align}\label{eq. only 2}
|S(1,\Ng f)\smallsetminus(S_1\cup S_3)|\leqslant 2,\nonumber\\
|S(\Ng f,\Ng g)\smallsetminus(S_1\cup S_2)|\leqslant 2,\nonumber\\
|S(1,\Ng f\Ng g)\smallsetminus(S_2\cup S_3)|\leqslant 2.\nonumber\\
\end{align}

Recall that \Cref{lem. basic of separating cosets} defines orders $<_{1,\Ng f},<_{\Ng f,\Ng g},<_{1,\Ng f\Ng g}$ on $S(1,\Ng f),S(\Ng f, \Ng g),S(1,\Ng f\Ng g)$, respectively. Moreover, $S_1$ is a terminal segment of $<_{1,\Ng f}$, as well as an initial segment of $<_{\Ng f,\Ng g}$, and analogous statements hold for $S_2$ and $S_3$ as well.

Let $v_{11}$ (resp. $v_{12},v_{21},v_{22},v_{31},v_{32}$) be the vertex where $p_1$ (resp. $p_1,p_2,p_2,p_3,p_3$) exits (resp. enters, exits, enters, exits, enters) the greatest (resp. least, greatest, least, greatest, least) element of $S_3$ (resp. $S_1,S_1,S_2,S_2,S_3$) under the order $<_{1,\Ng f}$ (resp. $<_{1,\Ng f},<_{\Ng f,\Ng g},<_{\Ng f,\Ng g},<_{\Ng f\Ng g,1},<_{\Ng f\Ng g,1}$). Let $p_{11}$ (resp. $p_{12},p_{21},p_{22},p_{31},p_{32}$) be the subpath of $p_1$ (resp. $p_1,p_2,p_2,p_3,p_3$) between $p^-_1$ and $v_{11}$ (resp. $v_{12}$ and $p^+_1$, $p^-_2$ and $v_{21}$, $v_{22}$ and $p^+_2$, $p^-_3$ and $v_{31}$, $v_{32}$ and $p^+_3$).

\begin{claim}\label{cl. agree}
    Suppose that there exist $\Ng x\bh\neq \Ng y\bh\in S_2$ that are consecutive in $<_{\Ng f,\Ng g}$. Let $r_2$ (resp. $r_3$) be the subpath of $p_2$ (resp. $p_3$) connecting $\Ng x\bh,\Ng y\bh$ such that neither of the initial and terminal letters of $\lab(r_2)$ (resp. $\lab(r_3)$) is from $H$. Then $r_2=r^{-1}_3$.
\end{claim}

\begin{proof}[Proof of the claim]
    It suffices to prove that both $r_2$ and $r_3$ belong to $\Path$. Consider $r_2$. Let $r'_2$ be the subpath of $\phi(\Ng g)$ connecting $\Ng f^{-1}\Ng x\bh,\Ng f^{-1}\Ng y\bh$ such that neither of the initial and terminal letters of $\lab(r'_2)$ is from $H$. By definition, we have $r'_2\in \Path$. As $r_2=\Ng f\cdot r'_2$, \Cref{prop. equivariant paths} implies that $r_2\in \Path$. The proof for $r_3\in \Path$ is similar.   
\end{proof}

By \Cref{cl. agree}, the paths $p^{-1}_{22}$ and $p_{31}$ agree up to their last edges, which are labeled by elements of $H$. So there is a disk van Kampen diagram $\Delta_3$ with $\partial \Delta_1=p_{22}p_{31}e_3$ where $e_3$ is a possibly non-essential edge with $\lab(e_3)\in H$. Similarly, there is a disk van Kampen diagram $\Delta_2$ (resp. $\Delta_1$) with $\partial \Delta_2=p_{12}p_{21}e_2$ (resp. $\partial \Delta_1=p_{32}p_{11}e_1$) where $e_2$ (resp. $e_1$) is a possibly non-essential edge with $\lab(e_2)\in H$ (resp. $\lab(e_1)\in H$).

Let $p_4$ (resp. $p_5,p_6$) be the subpath of $p_1$ (resp. $p_2,p_3$) between $v_{11}$ and $v_{12}$ (resp. $v_{21}$ and $v_{22}$, $v_{31}$ and $v_{32}$). By an abuse of notation, we may view $e_1,e_2,e_3,p_1,p_2,p_3$ as paths of $\rm{Cay}(\bg,\al A)$ such that the concatenation $e_1p_4e_2p_5e_3p_6$ is a cycle in $\rm{Cay}(\bg,\al A)$, so its label represents $1$ in $\bg$. By \Cref{lem. good diagram}, there is a good diagram $\Delta_4$ with $\pext\Delta_4=e_1p_4e_2p_5e_3p_6$. 

Without loss of generality, we may assume that $\pint\Delta_4$ has three components $c_1,c_2,c_3$, connected respectively to $e_1,e_2,e_3$. As $\Delta_4$ is a good diagram, it contains cuts $t_1,t_2,t_3$ that respectively connect $e_1$ to $c_1$, $e_2$ to $c_2$, and $e_3$ to $c_3$. 

There exists $h_1\in H$ (resp. $h_2\in H,h_3\in H$) with $\lab(c_1e_1)=_G h_1$ (resp. $\lab(c_2e_2)=_G h_2,\lab(c_3e_3)=_G h_3$). Let $\Delta_5$ (resp. $\Delta_6,\Delta_7$) be a disc diagram over \eqref{eq. pre G} with $\partial \Delta_5=t^{-1}_1c_1t_1e_1e_4$ (resp. $\partial \Delta_6=t^{-1}_2c_2t_2e_2e_5, \partial \Delta_7=t^{-1}_3c_3t_3e_3e_6$) where $e_4$ (resp. $e_5,e_6$) is an edge with $\lab(e_4)=h^{-1}_4$ (resp. $\lab(e_5)=h^{-1}_5,\lab(e_6)=h^{-1}_6$). By cutting $\Delta_4$ along $t_1,t_2,t_3$ and gluing $\Delta_5,\Delta_6,\Delta_7$ to the resulting diagram, we obtain a good diagram $\Delta_8$ with $\pext\Delta_8=e_4p_4e_5p_5e_6p_6$.

By \eqref{eq. only 2} and the construction of $\phi$, each of $p_4,p_5,p_6$ is a concatenation of at most five geodesics. Each component of $\pint\Delta_8$ is connected to an $H$-component of $\pext\Delta_8$ in one of $p_4,p_5,p_6$ with $\widehat\ell$-length strictly greater than $7\overline D$ by \Cref{lem. connect means long}. Different components of $\pint\Delta_8$ (as they are not connected to each other) are connected to different $H$-components of $\pext\Delta_8$. By \eqref{eq. only 2}, the union $p_4\cup p_5\cup p_6$ only has at most six $H$-components with $\widehat\ell$-length strictly greater than $7\overline D$. Therefore, there are at most six components of $\pint\Delta_8$. Therefore, $\pint\Delta_4$ has at most nine components.

By gluing $\Delta_1,\Delta_2,\Delta_3,\Delta_4$, we obtain a good diagram $\Delta\in \family D( w_{\Ng f}w_{\Ng g}w^{-1}_{\Ng f\Ng g})$ with $g(\Delta)\leqslant 9$, finishing the proof.

\subsection{Proof of Theorem \ref{Thm:main-gt}}

For our purpose, we will need the following more precise version of Theorem \ref{Thm:main-gt}, which includes certain results proved in \cite{Ols93} and \cite{CIOS21}. Recall that a hyperbolic group is called \textit{elementary} if it is virtually cyclic.

\begin{thm}\label{Thm:main-gt-full}
Let $G$ be a non-cyclic, torsion-free, hyperbolic group. Suppose that $g$ is a non-trivial element of $G$ that is not a proper power. Then for any sufficiently large prime $n\in \NN$, the following hold.
\begin{enumerate}
    \item[(a)] $ G/[\ll g^n\rr, \ll g^n\rr] \in \mathcal{WR}_b(\ZZ, G/\ll g^n\rr \ca I)$, where the action of $G/\ll g^n\rr $ on $I$ is transitive with stabilizers isomorphic to $\ZZ/n\ZZ$. 
    \item[(b)] $G/\ll g^n\rr $ is a non-elementary, ICC, hyperbolic group.
    \item[(c)] The image of the element $g$ in $G/\ll g^n\rr$ has order $n$ and every non-trivial element of $G/\ll g^n\rr$  has order $n$ or $\infty$.
\end{enumerate}
\end{thm}

\begin{proof}
Let $G$ be a torsion-free hyperbolic group, $g\in G$ a non-trivial element that is not a proper power. Let $\e\colon G\to \overline G:=G/\ll g^n\rr$ be the natural homomorphism. In \cite[Theorem 3]{Ols93}, Olshanskii proved that, for any sufficiently large natural number $n$, $\overline G$ is non-elementary hyperbolic, the image of the element $g$ in $G/\ll g^n\rr$ has order $n$,  and every finite order element of $\overline G$  is the image of a finite order element of $G$ or is conjugate to a power of $g$. Since $G$ is torsion-free and $n$ is prime, every non-trivial finite order element of $\overline G$  has order $n$. Further, by \cite[Theorem 2.6 (c)]{CIOS21}, $\overline G$ is ICC. This proves parts (b) and (c) of our theorem. 

To prove (a), let $K=G/[\ll g^n\rr, \ll g^n\rr]$ be the wreath-like product given by \Cref{prop. wreath product} with respect to the triple $\langle g^n\rangle \lhd \langle g\rangle \hookrightarrow_h G$. We first note that $K \in \mathcal{WR}(\ZZ, G/\ll g^n\rr \ca I)$, where the action of $\bg $ on $I$ is transitive with stabilizers isomorphic to $\ZZ/n\ZZ$ by \cite[Theorem 2.6]{CIOS21}; furthermore, the base of the wreath-like product $K$ canonically decomposes as the sum of conjugates of $\e(\langle g^n\rangle)$.  Thus, we only need to prove that, as a wreath-like product, the group $K$ has bounded cocycle. To this end, we first show that Theorem \ref{thm. general main} applies in our situation. 

Let $H=\langle g\rangle$. Then $H\hookrightarrow_h (G,X)$ for some subset $X\subset G$ by \cite[Theorem 6.8]{DGO11}. By part (b) of Definition \ref{def. he}, the subgroup $N=\langle g^n \rangle\lhd H$ satisfies condition \ref{item. long} in Theorem \ref{thm. general main} for all large enough $n$. Further, \ref{item. no involution} follows from part (b) of our theorem if $n>2$. It remains to show that \ref{item. malnormal} also holds. The proof is a straightforward combination of elementary arguments with the proof of part (7) of \cite[Theorem 3]{Ols93}. For convenience of the reader, we provide the proof here; to keep it reasonably short, we do not state all necessary definitions and theorems, and refer to \cite{Ols93} for details.  

Let $\overline H=H/N\leqslant \overline G$. Suppose that $t^{-1}\bh t\cap \bh \ne \{ 1\}$ for some $t\in G\setminus H$. Since $\bh \cong \ZZ/n\ZZ$ and $n$ is prime, we must have $t^{-1}\bh t=\bh$. If $t$ has finite order, it has order $n$ by (c) and, therefore, cannot act as a non-trivial automorphism of $\bh$; this implies $t\in C_{\bg}(\bh)$. If $t$ has infinite order, $t^{n-1}\in C_{\bg}(\bh)$ (by the Lagrange theorem for $Out(\ZZ/n\ZZ)\cong \ZZ/ (n-1)\ZZ$). In either case, we have that $C_{\bg}(\e(g))$ is strictly larger than $\e(C_G(g))= \e(H)$.

Fix some finite generating set $Y$ of $G$ and let $W$ be a shortest word in the alphabet $Y\cup Y^{-1}$ representing $g$ in $G$. Since $E(g)=\langle g\rangle$, \cite[Lemma 4.1 (2)]{Ols93} is applicable to the set of all cyclic shifts of words $W^{\pm n}$. Repeating the proof of part (7) of \cite[Theorem 3]{Ols93} verbatim, we conclude that $C_{\bg}(\e(g))=\e(C_G(g))$, which contradicts the fact established in the previous paragraph. This contradiction shows that $\bh $ must be malnormal in $\bg$, i.e., condition \ref{item. malnormal} also holds.

Thus, Theorem \ref{thm. general main} applies and immediately gives us that $\overline G$ has bounded cocycle with respect to the section $\overline G\to G$ provided by the theorem.
\end{proof}

\section{Arrays on wreath-like products}

Sinclair and the second named author introduced in \cite{CS11} a coarse version of quasicocycles on a countable group, called \emph{arrays}. For reader's convenience we recall the definition below.

\begin{defn}\label{arraydef}
Let $\pi\colon G \rightarrow \mathcal{O}(\mathscr{H}_\pi)$ be an orthogonal representation of a countable discrete group $G$.  A map $\mq\colon G \rightarrow \mathscr{H}_\pi$ is called an \textit{array} if for every $h,k\in G$ there exists $K\geqslant 0$ such that $$\sup_{ g \in G} \Norm{\pi_h(\mq(g) ) - \mq(h g k)} \leqslant K.$$
Throughout the paper we will denote the bounded equivariance constant, or defect, by  $K=: d^\mq(h,k)$. An array is said to be \textit{proper with respect to a family} $\S$ \textit{of subgroups of} $G$ if the map $g \mapsto \Norm{\mq(g)}$ is proper with respect to $\S$, i.e., for all $C>0$ there exist finite subsets $F,H \subset G$ and $\K \subset \S$ such that $$B_C^\mq := \{ g\in G : \Norm{\mq(g)} \leqslant C \} \subseteq F \K H; $$

\end{defn}

 Wreath-like product groups in the class $\mathcal {WR}_b(A, B\ca I )$ admit a natural class of arrays called \emph{support arrays}. An elementary way of constructing them is presented in the following result and its proof.

\begin{thm}\label{supportarray} Let $G\in \mathcal {WR}_b(A, B\ca I)$ be a wreath-like product group and let $\alpha\colon B\times B \ra A^{(I)}$ the corresponding $2$-cocycle. Let $C>0$ so that $|\supp(\alpha(x,y))|\leqslant C$ for all $x,y\in B$. Then one can find a map  $\mq\colon G \ra \ell^2(I)$ satisfying the following: \begin{enumerate}\item For every $x,y\in G$ we have 
\begin{equation} \label{1}
\sup_{z\in G}\|\mq(xzy)-\pi_{\epsilon(x)}(\mq(z))\|_2<\infty.
 \end{equation}
 \item For every $x\in A^{(I)}$ we have  $\|\mq(x)\|^2=|\supp(x)|$. 
 \end{enumerate}
 
\noindent Here we denoted by $\pi\colon B\ra \mathcal O(\mathcal \ell^2(I))$ the orthogonal representation induced by $B\ca I$. 
\end{thm}
\begin{proof} Throughout the proof, for every $c\in A^{(I)}$ we denote by $\supp(c)\Subset I$ its support.  Now we define  $\mq\colon G\ra \ell^2(I)$ as follows. For $G\ni g=cb$ with  $c\in A^{(I)}, b\in B$ we let \begin{equation}\label{2}\mq(cb )(i)=\begin{cases}
1 &\mbox {if } i\in \supp(c)\\
0&\mbox {otherwise.} \end{cases}\end{equation}
To get \eqref{1} it suffices to show the following. For every $a, c\in A^{(I)}$ and $s,t\in B$ we have that

\begin{equation}\label{3}
\begin{split} &\| \mq(ca s )- \mq(a s )\|_2, \leqslant \sqrt{2}\|\mq(c)\|_2,\\ &   \| \mq(ca s )- \mq(c s )\|_2 \leqslant \sqrt{2}\|\mq(a)\|_2,\\
&\| \mq(t a s )- \pi_t(\mq(a s))\|_2\leqslant  2\sqrt {C},\\
&\| \mq(c s a)- \mq(c s)\|_2\leqslant \sqrt{2}\|\mq(a)\|_2,\\
&\| \mq(c s t )- \mq(c s )\|_2\leqslant \sqrt{2C}.\\
\end{split}
\end{equation}

 To justify the first inequality observe that (\ref{2}) gives
$\| \mq(ca s )- \mq(a s)\|^2_2=\sum_{i\in I} |\mq(ca s)(i)- \mq(a s)(i)|^2 = |\supp(ca) \triangle \supp(a)| \leqslant 2 | \supp(c)|=2\|\mq(c)\|^2$. The second one follows similarly.
\vskip 0.02in 

\noindent To see the third inequality, using the second inequality of (\ref{3})  we have  \begin{equation*}\label{4}\begin{split}
&\| \mq(ta s )- \pi_t(\mq(a s))\|^2_2\leqslant   2\| \mq(\sigma_t(a) \alpha(t,s) ts )- \mq(\sigma_t(a) ts )\|^2_2 +2\|\mq(\sigma_t(a) ts )- \pi_t(\mq(a s)\|^2_2\\
&\leqslant 4\|\mq(\alpha(t,s))\|^2_2 + 2\|\mq(\sigma_t(a) ts )- \pi_t(\mq(a s))\|^2_2=4 C +2\|\mq(\sigma_t(a) ts )- \pi_t(\mq(a s))\|^2_2\\
&= 4C +\sum_{i\in I} |\mq(\sigma_t(a) ts)(i)- \mq(a s)( t^{-1}i)|^2= 4C+|\supp(\sigma_t(a)) \triangle t \cdot \supp(a)| =4C.
\end{split}\end{equation*}
 To derive the fourth inequality in (\ref{3}), using the second inequality in (\ref{3}) we have $\| \mq(csa)- \mq(c s)\|_2=\| \mq(c\sigma_s (a) s)- \mq(c s)\|_2\leqslant \sqrt 2 \| \mq(\sigma_s (a))\|_2=\sqrt 2 \| \mq(a)\|_2$.

Finally, we show the last inequality in (\ref{3}). Using again the second inequality in (\ref{3}) we have $\| \mq(c s t)- \mq(c s)\|^2_2=  \| \mq( c\alpha(t,s)st )- \mq(c s)\|^2_2 = \sum_{i\in I} | \mq( c\alpha(t,s)) (i)-\mq(c) (i)|^2 = |\supp(c \alpha(t,s)) \triangle \supp (c)|
=2  |\supp (\alpha(t,s)) | \leqslant 2C.$  \end{proof}

\begin{prop}
Let $G \in \WR_b(A,B\acts I)$ where $B$ is a non-amenable subgroup of a hyperbolic group, $A$ is amenable and the action $B\acts I$ has amenable stabilizers.  Then there exists a weakly-$\ell^2$ representation $\rho\colon G \ra \mathcal{O}(\sH)$ and an array $\ma\colon G \ra \sH$ such that for any $C>0$ there is a finite subset $F\subset B$ and $n\in \mathbb{N}$ such that
\begin{equation}\label{ball_sum_arrays}
    B_C^\ma \subset \{ ah \in G : a\in A^{(I)}, |\supp(a)| \leqslant n, h \in F \}\ .
\end{equation}
\end{prop}

\begin{proof}
Let $\ma_1\colon G \ra \ell^2(I)$ be the array constructed in the previous theorem.  By decomposing $I=\sqcup_k B/\stab(b_k)$ as a disjoint union of orbits, where $k$ runs through the orbits and we pick a representative $b_k\in I$ for each orbit $k$, we obtain a direct sum decomposition of the orthogonal representation of $G$ on $\sH_1:= \ell^2(I)$ as $\sH_1 \cong \oplus_k \ell^2 (G / {\epsilon^{-1}(\stab(b_k)))}$.  Notice that both $A$ and each $\stab(b_k)$ are amenable, and since $\epsilon^{-1}(\stab(b_k))$ is an extension of $A^{(I)}$ by $\stab(b_k)$, then it is also amenable. Thus, each representation $\ell^2 (G / {\epsilon^{-1}(\stab(b_k)))}$ of $G$ is weakly-$\ell^2$, and so is $\rho_1\colon G\ra \mathcal{O}(\sH_1)$.

 On the other hand, $B$ being a subgroup of a hyperbolic group implies by \cite{MMS03} that there is a proper array $\mathfrak{c}\colon B\ra \sH_2$ into a weakly-$\ell^2$ representation $\pi\colon B \ra \mathcal{O}(\sH_2)$.  We can pull-back this array and corresponding representation to $G$ (by composing with the canonical quotient map $\varepsilon$) to obtain the array $\ma_2 = \mathfrak{c}\circ \varepsilon$ into a representation $\rho_2=\pi \circ \varepsilon$ which is weakly contained in the pull-back to $G$ of the regular representation of $B$.  Moreover, since $\ker(\varepsilon)$ is amenable, then the pull-back to $G$ of the regular representation of $B$ is also weakly-$\ell^2$.  Since $\mathfrak{c}$ is proper, we have that $\ma_2$ is proper relative to ${\ker(\varepsilon)} = A^{(I)}$.

 Taking direct sum we obtain the array $\ma=\ma_1 \oplus \ma_2 \colon G \ra \sH_1 \oplus \sH_2$ into the weakly-$\ell^2$ representation $\rho= \rho_1 \oplus \rho_2$.  Given $C>0$, since $\Norm{\ma(g)}^2 = \Norm{\ma_1(g)}^2 + \Norm{\ma_2(g)}^2 $, we get that $B_C^\ma \subset B_C^{\ma_1} \cap B_C^{\ma_2}$.  Thus, the control on the finite radius balls $B_C^\ma$ in (\ref{ball_sum_arrays}) is obtained from the definition of $\ma_1$ and the fact that $\ma_2$ is proper relative to $A^{(I)}$.
\end{proof}

\begin{rem}
Without loss of generality we can assume the array $\ma$ obtained in the previous proposition satisfies $\Norm{\ma(g)}\geqslant 1$  for all $g\in G$ and $\Norm{\ma(1)}=1$.  Indeed, following the notation of the proof, we need only redefine $\ma_1(1)$ and $\mc(b)$ on the finitely many $b\in B$ for which $\Norm{\mc(b)}<1$, while noting this finitely many modifications do not affect the array property or the control on the finite radius balls.
\end{rem}

 We will turn our attention now to constructing arrays on a countable (possibly infinite) direct sum of groups in the class $\WR_b(A,B\acts I)$ into a tensor product representation.  Suppose $\{G_j : j\in J\}$ is a countable collection of groups and let $G=\oplus_{j\in J} G_j$.  For any subset $F \subset J$ let $p_F \colon G \ra \oplus_{j\in F} G_j$ be the canonical projection.  For simplicity of notation, when $i\in J$ we will denote by $p_i := p_{\{i\}}$ the projection onto $G_i$.  For $g\in G$ we will use the notation $g_i=p_i(g)$ and $\mathfrak{s}(g) := \supp(g) = \{ i\in J : g_i\neq 1 \}$.  Given orthogonal representations $\rho_j \colon G_j \ra \mathcal{O}(\sH_j)$ for each $j\in J$ we define the representation $\rho$ of $G$ on $\sH = \otimes_{j\in J} \sH_j$ by letting $\rho(g)= \otimes_{j\in \mathfrak{s}(g)} \rho_j(g_j)$.  When $J$ is infinite we can fix a dummy unit vector $\xi_j\in \sH_j$ for each $j\in J$ to form the infinite tensor product of Hilbert spaces $\sH = \otimes_{j\in J} (\sH_j,\xi_j)$, the choice of $\xi_j$ will not affect the result and the operators $\rho(g)$ are well defined since these can be seen as tensor products of countably many operators where only finitely many are not the identity.
 
 We continue with a result that is essential towards construction of arrays on infinite direct sum groups.

\begin{lem}
Let $J$ be an at most countable set.  For $j\in J$, let $G_j \in \WR_b(A_j,B_j\acts I_j)$ where $B_j$ is a non-amenable subgroup of a hyperbolic group, $A_j$ is amenable and the action $B_j\acts I_j$ has amenable stabilizers.  Let $\ma_j\colon G_j \ra \sH_j$ be the array into a weakly-$\ell^2$ orthogonal representation satisfying (\ref{ball_sum_arrays}) and assume both $\Norm{\ma_j(1)}=1$ and $\Norm{\ma_j(g)}\geq 1$  for all $g\in G_j$, $j\in J$.  Define $G=\oplus_{j\in J} G_j$ and the representation $\rho \colon G \ra \mathcal{\sH}$ as above, where $\sH = \otimes_{j\in J} (\sH_j,\ma_j(1))$.  Consider the collection of subsets $\mathscr{G} = \{ B_n^{\ma_j}\times G_{\hat{j}} : n\in\mathbb N, j\in J \}$ and the map $\mc\colon G \ra \sH$ is given by
$$\mc(g) = \frac{\otimes_{s\in \mathfrak{s}(g)}\ma_s(g_s)}{\max_{s\in \mathfrak{s}(g)} \prod_{t\neq s} \Norm{\ma_t(g_t)}}\ .$$
for any $g\in G$. For any $g\in G$ denote by $\zeta(g)= \frac{\mc(g)}{\|\mc(g)\|}$. Then for any $g,h\in G$ we have that
\begin{equation}\label{almostbimodularatinfty} \lim_{k\to \infty/\mathscr{G}} \Norm{ \zeta(gkh) - \rho_g \zeta(k)} = 0. \end{equation}
\end{lem}

\begin{rem}
In the previous statement, for a function $f\colon G \ra V$, with $V$ a normed vector space, we say that $\lim_{k \ra \infty /\mathscr G } \|f(k)\|=0$ if and only if for every $\varepsilon>0$ the set $\{k\in G \colon \|f(k)\|_V\geqslant \varepsilon \}$ is small relative to $\mathscr G$, i.e. it is contained in the union of finitely many sets of the form $t\Sigma v$, where $t,v\in G$,  $\Sigma \in \mathscr G$.  
\end{rem}

\begin{proof} For every $i\in J$ denote by $g^i=(g^i_s)_s\in G$ where $g^i_s=1$ for all $s\neq i$ and $g^i_i\neq 1$. First we argue that it suffices to prove \eqref{almostbimodularatinfty} only for  $(g,h)=(g^i,1)$ and for $(g,h)=(1,h^j)$. So assume we have established \eqref{almostbimodularatinfty} for such elements. As $g= \Pi_{i\in \ms(g)} g^i$, $h=\Pi_{j\in {\mathfrak s}(h)} h^j$, inserting terms and using the triangle inequality we get 
\begin{equation}\label{dom1}
    \|\zeta(gkh)-\rho_g\zeta(k)\| \leq \sum_{i\in \ms(g)
    }\|\zeta (g^i (\Pi_{s>i} g^s) k h)- \rho_{g^i}\zeta((\Pi_{s>i} g^s) kh)\|+ \sum_{j\in \ms(h)}\|\zeta(kh^j (\Pi_{s>j} h^s))- \zeta(k (\Pi_{s>j}h^s))\|
\end{equation}

 Since we assumed \eqref{almostbimodularatinfty} holds for all  $(g,h)=(g^i,1)$ and $(g,h)=(1,h^j)$
and since we have $k\ra \infty/\mathscr G$ if and only if $tku\ra \infty/\mathscr G$ for every fixed $t,u \in G$ we conclude  
 $\lim_{k\ra \infty/\mathscr G}\|\zeta (g^i (\Pi_{s>i} g^s) k h)- \rho_{g^i}\zeta((\Pi_{s>i} g^s) kh)\|=0$  and
$\lim_{k\ra \infty/\mathscr G}\|\zeta(kh^j (\Pi_{s>j} h^s))- \zeta(k (\Pi_{s>j}h^s))\|=0$. These combined with  \eqref{dom1} yield \eqref{almostbimodularatinfty}.

In the remaining part we prove \eqref{almostbimodularatinfty} for $(g,h)=(g^i,1)$ 
and $(g,h)=(1, h^j)$. Below we will include only a proof for the first case as the other is similar. 

Throughout the proof we fix  $k=(k_s)_s\in G $ with $k_i\neq 1,(g^{i}_i)^{-1}$ and notice $\mathfrak s(gk)=\mathfrak s(k)=:{\mathfrak s}$. Thus using the above formulae we see that \begin{equation}\label{beginningestimate}\begin{split}
    &\|\zeta(g^ik)-\rho_{g^i}\zeta(k)\|\leqslant \frac{2\|\mc(g^ik)-\rho_{g^i}\mc(k)
    \|}{\|\mc(k)\|}= \frac{2}{\|\mc(k)\|} \left \|\frac{\otimes_{s\in \mathfrak s} \ma_s(g_s^ik_s)}{\max_s(\Pi_{t\neq s}\|\ma_t(g^i_t k_t)\|)}- \frac{\otimes_{s\in\mathfrak s} \rho_{g^i_s}\ma_s(k_s)}{\max_s(\Pi_{t\neq s}\|\ma_t(k_t)\|)}\right \|
    \\&\leqslant 
    \frac{2\|\otimes_{s} \ma_s(g_s^ik_s)-\otimes_s \rho_{g^i_s}\ma_s(k_s)\|}{ \|\mc(k)\|\max_s(\Pi_{t\neq s}\|\ma_t(g^i_t k_t)\|)} + \frac{2 \Pi_{s} \|\ma_s(g^i_sk_s)\|}{\|\mc(k)\|} \left |\frac{1}{\max_s(\Pi_{t\neq s}\|\ma_t(g^i_t k_t)\|)}- \frac{1}{\max_s(\Pi_{t\neq s}\|\ma_t(k_t)\|)}\right|\\
    &\leqslant \frac{2d_i(g^i_i,1)\Pi_{s\neq i}\|\ma_s(k_s)\|  }{\|\mc(k)\|\max_s(\Pi_{t\neq s}\|\ma_t(g^i_t k_t)\|)} + \frac{2 \Pi_{s} \|\ma_s(g^i_sk_s)\|}{\|\mc(k)\|} \left |\frac{1}{\max_s(\Pi_{t\neq s}\|\ma_t(g^i_t k_t)\|)}- \frac{1}{\max_s(\Pi_{t\neq s}\|\ma_t(k_t)\|)}\right|\end{split}
\end{equation}

First we prove some bounds for the first term in the last sum in \eqref{beginningestimate}. 

 Fix $C>0$ and let $k=(k_s)_s\in G\setminus B^\ma_C$. Thus $\|\mc(k)\|=\min_s \|\ma_s(k_s)\|> C$ and using the definitions we have 
    \begin{equation}\label{firstbound}\frac{2d_i(g^i,1) \Pi_{s\neq i}\|\ma_s(k_s)\|  }{\|\mc(k)\|\max_s(\Pi_{t\neq s}\|\ma_t(g^i_t k_t)\|)}\leqslant \frac{2 d_i(g^i_i,1)}{\|\mc(k)\|}< \frac{2 d_i(g_i^i,1)}{C}. \end{equation} 

 Now fix $k=(k_s)_s\in B^\ma_C\setminus ((G_i \setminus B^{\ma_i}_{C+1+d_i(g_i^i,1)})\times G_{\hat i})$. Next we argue the maximum $\max_s(\Pi_{t\neq s}\|\ma_t(g^i_t k_t)\|)$ does not occur at $s= i$. If it did then we would have \begin{equation}\label{max2}\|\ma_s(k_s)\|\geqslant \|\ma_i(g_i^ik_i)\|\text{ for all }s\neq i.\end{equation} As $k_i \in G_i \setminus B^{\ma_i}_{C+1+d_i(g_i^i,1)}$ we see that $\|\ma_i(g_i^ik_i)\|\geqslant \|\rho_{g_i^i}\ma_i(k_i)\|-\|\ma_i(k_i)-\rho_{g_i^i}\ma_i(k_i)\|\geqslant \|\ma_i(k_i)\|-d_i(g_i^i,1)\geqslant C+1 $. Using \eqref{max2} this entails $\|\mc(k)\|=\min_s\|\ma_s(k_s)\|\geqslant C+1$, a contradiction. Thus $\max_s(\Pi_{t\neq s}\|\ma_t(g^i_t k_t)\|)$ occurs at $s_0\neq i$ and hence we have $\|\ma_s(g^i_sk_s)\|\geqslant \|\ma_{s_0}(k_{s_0})\|$ for all $s\neq s_0$. Moreover, as $k_i\in G_i\setminus B^{\ma_i}_{d_i(g_i^i,1)+C+1}$ and $\|\mc(k)\|=\min_s\|\ma_s(k_s)\|\leqslant C$ we have that $\|\mc(k)\|=\|\ma_{s_1}(k_{s_1})\|$
for some $s_1\neq i$.   In particular, $s_0=s_1$ and $\Norm{\mc(k)} = \Norm{\ma_{s_0}(k_{s_0})}$. Altogether, these show that 
\begin{equation}\label{secondbound}
\frac{2 \Pi_{s\neq i}\|\ma_s(k_s)\| d_i(g^i,1) }{\|\mc(k)\|\max_s(\Pi_{t\neq s}\|\ma_t(g^i_t k_t)\|)}= \frac{2 \|\ma_{s_0}(k_{s_0})\| d_i(g^i_i,1)}{\|\mc(k)\| \|\ma_i(g_i^ik_i)\|}= \frac{2 d(g_i^i,1)}{\|\ma_i(g_i^ik_i)\|} \leqslant \frac{2d_i(g_i^i,1)}{\|\ma_i(k_1)\|-d_i(g_i^i,1)}<\frac{2d_i(g_i^i,1)}{C+1}.\end{equation} In conclusion, \eqref{firstbound} and \eqref{secondbound} show that for all $k=(k_s)_s\in (G_i\setminus (B^{\ma_i}_{C+1+d_i(g_i^i,1)}\cup\{1, (g_i^i)^{-1}\})) \times G_{\hat i}$ we have 
\begin{equation}\label{bound firstquantity}
    \frac{2d_i(g^i,1) \Pi_{s\neq i}\|\ma_s(k_s)\|  }{\|\mc(k)\|\max_s(\Pi_{t\neq s}\|\ma_t(g^i_t k_t)\|)} \leqslant\frac{2d_i(g_i^i,1)}{C}.\end{equation}

 Next we proceed to show bounds similar to \eqref{bound firstquantity} for the second term in the last sum of \eqref{beginningestimate}. First note that using the bounded equivariance of the array  $\ma_i$, a basic computation shows that 
\begin{equation}\label{secondquantitybound1}
    \left |\max_s \left (\Pi_{t\neq s}\|\ma_t(k_t)\|\right )- \max_{s}\left (\Pi_{t\neq s}\|\ma_t(g^i_tk_t)\|\right)\right |\leqslant  \left (d_i(g_i^i,1)+d_i((g_i^i)^{-1},1)\right )\max_{s\neq i}\left (\Pi_{t\neq s,i}\|\ma_t(k_t)\|\right ).
\end{equation}
Letting $D_i := 2\left (d_i(g_i^i,1)+d_i((g_i^i)^{-1},1)\right )$ and using the above estimate we have \begin{equation}\label{secondquantityfisrtcase1}\begin{split}
&\frac{2 \Pi_{s} \|\ma_s(g^i_sk_s)\|}{\|\mc(k)\|} \left |\frac{1}{\max_s(\Pi_{t\neq s}\|\ma_t(g^i_t k_t)\|)}- \frac{1}{\max_s(\Pi_{t\neq s}\|\ma_t(k_t)\|)}\right|\leqslant \\
&\leqslant
\frac{D_i\left (\Pi_{s} \|\ma_s(g^i_sk_s)\|\right)\max_{s\neq i}\left (\Pi_{t\neq s,i}\|\ma_t(k_t)\|\right )}{\|\mc(k)\|\max_s(\Pi_{t\neq s}\|\ma_t(g^i_t k_t)\|) \max_s(\Pi_{t\neq s}\|\ma_t(k_t)\|)}\end{split}
\end{equation}

 Fix $k=(k_s)_s\in G\setminus B^\ma_C$. Hence $\|\mc(k)\|=\min_{s}\|\ma_s(k_s)\|\geqslant C$. Using  $\Pi_s\|\ma_s(g^i_s k_s)\|= \Pi_{t\neq i}\|\ma_t(k_t)\| \|\ma_i(g^i_i k_i)\|$, regrouping in \eqref{secondquantitybound1}, and using the previous inequality we see the last expression in \eqref{secondquantitybound1} further satisfies

\begin{equation}\label{secondquantityfirstcase2}
    = \frac{D_i}{\|\mc(k)\|} \cdot \frac{\max_{s\neq i}\left (\Pi_{t\neq s,i}\|\ma_t(k_t)\|\|\ma_i(g^i_ik_i)\|\right)}{\max_s(\Pi_{t\neq s}\|\ma_t(g^i_t k_t)\|)} \cdot \frac{\left (\Pi_{t\neq i} \|\ma_t(k_t)\|\right)}{ \max_s(\Pi_{t\neq s}\|\ma_t(k_t)\|)}\leqslant \frac{D_i}{C}.\end{equation}

 Now fix $k=(k_s)_s\in B_C^\ma\cap ((G_i\setminus B^{\ma_i}_{C+1+d_i(g_i^i,1)})\times G_{\hat{i}})$. Assume that the maximum $\max_{s\neq i}\left (\Pi_{t\neq s,i}\|\ma_t(k_t)\|\right )$ occurs at $s=s_0\neq i$. Hence $\|\ma_t(k_t)\|\geqslant \|\ma_{s_0}(k_{s_0})\|$  for all $t\neq i$ which in particular implies both maxima $\max_s(\Pi_{t\neq s}\|\ma_t(k_t)\|)$ and  $\max_s(\Pi_{t\neq s}\|\ma_t(g^i_t k_t)\|$ occur at $s_0$ or $i$. Next we argue that neither of them can occur at $i$. To see this assume by contradiction the first maximum occurs at $i$. Thus  $\|\ma_t(k_t)\|\geqslant \|\ma_i(t_i)\|$ for all $t$. Hence $C\geqslant \|\mc(k)\|=\min \|\mc_t(k_t)\|\geqslant \|\ma_i(k_i)\|\geqslant C+1+ d_i(g_i^i,1)$, a contradiction. Now assume the second maximum occurs at $i$. Hence, $\|\ma_t(k_t)\|\geqslant \|\ma_i(g_i^ik_i )\|$, for all $t\neq i$ and thus for every $t\neq i$ we have $\|\ma_t(k_t)\|\geqslant \|\ma_i(g_i^ik_i)\|\geqslant \|\ma_i(k_i)\|-\| \ma_i(g_i^i k_i)-\lambda_{g_i^i}(\ma_i(k_i))\|\geqslant C+1$, as $k_i \in G_i \setminus B^{\ma_i}_{C+1+d_i(g_i^i,1)}$. Hence $C\geqslant \|\mc(k)\|= \min_{t}\|\ma_t(k_t)\|\geqslant C+1$, a contradiction.

 In conclusion, both maxima occur at $s_0$ and hence $\|\ma_t(k_t)\|,\|\ma_t (g^i_t k_t)\| \geqslant \|\ma_{s_0}(k_{s_0})\|$ for all $t$. Thus we have

\begin{equation}\label{boundsecondquantity}\begin{split}
    &\frac{D_i\left (\Pi_{s} \|\ma_s(g^i_sk_s)\|\right)\max_{s\neq i}\left (\Pi_{t\neq s,i}\|\ma_t(k_t)\|\right )}{\|\mc(k)\|\max_s(\Pi_{t\neq s}\|\ma_t(g^i_t k_t)\|) \max_s(\Pi_{t\neq s}\|\ma_t(k_t)\|)}=  \frac{D_i\left (\Pi_{s} \|\ma_s(g^i_sk_s)\|\right)\left (\Pi_{t\neq s_0,i}\|\ma_t(k_t)\|\right )}{\|\mc(k)\|(\Pi_{t\neq s_0}\|\ma_t(g^i_t k_t)\|) (\Pi_{t\neq s_0}\|\ma_t(k_t)\|)}\\
    &=\frac{D_i\left (\Pi_{s} \|\ma_s(g^i_sk_s)\|\right)}{\|\mc(k)\|(\Pi_{t\neq s_0}\|\ma_t(g^i_t k_t)\|) \|\ma_i(k_i)\|} = \frac{D_i \|\ma_{s_0}(k_{s_0})\|}{\|\mc(k)\| \|\ma_i(k_i)\|} \leqslant 
\frac{D_i}{ C+1+ d_i(g_i^i,1)}.\end{split}\end{equation}

 Altogether, estimates \eqref{bound firstquantity}, \eqref{secondquantityfisrtcase1}, \eqref{secondquantityfirstcase2}, \eqref{boundsecondquantity} and \eqref{beginningestimate} show that for all $C>0$ and  $k=(k_s)_s\in \left (G_i\setminus (B^{\ma_i}_{C+1+d_i(g_i^i,1)} \cup\{ (g_i^i)^{-1},1\})\right )\times G_{\hat i}$ we have that  \begin{equation}
    \|\zeta(g^ik)-\rho_{g^i} \zeta (k)\|\leqslant \frac{4(d_i(g_i^i,1)+d_{i}((g_i^i)^{-1},1))}{C}.
\end{equation}
Letting $C\nearrow \infty$ this clearly gives \eqref{almostbimodularatinfty} for all $(g,h)=(g^i,1)$.\end{proof}
 Finally, we use the prior lemma, together with a general method from \cite{BO08} and \cite{PV12}, to construct an array on an infinite direct sum of wreath-like product groups.
\begin{prop}\label{relsolarray}
Let $J$ be an at most countable set and for every $j\in J$, let $G_j \in \WR_b(A_j,B_j\ca I_j)$ where $A_j$ is amenable, $B_j$ is a non-amenable subgroup of a hyperbolic group and the action $B_j\ca I_j$ has amenable stabilizers.  Let $G=\oplus_{j\in J} G_j$. Then there exists a weakly-$\ell^2$ representation $\rho \colon G \ra \mathcal{O}(\sH)$ and an array $\ma\colon G \ra \sH$ satisfying that for every $C\geqslant 0$ there is a finite subset $J_0 \subseteq J$ such that for each $j \in J_0$ there is a finite subset $F_j\subset G_j$ and $n_j \in \mathbb{N}$ satisfying 
\begin{equation}
B_C^\ma \subset \bigcup_{j\in J_0} \{ a_j h_j \colon a_j \in A_j^{(I_j)}, |\supp(a_j)|\leqslant n_j, h_j \in F_j \} \times G_{\hat{j}}\,.
\end{equation}
\end{prop}

\begin{proof}
The construction of this array will follow a method that can be traced back to \cite{BO08} and more concretely used in \cite[Proposition 2.7]{PV12}.  Let $\mc \colon G \ra \sH$ and $\zeta(g) = \frac{\mc(g)}{\Norm{\mc(g)}}$ be the maps from the previous lemma.  Take any nested sequence of finite subsets $\{1\} = E_0 \subset ... \subset E_n \subset ... \subset G$ such that $E_n^{-1} = E_n$ and $\bigcup_{n\in \mathbb{N}} E_n = G$.  Now define inductively the sets $\Omega_n \subset G$ by setting $\Omega_0 =\{1\}$ and for $n\geq 1$
$$ \Omega_n := E_n \Omega_{n-1} E_n \cup \bigcup_{g,h \in E_n} \{ k \in G : \Norm{\zeta(gkh) - \lambda_g \zeta(k)} > 1/n \}\,. $$

 By construction, $(\Omega_n)_n$ is an increasing sequence satisfying $\bigcup_n \Omega_n =G$ and $E_n \subset \Omega_n$ for every $n\in \mathbb{N}$.  Now we define $\ma\colon G\ra \sH$ by

\begin{equation*}
\ma(k) = \begin{cases} 0 &\mbox {if } k\in \Omega_1\\ n\zeta(k) &\mbox {if } k\in \Omega_{n+1}\setminus \Omega_n \text{ for some } n\geqslant 1\,.
\end{cases}
\end{equation*}

 Notice that $\Norm{\ma(k)} \geq n$ whenever $k \in G\setminus \Omega_n$.  Hence, from the previous lemma, we obtain that the finite radius balls $B_C^\ma$ are small relative to $\mathscr{G}$, i.e. they have the desired structure.

 We are left to prove that $\ma$ is indeed an array.  Let $g,h \in G$, so $g,h \in E_m$ for some $m\geqslant 1$.  We shall prove $\Norm{\ma(gkh) - \lambda_g \ma(k)} \leqslant 2m$ for all $k\in G$.  Indeed, if $k\in G\setminus \Omega_m$, then $k\in \Omega_{n+1}\setminus\Omega_n$ for some $n\geqslant m$ and $\ma(k) = n\zeta(k)$.  Since $g,h \in E_m \subset E_n = E_n^{-1}$ and $k\notin \Omega_n$, we must have $gkh \in \Omega_{n+2}\setminus \Omega_{n-1}$.  Thus, $\ma(gkh)$ is one of $(n+1)\zeta(gkh)$, $n\zeta(gkh)$ or $(n-1)\zeta(gkh)$.  In any case, we have $\Norm{\ma(gkh) - n \zeta(gkh)} \leqslant 1$.  Moreover, since $k \notin \Omega_n$ and $g,h \in E_n$ we must have $\Norm{\zeta(gkh) - \lambda_g \zeta(k)}\leqslant 1/n$.  Therefore 
$$ \Norm{\ma(gkh) - \lambda_g \ma(k)} \leqslant \Norm{\ma(gkh) - n\zeta(gkh)} + n\Norm{\zeta(gkh) - \lambda_g \zeta(k)} \leqslant 2 \leqslant 2m. $$
On the other hand, if $k\in \Omega_m$ then $\Norm{\ma(k)} \leqslant m-1$ and since $g,h \in E_m$ we also have $gkh \in \Omega_{m+1}$ and $\Norm{\ma(gkh)} \leqslant m$.  Thus,
$$ \Norm{\ma(gkh) - \lambda_g \ma(k)} \leqslant \Norm{\ma(gkh)} + \Norm{\ma(k)} \leqslant 2m - 1 < 2m\,. $$\end{proof}

\section{Some analytic properties of deformations on von Neumann algebras of wreath-like product groups}

In the first part of this section we recall the construction of von Neumann algebras deformations associated with arrays on groups and we collect together some of their basic properties. Then we will  establish a subordination property (Theorem \ref{subrel}) between a restriction of the array deformation \cite{Si10,CS11} and Ioana's deformation on infinite tensor algebras \cite{Io06}. This is instrumental in the infinitesimal analysis performed in Section~\ref{sec. vNa main} towards reconstructing the infinite folded direct sum feature of a group from $W^*$-equivalence.  Incidentally, this subordination property has been also used in an essential way in the reconstruction of a group center under $W^*$-equivalence, \cite{CFQT23}.

\subsection{Array deformations of von Neumann algebras of wreath-like product groups}\label{sec:deform_array}

Suppose $\pi\colon G \rightarrow \mathcal{O}(\sH)$ is an orthogonal representation of a countable discrete group on a separable real Hilbert space.  Then there is a p.m.p. action $G\acts^{\sigma^\pi} (Y^\pi,\nu^\pi)$ (known as the Gaussian action) on a non-atomic standard probability space such that the Koopman representation $G\acts^{\pi_0} L^2_0(Y^\pi, \nu^\pi) = L^2(Y^\pi, \nu^\pi) \ominus \mathbb{C}1$ is unitarily equivalent to the direct sum over the (strictly positive) symmetric tensor powers of the complexified representation $\pi_\mathbb{C}\colon  G \rightarrow \mathscr{U}(\mathscr{H}\otimes\mathbb{C})$.  In \cite{Si10} this Gaussian construction was used to construct a deformation of $\mathcal{L}(G)$ via the exponentiation of a 1-cocycle $\mq \colon G\rightarrow \mathscr{H}$ into an orthogonal representation $\pi$, this was then generalized in \cite{CS11} by allowing the map $\mq$ to be an array.  We briefly illustrate this procedure.

 Let $S\mathscr{H}_\mathbb{C} = \oplus_{n=0}^\infty \mathscr{H}_\mathbb{C}^{\odot n}$ be the direct sum over all symmetric tensor powers (where we make the identifications $\mathscr{H}_\mathbb{C}^{\odot 0}= \mathbb{C}\Omega$, $\Norm{\Omega}=1$ and $\xi^{\otimes 0} = \Omega$ for any $\xi \in \mathscr{H}$).  The exponential of a vector $\xi\in \mathscr{H}$ can be defined by $\mathrm{Exp}(\xi) = \sum_{n=0}^\infty (n!)^{-1/2}\xi^{\otimes n}$, so $\innpr{\mathrm{Exp}(\xi),\mathrm{Exp}(\eta)} = \exp\innpr{\xi,\eta} $ and the set $\{\mathrm{Exp}(\xi) \colon \xi\in\mathscr{H} \}$ is linearly independent and total in $S\mathscr{H}_\mathbb{C}$ \cite[Proposition 2.2]{Gui72}.  Thus, for each $\xi\in \mathscr{H}$ the map $\omega(\xi)$ defined by $\mathrm{Exp}(\eta) \mapsto e^{-\Norm{\xi}^2 - \innpr{\sqrt{2}\xi,\eta}} \mathrm{Exp}(\eta+\sqrt{2}\xi)$ (for any $\eta \in \mathscr{H}$) extends to a unitary on $S\mathscr{H}_\mathbb{C}$.  We let $\mathcal{D}$ be the von Neumann algebra $\mathcal{D}$ generated by $\{ \omega(\xi) : \xi \in \mathscr{H} \}$ and notice:
\begin{itemize}
	\item $\omega(0) = 1$, $\omega(\xi + \eta) = \omega(\xi)\omega(\eta)$ and $\omega(\xi)^* = \omega(-\xi)$ for all $\xi, \eta \in\mathscr{H}$;
	\item the vector $\Omega$ is cyclic and separating for $\mathcal{D}$ and defines a faithful normal trace $\tau(\omega(\xi)) := \innpr{\omega(\xi) \Omega, \Omega} = e^{-\Norm{\xi}^2}$.
\end{itemize}
Since $\mathcal{D}$ is a separable abelian von Neumann algebra, there is a standard probability space $(Y^\pi,\nu^\pi)$ such that $\mathcal{D}\cong L^\infty (Y^\pi,\nu^\pi)$ and $\tau$ is implemented by the integral with respect to $\nu^\pi$.  Moreover, $\mathcal{D}$ is diffuse and so $(Y^\pi, \nu^\pi)$ is non-atomic.  Notice that any orthogonal operator $A\in \mathcal{O}(\mathscr{H})$ extends to a unitary $A_\mathbb{C} \in \mathscr{U}(\mathscr{H}_\mathbb{C})$ and to a unitary $U_A \in \mathscr{U}(S\mathscr{H}_\mathbb{C})$ (via the direct sum of the tensor powers).  Hence, the above definitions imply that conjugation by $U_A$ gives a trace-preserving automorphism $\sigma_A$ of $\mathcal{D}$ where $\sigma_A (\omega(\xi)) = U_A \omega(\xi) U_A^* = \omega(A\xi)$.  Consequently, the orthogonal representation $\pi\colon G \rightarrow \mathcal{O}(\mathscr{H})$ induces an action by trace-preserving automorphisms $G \acts^{\sigma^\pi} \mathcal{D}$ (where $\sigma^\pi_g = \sigma_{\pi(g)}$) and thus a p.m.p action $G \acts^{\sigma^\pi} (Y^\pi,\nu^\pi)$.

 We now proceed to describe the Gaussian deformation and recall some of its properties.  Let $G \acts^\rho (\mathcal{N},\tau)$ be a trace-preserving action on a finite von Neumann algebra, and consider the product action $G\acts^{\sigma^\pi \otimes \rho} \mathcal{D}\bar{\otimes} \mathcal{N}$. Given a map $\mq\colon G \rightarrow \mathscr{H}$ and $t\in \mathbb{R}$ we may define a unitary $V^\mq_t \in \mathscr{U}(L^2(Y^\pi, \nu^\pi) \otimes L^2(\mathcal{N}) \otimes \ell^2 (G))$ by extending $V^\mq_t(\zeta \otimes x \otimes \delta_h) = \omega(t\mq(h)) \zeta \otimes x \otimes \delta_h$ for every $\zeta \in L^2(Y^\pi, \nu^\pi)$, $x\in L^2(\mathcal{N})$ and $h\in G$.  When the map $\mq$ is a 1-cocycle into $\pi$, then conjugation by $V^\mq_t$ is an automorphism of $(\mathcal{D}\Bar{\otimes} \mathcal{N})\rtimes_{\sigma^\pi\otimes \rho} G$ and we have a deformation at the level of the von Neumann algebras $\mathcal{N}\rtimes_\rho G \subset (\mathcal{D}\Bar{\otimes} \mathcal{N})\rtimes_{\sigma^\pi\otimes \rho} G$.  This is not true when $\mq$ is only an array, but in that case we have a deformation at the level of the $C^*$-algebras.

\begin{prop}\label{propo:Gaussian_properties}
Let $G \acts^\rho ( \mathcal{N},\tau)$ be a trace-preserving action on a finite von Neumann algebra, $\pi\colon G \rightarrow \mathcal{O}(\mathscr{H})$ be an orthogonal representation of a countable discrete group on a separable real Hilbert space and $\mq \colon G \rightarrow \mathscr{H}$ be an array into $\pi$.  Let $V^\mq_t$ be the Gaussian deformation as described above and $e \in \mathbb{B}(L^2(Y^\pi,\nu^\pi) \otimes L^2(\mathcal{N}) \otimes \ell^2(G))$ be the orthogonal projection onto $1\otimes L^2(\mathcal{N}) \otimes \ell^2(G)$.  Then the following hold:
\begin{enumerate}
	\item\label{propo:Gaussian_transversality} (Transversality, \cite[Lemma 2.8]{CS11}) Given $t\in \mathbb{R}$, $\eta \in L^2(\mathcal{N}) \otimes\ell^2(G)$ and fixing $e^\perp = 1 - e$ we have
	$$\Norm{e^\perp\circ V_t^\mq (\eta)}_2^2 \leqslant \Norm{\eta - V_t^\mq(\eta) }_2^2 \leqslant 2 \Norm{e^\perp \circ V_t^\mq(\eta)}_2^2 . $$

    \item\label{propo:Gaussian_bimodularity} (Asymptotic bimodularity, \cite[Corollary 5.5]{CSU11} and \cite[Proposition 1.10]{CSU13})\\
    For any $x,y\in \mathcal{N} \rtimes_{\rho,r} G $ (the reduced $C^*$-crossed product) we have
    $$ \lim_{t\rightarrow 0} \left( \sup_{\eta \in (L^2(\mathcal{D}\rtimes G))_1}\Norm{V^\mq_t(x\eta y) - x V_t^\mq(\eta)y}_2 \right) = 0 $$
\item\label{spectralgap} (Spectral gap,  \cite[Theorem 3.2]{CS11})\\ Assume that $\pi$ is a weakly-$\ell^2$ representation. For any projection $p\in \mathcal N\rtimes G=:\cM$ and any von Neumann algebra $\mathcal A\subseteq p\cM p$ with non-amenable relative commutant $\mathcal A'\cap p\cM p$ we have $$\lim_{t\ra 0}\left(\sup_{x\in (\mathcal A)_1}\| x-V^\mq_t (x)\|_2\right)=0.$$
\end{enumerate}
\end{prop}

\subsection{Ioana's deformation of infinite tensor product von Neumann algebras}

Next we recall the deformation of a von Neumann algebra $\mathcal M \bar\otimes \L(G)$ arising from a group $G =  \oplus_J A $ where $A$ is a countable group and $J$ is an infinite countable set, introduced in \cite{Io06}. Notice that we have $\Nn:=\mathcal M \bar\otimes \L(G)= \mathcal M\bar\otimes (\bar \otimes_J \L(A))$.  Following \cite[Proposition 2.3]{Io06}  consider the group $\tilde A= A\ast \mathbb Z$ and denote by $\tilde G = \oplus_J \tilde A$. Notice that $\mathcal M\bar\otimes \L(G)\subset \mathcal M\bar\otimes \L(\tilde G)=:\tilde \Nn$. Moreover, let $u\in \L(\mathbb{Z})\cong L^\infty([-\pi,\pi])$ be the canonical generating Haar unitary and $h=h^* \in \L(\mathbb{Z})$ be such that $e^{ih} = u$ (we fix $h\colon [-\pi,\pi] \ra [-\pi,\pi]$ the identity function for simplicity and for fixing constants). Then for every $t\in \mathbb R$ we consider the unitary $u^t = \exp(it h)\in \L(\mathbb Z)$. This enables us to define a path $\alpha\colon \mathbb{R} \rightarrow \Aut(\Tilde{\Nn})$ given by
$$\alpha_t (x \otimes (\otimes_j a_j)) = x \otimes (\otimes_j \Ad_{u^t}(a_j)) \text{\ \ for every }t\in \mathbb{R} \text{ ,\ } x\in \mathcal M, \otimes_j a_j \in \L(\Tilde{G}).$$
It is straight forward to check that this is a continuous path of trace-preserving automorphisms.  The choice of $h$ above follows the conventions used in \cite[Theorem 4.2]{IPV10}, in particular notice $\tau(u^t) = \frac{\sin(\pi t)}{\pi t}$ and $E_{\Nn}(\alpha_t(x\otimes (\otimes_j a_j)) = |\frac{\sin(\pi t)}{\pi t}|^{2|\ms(a)|} (x\otimes (\otimes_j a_j))$ whenever $x\in \M$ and $(\otimes_j a_j)\in \oo_J \cL(A)$ satisfies $a_j\in \cL(A) \ominus \mathbb{C}1$ for $j\in \ms(a)$, and $a_j\in \mathbb{C}1$ for $j\notin \ms(a)$.  For further use we next recall two important properties established in \cite{Io06} regarding the analysis of this deformation. 

\begin{lem}[{\cite[Lemma 2.4]{Io06}}]
Assume the notations of the preceding paragraph and let $\P\subseteq \Nn$ be a von Neumann subalgebra.  Suppose there is $t>0$, a projection $p\in \P$ and a non-zero partial isometry $v\in \Tilde{\Nn}$ satisfying $v^*v \leqslant p$ and $\alpha_t(x)v=vx$ for all $x \in p\P p$.  Then there is a $g\in G$, a finite set $F\subset J$ and a constant  $C>0$ such that for all $u \in \sU(p\P p)$ we have
$$ \Norm{ E_{\mathcal M \bar\otimes \L (\oplus_F A))} (\alpha_t(u)vu_g^*)) }_2 \geqslant C . $$
\end{lem}

\begin{lem}[{Proof of \cite[Theorem 3.6]{Io06}}]\label{ioanacontrol} Assume the notations of the preceding paragraph and let $\P\subseteq \Nn$ be a von Neumann subalgebra.  Suppose there is $t>0$, a projection $p\in \P$ and a non-zero partial isometry $v\in \Tilde{\Nn}$ satisfying $v^*v \leqslant p$ and $\alpha_t(x)v=vx$ for all $x \in p\P p$.  Then there is  a finite set $F\subset J$ such that $\P\prec \mathcal M \bar\otimes \L(\oplus_F A)$.
\end{lem}

 Now we are ready to derive a subordination relation between the restriction of the support array deformation $V^\mq_t$ to the core of the wreath-like product von Neumann algebra and Ioana's deformation $\alpha_t$ on this core. This relies on two elementary inequalities. For reader's convenience we also include their proofs.

\begin{prop}\label{ineq1} For every $t>0$ the following inequalities hold
\begin{equation}\label{basicineq1}
    1-\frac{1}{t+1}\geqslant e^{-\frac{1}{t}}\geqslant 1-\frac{1}{t}.
\end{equation}
\end{prop}

\begin{proof} Using the MVT for $x\ra \ln(x)$ in the interval $[t,t+1]$ we see that for all $t>0$ we have $\frac{1}{t}\geqslant \ln(t+1)-\ln(t)$. This implies that $-\frac{1}{t}\leqslant \ln(\frac{t}{t+1})$. Exponentiating we get $e^{-\frac{1}{t}}\leqslant \frac{t}{t+1}=1-\frac{1}{t+1}$, which yields the first inequality. The second inequality holds trivially when $t\leqslant1$, so assume $t> 1$. Once again using MVT in the interval $[t-1,t]$ we see that $\ln(t)-\ln(t-1)\geqslant \frac{1}{t}$. Thus $\ln(\frac{t}{t-1})\geqslant \frac{1}{t}$ and hence $-\frac{1}{t}\geqslant \ln(\frac{t-1}{t})$. Therefore $e^{-\frac{1}{t}}\geqslant 1-\frac{1}{t}$, as desired. \end{proof}

\begin{prop}\label{ineq2} For every $1>t\geqslant 0$ the following inequalities hold
\begin{equation}\label{basicineq2}
e^{-\frac{1}{3}t^2}\geqslant \frac{\sin^2(t)}{t^2}\geqslant e^{-\frac{(\pi^2+5)}{3(\pi^2-1)} t^2}.
\end{equation}
    
\end{prop}

\begin{proof} Recall that for all $t\in \mathbb R$ we have that 
\begin{equation}\label{decomp}
    \frac{\sin(t)}{t}= \prod_{n\in \mathbb N}\left(1-\frac{t^2}{\pi^2n^2}\right).
\end{equation}
Using this in combination with the second inequality in Proposition \ref{ineq1} we obtain 
 that $ \frac{\sin(t)}{t}= \prod_{n\in \mathbb N}(1-\frac{t^2}{\pi^2n^2})\leqslant \prod_{n\in \mathbb N} e^{-\frac{t^2}{\pi^2n^2}}=  e^{-\frac{t^2}{\pi^2}(\sum_n \frac{1}{n^2})}= e^{-\frac{t^2}{6}}$. Squaring, we get the first inequality in \eqref{basicineq2}.    

 Using the first inequality in Proposition \ref{ineq1} together with \eqref{decomp} we can see for all $1>t\geqslant 0$ we have \begin{equation*}
    \begin{split}
        \frac{\sin(t)}{t}&=\prod_{n\in \mathbb N}\left(1-\frac{t^2}{\pi^2n^2}\right)\geqslant   \prod_{n\in \mathbb N} e^{-\frac{t^2}{\pi^2n^2-t^2}}\\&=e^{-\frac{t^2}{\pi^2-t^2}} \prod^\infty_{n=2} e^{-\frac{t^2}{\pi^2n^2-t^2}}\geqslant e^{-\frac{t^2}{\pi^2-t^2}} \prod^\infty_{n=2} e^{-\frac{t^2}{\pi^2(n-1)^2}}\\
        &=e^{-\frac{t^2}{\pi^2-t^2}}e^{-\frac{t^2}{\pi^2}(\sum_n \frac{1}{n^2})}= e^{-\frac{t^2}{\pi^2-t^2}- \frac{t^2}{6} }\\&\geqslant e^{-\frac{(\pi^2+5)t^2}{6(\pi^2-1)}}. 
    \end{split}
\end{equation*}
Squaring we get the second inequality in \eqref{basicineq2}.\end{proof}

\begin{thm}\label{subrel} Let $\mathcal{M} = \mathcal{N}\bar \otimes \cL(A^{(I)})\subset \mathcal N \oo \cL(G)$ where $G \in \mathcal{WR}_b(A, B\ca I)$. Then for every $\xi\in L^2( \cM)$ and any $0 \leq t < 1/\pi $ we have the following inequalities for the array $\mq$ defined in Theorem \ref{supportarray}:
    \begin{equation}
        \|e_\mathcal{M}^{\perp}\circ V^\mq_{\sqrt{\frac{1}{6}} \pi t}(\xi)\|_2\geqslant \| e_\mathcal{M}^{\perp}\circ \alpha_t(\xi)\|_2\geqslant\|e_\mathcal{M}^{\perp}\circ V^\mq_{\sqrt{\frac{1}{6}} \pi t}(\xi) \|_2.
    \end{equation}
\end{thm}

\begin{proof} Let $\xi= \sum_{h\in A^{(I)}}\xi_h \otimes \delta_h$ the canonical decomposition for $\xi_h\in L^2(\cN)$. Using the definitions of $V^\mq_t$ and the fact that $|\ms(h)|:=|\supp(h)| =\|\mq(h)\|^2$ for all $h\in A^{(I)}$ we see that 
\begin{equation}\begin{split}
   \| e_\mathcal{M}^{\perp}\circ V^\mq_{\sqrt{\frac{1}{3}} t}(\xi)\|_2^2&= \sum_h\| (\omega(\sqrt{\frac{1}{3}} t \mq(h))-\tau(\omega(\sqrt{\frac{1}{3}} t \mq(h))) (1\otimes \xi_h) ) \|^2_2\\
   &= \sum_h \|\xi_h\|^2_2 (1- e^{-2\frac{t^2}{3} \|\mq(h)\|^2})\\
   &= \sum_h \|\xi_h\|^2_2 (1- e^{-2\frac{t^2}{3} |\ms(h)|}).\end{split}
\end{equation}

 On the other hand, using \cite{Io06,IPV10}, for every $t$ we have that 

\begin{equation}
    \|e_\mathcal{M}^{\perp}\circ \alpha_t(\xi)\|_2^2 = \sum_h \|\xi_h\|_2^2\left (1- \left(\frac{\sin^2(\pi t)}{\pi^2 t^2}\right)^{|\ms(h)|} \right)
\end{equation}
  
  These formulae combined with the inequalities in Proposition \ref{ineq2} give the desired conclusion.  \end{proof}

\begin{prop}\label{ineqballsioanadef} Let $G \in WR_b(A, B\ca I )$ be a wreath-like product and let $H$ be an arbitrary countable group  Also let $\cN = \cL(A^{(I)})\bar \otimes \cL(H)\subset \cL(G)\bar\otimes \cL(H)=\cM$.  Let $\ma\colon G \ra \mathscr H $ be the array introduced in Proposition \ref{relsolarray}.  Then for every $d\geq 0$ one can find a finite subset $F\subset B$  and $d'>2$ such that $$B^\ma_d \subset \{ \gamma b \in G \colon \gamma\in A^{(I)}, |\supp(\gamma)| \leqslant d', b \in F\}.$$  In addition, for every $z\in \cN$ we have that 
\begin{equation}\label{coparisonballsioanadef}
\Norm{P_{B_d^{\ma}\times H}(z)}_2^2 \leqslant 2 \innpr{(\alpha_t\otimes\mathrm{Id}) (z), z }, \quad \text{for all }  0 \leqslant t < \min \left\{ \sqrt{\frac{\ln(2)}{13d'}}, \frac{1}{\sqrt{2} \pi} \right\} .
\end{equation}
\end{prop}

\begin{proof} First assertion follows directly from Proposition \ref{relsolarray}.  To see the second one, fix $z\in \cN$.  Using the Fourier expansion of $z = \sum_{h\in A^{(I)}} u_h z_h  \in \cL(A^{(I)})\oo\cL(H)=\cN$ with $z_h \in \cL(H)$ and $\{u_h\}_{h\in A^{(I)}}$ the canonical group unitaries in $\cL(A^{(I)})$ we  have
\begin{equation}\label{computation2}
\begin{split}
& \innpr{\alpha_{\sqrt{2}t} \otimes {\rm Id}(z),z } = \sum_{h\in A^{(I)}} \tau(z^* E_\cM \circ (\alpha_{\sqrt{2}t} \otimes {\rm Id})(u_hz_h)) = \sum_{h\in A^{(I)}} \Norm{z_h}_2^2\, \tau(u_h^*\,E_\cM (\alpha_{\sqrt{2}t} (u_h)))
\\
&= \sum_{h\in A^{(I)}} \Norm{z_h}_2^2\, \left(\frac{\sin^2(\sqrt{2}\pi t)}{2\pi^2 t^2}\right)^{|\ms(h)|} \geqslant \sum_{h\in A^{(I)}} \Norm{z_h}_2^2\, e^{-\frac{2(\pi^2 +5)}{3(\pi^2-1)}\pi^2 t^2|\ms(h)|} \geqslant \sum_{h\in A^{(I)}} \Norm{z_h}_2^2\, e^{-13t^2 |\ms(h)|}
\\  &\geqslant \sum_{h\in A^{(I)}\cap B_d^{\ma}} \Norm{z_h}_2^2\, e^{-13t^2 |\ms(h)|} \geqslant \sum_{h\in A^{(I)}\cap B_d^{\ma}} \Norm{z_h}_2^2\, e^{-13t^2 d'} = \Norm{P_{B_d^{\ma} \times H}(z)}_2^2\, e^{-13t^2 d'},
\end{split}
\end{equation}
where the equality in the second line comes from \cite[Equation (4.1)]{IPV10} and for the inequality right after we used \eqref{basicineq2} from Proposition \ref{ineq2}.  Letting $0 \leqslant t$ small enough in the inequality \eqref{computation2} we get  \eqref{coparisonballsioanadef}.
\end{proof}

\section{Popa's intertwining techniques}Almost two decades ago,  S.  Popa  introduced  in \cite [Theorem 2.1 and Corollary 2.3]{Po03} a powerful analytic criterion for identifying intertwiners between arbitrary subalgebras of tracial von Neumann algebras, see Theorem \ref{corner} below. This technique,  known as \emph{Popa's intertwining-by-bimodules technique},  has played an essential role in the classification of von Neumann algebras program via Popa's deformation/rigidity theory.  

\begin{thm}\emph{\cite{Po03}} \label{corner} Let $( \mathcal{M},\tau)$ be a  tracial von Neumann algebra and let $ \mathcal{P},  \mathcal{Q}\subseteq  \mathcal{M}$ be (not necessarily unital) von Neumann subalgebras. 
	Then the following are equivalent:
	\begin{enumerate}
		\item There exist projections $ p\in    \mathcal{P}, q\in    \mathcal{Q}$, a $\ast$-homomorphism $\theta \colon p  \mathcal{P} p\rightarrow q \mathcal{Q} q$  and a partial isometry $0\neq v\in  \mathcal{M} $ such that $v^*v\leqslant p$, $vv^*\leqslant q$ and $\theta(x)v=vx$, for all $x\in p  \mathcal{P} p$.
		\item For any group $\mathcal G\subset \mathscr U( \mathcal{P})$ such that $\mathcal G''=  \mathcal{P}$ there is no net $(u_n)_n\subset \mathcal G$ satisfying $\|E_{  \mathcal{Q}}(xu_ny)\|_2\rightarrow 0$, for all $x,y\in   \mathcal{M}$.
		\item There exist finitely many $x_i, y_i \in  \mathcal{M}$ and $C>0$ such that  $\sum_i\|E_{  \mathcal{Q}}(x_i u y_i)\|^2_2\geqslant C$ for all $u\in \mathscr U( \mathcal{P})$.
		\item There exists a non-zero projection $f\in \mathcal{P}'\cap \langle \mathcal{M}, e_\mathcal{Q} \rangle$ such that ${\rm Tr}(f)<\infty$.
  \item There exists a $\mathcal{P}$-$\mathcal{Q}$-subbimodule $\mathscr H$ of $1_\mathcal{P}\emph{L}^2(\mathcal{M})1_\mathcal{Q}$ such that $\emph{dim}(\mathscr H_\mathcal{Q})<+\infty$.
	\end{enumerate}
\end{thm} 

 If one of the five equivalent conditions from Theorem \ref{corner} holds, then we say that \emph{ a corner of $ \mathcal{P}$ embeds into $ \mathcal{Q}$ inside $ \mathcal{M}$}, and write $ \mathcal{P}\prec_{ \mathcal{M}} \mathcal{Q}$. If we moreover have that $ \mathcal{P} p'\prec_{ \mathcal{M}} \mathcal{Q}$, for any projection  $0\neq p'\in  \mathcal{P}'\cap 1_{ \mathcal{P}}  \mathcal{M} 1_{ \mathcal{P}}$, then we write $ \mathcal{P}\prec_{ \mathcal{M}}^{\rm s} \mathcal{Q}$ (where the superscript ``s" stands for {\it strong}).

For further use, we recall several useful intertwining results for von Neumann subalgebras. The first is a result controlling quasinormalizers in wreath-like crossed-product von Neumann algebras, in the same spirit with \cite[Theorem 3.1]{Po03}. For a proof, in the same vein with this, the reader may consult \cite[Corollary 3.8]{CFQT23}.

\begin{thm}[\text{\cite[Theorem 3.1]{Po03} and \cite[Corollary 3.8]{CFQT23}}] Let $G\in \mathcal{WR}(A,B\curvearrowright I)$ and let $\mathcal N$ be a tracial von Neumann algebra. Let $G\curvearrowright \mathcal N$ be a trace-preserving action and denote by $\M=\mathcal N\rtimes G$ the corresponding crossed-product von Neumann algebra. Let $F \subset I$ be a finite subset and let $q\in \mathcal N\rtimes A^F$ be a non-zero projection. 

Assume $\Q \subseteq q(\mathcal{N}\rtimes A^F)q$ is a von Neumann subalgebra satisfying $\Q \nprec_{\mathcal N\rtimes A^F} \mathcal N \rtimes A^K$, for all $K\subsetneq F$. 

If we denote by ${\rm Norm}(F)= \{g \in B \,:\, gF =F\}$, then we have that $\mathscr {Q N}^{(1)}_{q\M q} (\Q)\subseteq \mathcal N\rtimes (A^{(I)} {\rm Norm}(F))$.
\end{thm}

\begin{thm}\label{quasinormcontrol2} Let $G\in \mathcal{WR}(A,B\curvearrowright I)$ and let $\mathcal N$ be a tracial von Neumann algebra. Let $G\curvearrowright \mathcal N$ be a trace-preserving action and denote by $\M=\mathcal N\rtimes G$ the corresponding crossed-product von Neumann algebra. Let $F \subset I$ be a finite subset and let $q\in \mathcal N\rtimes A^F$ be a non-zero projection. 

 Assume $\Q \subseteq q\M q$ is a von Neumann subalgebra satisfying  $\Q \prec_\M \mathcal N\rtimes A^F$ and $\Q \nprec_{\M} \mathcal N \rtimes A^K$, for all $K\subsetneq F$. 

 Then we have that $\mathscr {Q N}_{q\M q} (\Q)''\prec \mathcal N\rtimes (A^{(I)} {\rm Norm}(F))$.
\end{thm}

\begin{cor}\label{quasinormcontrol3} Let $G\in \mathcal{WR}(A,B\curvearrowright I)$ where $A$ and $B$ are nontrivial and the action $B \curvearrowright I$ has finite stabilizers. Let $\mathcal N$ be a tracial von Neumann algebra and assume that $G\curvearrowright \mathcal N$ is a trace-preserving action. Denote by $\M=\mathcal N\rtimes G$ the corresponding crossed-product von Neumann algebra. Let $F \subset I$ be a finite subset and let $q\in \mathcal N\rtimes A^F$ be a non-zero projection. 

Assume $\Q \subseteq q\M q$ is a von Neumann subalgebra satisfying  $\Q \prec_\M \mathcal N\rtimes A^F$ and $\Q \nprec_{\M} \mathcal N \rtimes A^K$, for all $K\subsetneq F$. 

 Then we have that $\mathscr {Q N}_{q\M q} (\Q)''\prec \mathcal N\rtimes A^{(I)}$.
\end{cor}

\begin{proof} Applying Theorem \ref{quasinormcontrol2} we get that $\mathscr {Q N}_{q\M q} (\Q)''\prec \mathcal N\rtimes (A^{(I)} {\rm Norm}(F))$. Since $F$ is nonempty and finite and the stabilizers of $B\ca I$ are finite it follows that ${\rm Norm} (F)$ is finite. This implies in particular that $\mathcal N \rtimes A^{(I)}\subseteq \mathcal N\rtimes (A^{(I)} {\rm Norm}(F)) $ is a finite index inclusion of von Neumann algebras. Using Lemma 3.9 in \cite{Va08} we conclude that $\mathscr {Q N}_{q\M q} (\Q)''\prec \mathcal N\rtimes A^{(I)} $. \end{proof}

\section{Proofs of the main results}\label{sec. vNa main}

Before turning to the proofs of the main results, we discuss a natural extension of the notion of McDuff superrigidity introduced earlier.

Since the isomorphism class of McDuff factors does not change under tensoring with the hyperfinite factor, any $W^*$-McDuff group $G$ satisfies  
\[
\L( G\times A)\cong \L(G\times B)
\]
for all ICC amenable groups $A$ and $B$. Thus, a natural broadening of McDuff superrigidity is a formulation in which the reconstruction of a $W^*$-McDuff group from its factor is complete only modulo this intrinsic obstruction. This motivates the following notion.

\begin{defn}\label{genmdsup}
A group $G$ is \emph{McDuff $W^*$-superrigid} if for every group $H$ with $\L(G)\cong \L(H)$, there exist ICC amenable groups $A$ and $B$ such that  
\[ H\times B \cong G\times A. \]
If one may take $B=1$ above, we say the reconstruction is the \emph{strongest possible} in this framework, corresponding to the McDuff superrigidity notion presented in the introduction.
\end{defn}

In this section, we show that a large class of direct-sum groups satisfies McDuff superrigidity from Definition~\ref{genmdsup} across its full range---for example, when one can take $B=1$ (Theorem~\ref{main2}), but also when $A$ and $B$ are both arbitrary nontrivial ICC amenable groups (Theorem~\ref{main3}). Our arguments follow the general strategy for reconstructing infinite direct sums developed in \cite[Sections~2--3]{CU18}. Several results from those sections will be used without proof, and we encourage the reader to consult them in advance.

We begin by setting up notation and establishing several preliminary facts.

\noindent \emph{Notation A.} Let $0\in J$ where $J$ is a countable (possibly finite) set.  Let $G_0$ be an ICC amenable group which is possibly trivial when $J$ is infinite.  For every $j\in J\setminus \{0\}$ let $G_j \in \mathcal \WR_b(A_j,B_j\ca I_j)$ be a wreath-like product group where $A_j$ is an amenable group, $B_j$ is a non-amenable subgroup of a hyperbolic group and the action $B_j\ca I_j$ has amenable stabilizers.   Denote by $G=\oplus_j G_j$ and let $\cM=\cL(G)$.  For $F\subset J$ we denote $\hat{F}=J\setminus F$, $\cM_F = \cL(\oplus_{j\in F} G_j)$, $\C_F = \cL(\oplus_{j\in F} A_j^{(I_j)})$ and $\D_F = \C_F \oo \M_{\hat{F}}$.  When we deal with $G\times G = (\oplus_{j\in J} G_j) \times (\oplus_{j\in J} G_j)$ we will see this as $G^2 = \oplus_{j\in J\times \{1,2\}} G_j$ where the second index  $1$ or $2$ indicates if we are in the first or second summand of $G$ inside $G\times G$.  In that case $\hat{F}$ will represent $(J\times \{1,2\}) \setminus F$ and we will use $\cM^{\oo2}_F = \cL(\oplus_{j\in F} G^2_j)$.  The notational use of single or double index will also be clear from context. 

Assume that  $\M=\cL(H) $ for an arbitrary group $H$.  Following \cite{IPV10}, let $\Delta\colon \cM\to \cM\bar\otimes \cM $ be the commultiplication along $ H$, i.e.\ $\Delta(v_h)=v_h\otimes v_h$, where $\{v_h\}_{h\in H} $ are the canonical unitaries generating $\cL(H)$. 

\begin{thm}\label{intertwining}
Assume the Notation A above. Then for each $i\in J\setminus \{0\}$ there is $k\in J$ so that $\Delta(\cM_{\hat{i}}) \prec_{\cM\oo\cM} \cM \oo \cM_{\hat{k}}$.
\end{thm}

Before proceeding to the proof of the theorem we stablish the following preliminary result.

\begin{lem}\label{proptfirst intertwining}
Under the Notation A,  there exist non-empty finite subsets $K\subset F \subset J$ where $K$ is maximal and $F$ minimal with the property that $\Delta(\cM_{\hat{i}}) \prec^s \D_K \oo \D_K$ and $\Delta(\cM_i) \prec^s \cM_F \oo \cM_F$.
\end{lem}

\begin{proof}
Since $G_i$ has property (T), then so does $\Delta(\cM_i)$.  Given that $\Delta(\cM_i)\subset (\bigcup_{F\Subset J} \cM \oo \cM_F)''$ then one can find a finite subset $F\Subset J$ such that $\Delta(\cM_i)\prec \cM \oo \cM_F$.  Moreover, we can pick $F$ to be minimal admitting this intertwining.  Notice that $\cN_{\cM\oo\cM}(\Delta(\cM_i))'\cap \cM \oo\cM = \mathbb{C}1$ and so we further obtain $\Delta(\cM_i)\prec^s \cM \oo \cM_F$ by applying \cite[Lemma 2.4(3)]{DHI16}.  Since $\Delta$ is invariant with respect to the flip automorphism of $\cM\oo\cM$ we also have $\Delta(\cM_i)\prec^s \cM_F \oo \cM$ and thus  we obtain $\Delta(\cM_i)\prec^s \cM_F \oo \cM_F$ by applying \cite[Lemma 2.8(2)]{DHI16}.  From Proposition \ref{relsolarray} we see that $G\times G$ is bi-exact relative to the collection $\mathscr{G}=\{ A_{j}^{(I_j)}\times G_{\hat{j}} \colon j \in J\times \{1,2\} \}$.  As $\cM\oo\cM = \cL(G\times G)$ and $\Delta(\cM_i)$ is non-amenable, then $\Delta(\cM_{\hat{i}}) \prec_{\cL(G\times G)} \cL(\Lambda)$ for some $\Lambda \in \mathscr{G}$ by \cite[Theorem 15.1.5]{BO08}.  From the invariance of $\Delta$ under the flip automorphism, we can assume there is $j\in J$ such that $\Delta(\cM_{\hat{i}}) \prec \cM \oo \D_j$ and with a similar argument as the one above using \cite[Lemma 2.4(3)]{DHI16}, we obtain $\Delta(\cM_{\hat{i}}) \prec^s \cM \oo \D_j$.

 We now show $j$ must belong to $F$.  Indeed, assuming for a contradiction that $F\subset \hat{j}$ we obtain $\Delta(\cM_{\hat{i}}), \Delta(\cM_i) \prec \cM \oo \D_j$.  Since $\cM\oo\D_j$ is regular in $\cM\oo\cM$, and $\Delta(\cM_{\hat{i}}),\Delta(\cM_{i})$ commute with each other, then by \cite[Proposition 4.3]{CD-AD21} we have that $\Delta(\cM) \prec \cM\oo \D_j = \cL(G\times (A_j^{(I_j)} \times G_{\hat{j}}))$.  However, this intertwining would imply by \cite[Proposition 7.2(2)]{IPV10} that $A_j^{(I_j)} \leqslant G_j$ has finite index, which in turn contradicts $B_j$ being infinite.

 So far we have proved that there is a minimal finite subset $F\subset J$ such that  $\Delta(\cM_i)\prec^s \cM_F \oo \cM_F$ and there exists $j\in F$ such that $\Delta(\cM_{\hat{i}}) \prec^s \cM \oo \D_j$.  Let $K\subset F$ be the set of all $j\in F$ such that $\Delta(\cM_{\hat{i}}) \prec^s \cM \oo \D_j$.  By applying \cite[Lemma 2.8(2)]{DHI16} as above repeatedly, $|K|-1$ many times, we get $\Delta(\cM_{\hat{i}}) \prec^s \cM \oo\D_K$ and $K$ is maximal with this property.  Moreover, by the invariance of $\Delta$ under the flip and using \cite[Lemma 2.8(2)]{DHI16} we obtain $\Delta(\cM_{\hat{i}}) \prec^s \D_K \oo \D_K$.

 We note in passing that when $J$ is finite, at the beginning of the proof, we do not need to use that $G_i$ has property (T). 
\end{proof}

 We now proceed to the proof of the Theorem \ref{intertwining}.

\begin{proof}
From Lemma \ref{proptfirst intertwining} and Popa's intertwining techniques we obtain non-zero projections $p\in \Delta(\cM_{\hat{i}})$, $q\in \D_K \oo \D_K$, a non-zero partial isometry $v\in q (\cM\oo\cM) p$ and a normal $*$-isomorphism onto its image
\begin{equation}\label{intert1here}
\theta \colon p \Delta(\cM_{\hat{i}}) p \ra \theta(p\Delta(\cM_{\hat i})p)=:\cQ \subset q (\D_K \oo \D_K) q \ , \text{ so that } \ \theta(x)v = vx \ \text{ for all } \ x\in p \Delta(\cM_{\hat{i}}) p\,.
\end{equation}

Consider the array $\ma\colon G\times G \ra \sH$ constructed in Proposition \ref{relsolarray} by seeing $G\times G$ as $\oplus_{j\in J\times\{1,2\}} G_j$ and let $V_t^\ma \in \sU(L^2(L^\infty(\mathbb{X}^\pi)\rtimes (G\times G)))$ be the path of unitaries corresponding to the deformation associated to this array as in \cite{CS11}.  Using part \ref{spectralgap} in Proposition \ref{propo:Gaussian_properties}  and the non-amenability of $\Delta(\cM_i)$ we obtain that $$\sup_{x\in (\Delta(\cM_{\hat{i}}))_1} \Norm{x - V_t^\ma(x)}_2 \xrightarrow{t\to 0} 0\,.$$

Using Kaplanski's Density Theorem, the asymptotic bimodularity of $V_t^\ma$ over $C_r^*(G\times G)$ from part \eqref{propo:Gaussian_bimodularity} in Proposition \ref{propo:Gaussian_properties} and the triangle inequality we obtain $\Norm{vxv^* - V_t^\ma(vxv^*)}_2 \xrightarrow{t\to0} 0$ uniformly on the unit ball of $\Delta(\cM_{\hat{i}})$ and therefore
\begin{equation}
\sup_{y\in (\Q)_1} \Norm{yvv^* - V_t^\ma(yvv^*)}_2 \xrightarrow{t\to 0} 0
\end{equation}
by applying the intertwining relation \eqref{intert1here}.  Let $1>\varepsilon>0$ and use Kaplanski's Density Theorem to find $x \in (\mathbb{C}[G\times G])_1$ such that $\Norm{vv^* - x}_2 \leqslant \varepsilon/10$.  Let $\tilde{K}= K\times \{1,2\}$ (using the double index notation for $G^2$ introduced earlier) and write $x = \sum_{h\in R_\varepsilon} x_h u_h$ where $R_\varepsilon \Subset B_{\tilde{K}}$ is a finite subset (seen inside $G^2_{\tilde{K}}$ via any set-theoretic section) and $x_h \in \mathbb{C}[G_{\hat{K}}\times \oplus_{j\in K} A_j^{(I_j)} \times G_{\hat{K}}\times \oplus_{j\in K} A_j^{(I_j)}]$.  Let $t_\varepsilon>0$ be such that $\Norm{yvv^* - V_t^\ma(yvv^*)}_2 \leqslant \varepsilon$ and $\Norm{V_t^\ma(yx) - \sum_{h\in R_\varepsilon} V_t^\ma(yx_h)u_h}_2 \leqslant \varepsilon/10$ (from asymptotic bimodularity) for all $y\in (\Q)_1$ and all $0<t\leqslant t_\varepsilon$, then we have
\begin{equation}\label{eq approximIntertwiner}
\begin{split}\sum_{h\in R_\varepsilon} \Norm{yx_h - V_t^\ma(yx_h)}^2_2 = \Norm{yx - yvv^* + yvv^* - V_t^\ma(yvv^*) + V_t^\ma(yvv^*) - \sum_{h\in R_\varepsilon} V_t^\ma(yx_h)u_h }_2^2\\
\leqslant (\varepsilon/10 + \varepsilon + \Norm{V_t^\ma(yvv^*) - V_t^\ma(yx) + V_t^\ma(yx) - \sum_{h\in R_\varepsilon} V_t^\ma(yx_h)u_h }_2 )^2 \leqslant 2\varepsilon^2\end{split}
\end{equation}
for all $y\in (\Q)_1$ and all $0<t\leqslant t_\varepsilon$.

 From the construction of the Gaussian deformation and the array $\ma$ we have some $f>0$ such that $\sum_{h\in R_\varepsilon} \Norm{yx_h - P_{B_f^\ma}(yx_h)}^2_2 \leqslant 2\varepsilon^2$ for all $y\in (\Q)_1$ and so $\sum_{h\in R_\varepsilon} \Norm{yx_h}^2_2 \leqslant 2\varepsilon^2 + \sum_{h\in R_\varepsilon} \Norm{P_{B_f^\ma}(yx_h)}^2_2$.  Furthermore,  using the structure of the finite radius balls of the array $\ma$ from Proposition \ref{relsolarray}, there is a finite subset $S\Subset J\times\{1,2\}$ and a constant $d > 0$ such that $B_f^\ma \subset \bigcup_{j\in S}B_d^{\ma_j}\times G^2_{\hat{j}}$.  Now we split $S = (S\cap \tilde{K}) \sqcup (S\setminus \tilde{K})$ and use this to rewrite and bound the last inequality as
\begin{equation}\label{firstestimate}
\begin{split}
\sum_{h\in R_\varepsilon} \Norm{yx_h}^2_2 \leqslant 2\varepsilon^2 + \sum_{h\in R_\varepsilon} \Norm{P_{B_f^\ma}(yx_h)}^2_2 \leqslant 2\varepsilon^2 + \sum_{h\in R_\varepsilon} ( \sum_{j\in S\cap\tilde{K}} \Norm{P_{B_d^{\ma_j}\times G^2_{\hat{j}}}(yx_h)}^2_2 + \sum_{j\in S\setminus\tilde{K}} \Norm{P_{B_d^{\ma_j}\times G^2_{\hat{j}}}(yx_h)}^2_2)\\
\leqslant 2\varepsilon^2 + \sum_{h\in R_\varepsilon} ( \sum_{j\in S\cap\tilde{K}} \Norm{P_{B_{d}^{\ma_j}\times G^2_{\hat{j}}}(yx_h)}^2_2 + \sum_{j\in S\setminus\tilde{K}} \sum_{\lambda \in F_j} \Norm{E_{\C_j \oo \M^{\oo2}_{\hat{j}}}(yx_h u_{\lambda}^*)}^2_2)
\end{split}
\end{equation}
In the last line of \eqref{firstestimate},  we have used from Proposition \ref{ineqballsioanadef} the existence of a finite subset $F_j\subset B_j$ and an integer $d'>2$ such that $B_d^{\ma_j} \subset \{ \gamma \lambda \in G_j : \gamma\in A_j^{(I_j)}, |\supp(\gamma)| \leqslant d', \lambda \in F_j \}$ for each $j\in S$.  Notice that $yx_h \in \D_K \oo \D_K \subset \C_j \oo \M^{\oo2}_{\hat{j}}$ so we can apply Ioana's tensor length deformation $\alpha^j_t$ (acting on $\C_j$ and leaving $\M^{\oo2}_{\hat{j}}$ invariant) to each $yx_h$ for every $j\in \tilde{K}$.  Thus using the following inequality 
\begin{equation}
\Norm{P_{B_d^{\ma_j}\times G^2_{\hat{j}}}(z)}_2^2 \leqslant 2 \innpr{\alpha_t^j (z), z } \quad \text{ for all } \ z \in \C_j\oo \cM^{\oo2}_{\hat{j}} \ \text{ and } \ 0 \leqslant t < \min \left\{ \sqrt{\frac{\ln(2)}{13d'}}, \frac{1}{\sqrt{2} \pi} \right\}, \,
\end{equation}
from Proposition \ref{ineqballsioanadef}, we can bound from above the last quantity in \eqref{firstestimate} as follows.

For any $y\in (\Q)_1$ and $t$ satisfying the above bounds, we have
$$ \sum_{h\in R_\varepsilon} \Norm{yx_h}^2_2 \leqslant 2\varepsilon^2 + \sum_{h\in R_\varepsilon} ( \sum_{j\in S\cap \tilde{K}} 2\innpr{\alpha^j_t (yx_h), yx_h} + \sum_{j\in S\setminus \tilde{K}} \sum_{\lambda \in F_j} \Norm{E_{\C_j \bar\otimes \M^{\oo2}_{\hat{j}}} (yx_h u_{\lambda}^*)}^2_2)\,. $$
Since $\alpha^j_t(yx_h) = \alpha_t^j(y)\alpha_t^j(x_h)$ and $\alpha_t^j \xrightarrow{t\to 0} 0$ pointwise,  then taking $t_\varepsilon>0$ smaller if needed, we obtain
\begin{equation} \label{thirdestimate}\sum_{h\in R_\varepsilon} \Norm{yx_h}^2_2 \leqslant 3\varepsilon^2 + \sum_{h\in R_\varepsilon} ( \sum_{j\in S\cap \tilde{K}} 2\innpr{\alpha^j_t (y)x_h, yx_h} + \sum_{j\in S \setminus \tilde{K}} \sum_{\lambda \in F_j} \Norm{E_{\C_j \bar\otimes \M^{\oo2}_{\hat{j}}} (yx_h u_{\lambda}^*)}^2_2)\,. \end{equation}

 Now,  notice that for any unitary $u\in \sU(\Q)$ we have $\Norm{vv^*}_2^2 - \varepsilon/4 \leq \Norm{x}_2^2 = \sum_{h\in R_\varepsilon}\Norm{u x_h}_2^2$.  For each $j\in S \setminus \tilde{K}$ we consider the inclusion of von Neumann algebras $\C_j\oo\M^{\oo2}_{\hat{j}} \subset \M\oo\M$ and let $\mathrm{Tr}_{\langle \M\oo\M, e_{\C_j\oo \M^{\oo2}_{\hat{j}}} \rangle}$ be the canonical semi-finite trace on the corresponding basic construction von Neumann algebra  $\langle \M\oo\M, e_{\C_j\oo \M^{\oo2}_{\hat{j}}} \rangle$.  Using the formulae ${\rm Tr}(xe_{\C_j\oo\M^{\oo2}_{\hat{j}}} y)=\tau(xy)$ and  $e_{\C_j\oo \M^{\oo2}_{\hat{j}}} xe_{\C_j\oo \M^{\oo2}_{\hat{j}}}=E_{\C_j\oo \M^{\oo2}_{\hat{j}}}(x)e_{\C_j\oo \M^{\oo2}_{\hat{j}}}$ for all $x,y\in \M\oo\M$ in conjunction with the above estimates and the inequality \eqref{thirdestimate} we obtain that 
\begin{equation}\label{fourthestimate}\begin{split}
&\Norm{vv^*}_2^2 -3\varepsilon^2 - \varepsilon/4\\& \leqslant 2 \sum_{j\in S\cap \tilde{K}} \innpr{\alpha^j_t(u), u ( \sum_{h\in R_\varepsilon} x_h x_h^*)} + \sum_{j\in S \setminus\tilde{K}} \innpr{ u^* e_{\C_j\oo \M^{\oo2}_{\hat{j}}} u, \sum_{h\in R_\varepsilon} \sum_{\lambda \in F_j} x_h u^*_\lambda e_{\C_j\oo \M^{\oo2}_{\hat{j}}} u_\lambda x_h^* }_{\mathrm{Tr}_{\langle \M\oo\M, e_{\C_j\oo \M^{\oo2}_{\hat{j}}} \rangle}}\end{split}
\end{equation}
for all $u\in \mathscr U (\Q)$ and $0< t < t_\varepsilon$. Next, we may assume $\varepsilon>0$ is small enough so that $D:=\|vv^*\|_2 -3\varepsilon^2-\varepsilon/4>0$.

In the last part of the proof we will use this inequality together with an averaging argument on Hilbert space to get our conclusion.  We will write the right hand side of (\ref{thirdestimate}) in a more compact and revealing way by considering the von Neumann algebra
$$\S = \bigoplus_{j\in S\cap \tilde{K}} (\tilde{\C}_j\oo \M^{\oo2}_{\hat{j}}) \oplus \bigoplus_{j\in S\setminus \tilde{K}} \innpr{\M\oo\M, e_{\C_j\oo\M^{\oo2}_{\hat{j}}}} $$
together with the faithful normal semi-finite tracial weight $\mathrm{Tr}_\S = \bigoplus_{j\in S\cap \tilde{K}} \tau_{\tilde{\C}_j\oo \M^{\oo2}_{\hat{j}}} \oplus \bigoplus_{j\in S\setminus \tilde{K}} \mathrm{Tr}_{\innpr{\M\oo\M, e_{\C_j\oo\M^{\oo2}_{\hat{j}}}}} $ on $\S$ (here $\tilde{\C}_j = \cL((A_j*\mathbb{Z})^{(I_j)})$ is the von Neumann algebra on which the deformation $\alpha^j_t$ acts).  Notice that $\bigoplus_{j\in S} \Q \subset \S$.  For each $j \in S \cap \tilde{K}$ let $a_j = \sum_{h\in R_\varepsilon} x_h x_h^* \in \D_K \oo \D_K \subset \tilde{\C}_j \oo \M^{\oo2}_{\hat{j}}$ and for each $j \in S \setminus \tilde{K}$ let $b_j = \sum_{h\in R_\varepsilon} \sum_{\lambda \in F_j} x_h u^*_\lambda e_{\C_j\oo\M^{\oo2}_{\hat{j}}} u_\lambda x_h^* \in \innpr{\M\oo\M, e_{\C_j\oo\M^{\oo2}_{\hat{j}}}}$ and define the operators
\begin{equation*}
\xi = \bigoplus_{j\in S\cap \tilde{K}} a_j \oplus \bigoplus_{j\in S\setminus \tilde{K}} e_{\C_j\oo\M^{\oo2}_{\hat{j}}} \qquad \eta = \bigoplus_{j\in S\cap \tilde{K}} 2 \oplus \bigoplus_{j\in S\setminus \tilde{K}} b_j
\end{equation*}
which can be seen to belong to $\S_+\cap L^2(\S,\mathrm{Tr}_\S)$.  Let $\varphi \in L^2(\S,\mathrm{Tr}_\S)^*$ be the linear functional given by inner product with $\eta$, then we see that the right hand side of (\ref{fourthestimate}) can be expressed in terms of $\varphi$ and the inequality can be rewritten as
\begin{equation}\label{compHspace_gap}
0<D \leqslant  \varphi ((\bigoplus_{j\in S\cap \tilde{K}} \alpha_t^j (u) \oplus \bigoplus_{j\in S\setminus \tilde{K}} u^*) \,\xi \,(\bigoplus_{j\in S\cap \tilde{K}} u^* \oplus \bigoplus_{j\in S\setminus \tilde{K}} u) )
\end{equation}
for all $u\in \sU(\Q)$ and $0< t \leqslant t_\varepsilon$. Consider the closed convex set
$$ \kappa = \overline{\mathrm{co}} \,\{ (\bigoplus_{j\in S\cap \tilde{K}} \alpha_t^j (u) \oplus \bigoplus_{j\in S\setminus \tilde{K}} u^*) \,\xi \,(\bigoplus_{j\in S\cap \tilde{K}} u^* \oplus \bigoplus_{j\in S\setminus \tilde{K}} u) : u\in \sU(\Q)\,, 0< t \leqslant t_\varepsilon \} \subset L^2(\S,\mathrm{Tr}_\S)$$
and let $m\in \kappa$ be the unique element of minimal $\Norm{\cdot}_{\mathrm{Tr}_\S,2}$ norm.  By the uniqueness of $m$ we obtain that
\begin{equation}\label{conjuginvariance}
(\bigoplus_{j\in S\cap \tilde{K}} \alpha_t^j (u) \oplus \bigoplus_{j\in S\setminus \tilde{K}} u^*) \,m \,(\bigoplus_{j\in S\cap \tilde{K}} u^* \oplus \bigoplus_{j\in S\setminus \tilde{K}} u) = m
\end{equation}
for any $u\in \sU(\Q)$ and $0< t \leqslant t_\varepsilon$. Moreover, from (\ref{compHspace_gap}) we observe that $0<D \leq \varphi(m)$, so $m = \oplus_{j\in S} m_j \neq 0$ and there are $k\in J$ and $j \in (\{k\} \times \{1,2\}) \cap S$ such that $m_{j} \neq 0$.  

 We first prove $k\in K$.  Assume for a contradiction that $k\notin K$, so $j \notin \tilde{K}$.  Equation (\ref{conjuginvariance}) implies $u^* \,m_j u = m_j$ for all $u\in \sU(\Q)$ and taking any non-zero spectral projection of the unbounded operator $m_j$ associated to $\Q'\cap \innpr{\M\oo\M, e_{\C_j\oo\M^{\oo2}_{\hat{j}}}}$ we obtain by Popa's intertwining techniques that $\Q\prec \C_j\oo\M^{\oo2}_{\hat{j}}$.  This intertwining together with the $*$-isomorphism and intertwining in (\ref{intert1here}) implies that $\Delta(\cM_{\hat{i}}) \prec \C_j\oo\M^{\oo2}_{\hat{j}} = \cM \oo \C_k \oo \cM_{\hat{k}}$ (or possibly $= \C_k \oo \cM_{\hat{k}} \oo \cM$) where $k\notin K$, thus contradicting the maximality of $K$ from Lemma \ref{proptfirst intertwining}.

 Now that we have $k\in K$, we assume without loss of generality that $j=(k,2)$.  Equation (\ref{conjuginvariance}) now implies $\alpha^j_t(u) m_j = m_j u$ for all $u\in \sU(\Q)$ and $0< t \leqslant t_\varepsilon$.  Moreover, $m_j \in \C_j\oo\M^{\oo2}_{\hat{j}}$ and the partial isometry $v_j$ from its polar decomposition satisfies $\alpha^j_t(u) v_j = v_j u$ for all $u\in \sU(\Q)$ and $0<t \leqslant t_\varepsilon$.  Therefore, Lemma \ref{ioanacontrol}  provides the existence of a finite subset $\tilde{F} \Subset I_k$ such that $\Q \prec \cM \oo \cL(A_k^{\tilde{F}}) \oo \cM_{\hat{k}}$ and again by (\ref{intert1here}) we obtain $\Delta(\cM_{\hat{i}}) \prec \cM \oo \cL(A_k^{\tilde{F}}) \oo \cM_{\hat{j}}$.  We can assume $\tilde{F} \subset I_k$ is minimal such that the intertwining happens.  Finally, we argue that $\Delta(\cM_{\hat{i}}) \prec_{\cM\oo\cM} \cM \oo \cM_{\hat{k}}$, i.e. $\tilde{F} = \emptyset$.  Indeed, assume for a contradiction that $\Delta(\cM_{\hat{i}}) \nprec_{\cM\oo\cM} \cM \oo \cM_{\hat{k}}$, then Theorem \ref{quasinormcontrol2} implies that $\QN_{\cM\oo\cM} (\Delta(\cM_{\hat{i}}))'' \prec \cM \oo \cL(A_k^{(I_k)} {\rm Norm}(\tilde{F})) \oo \cM_{\hat{k}}$.  However, since $\Delta(\cM) \subset \cN_{\cM\oo\cM} (\Delta(\cM_{\hat{i}}))''$, we would obtain $\Delta(\cM) \prec \cM \oo \cL(A_k^{(I_k)} {\rm Norm}(\tilde{F})) \oo \cM_{\hat{k}}$ which implies $A_k^{(I_k)} {\rm Norm}(\tilde{F})$ has finite index in $G_k$ contradicting the fact that $B_k\ca I_k$ has amenable stabilizers and $B_k$ is non-amenable. \end{proof}

\begin{thm}[{\cite[Theorem 3.6]{CU18}}]\label{TwoFoldedProduct}Assume $G= \oplus_{j\in J} G_j$ where the groups $G_j$'s are as in Notation A. Assume that $H$ is any group satisfying $\L(G)=\L(H)$. For each $j\in J\setminus\{0\}$ there is a decomposition $H=\Psi_j\oplus \Theta_j$, a scalar $t_j>0$ and $u_j\in \sU(\M)$ satisfying \begin{equation}u_j\L(G_j)^{t_j}u_j^*=\L(\Psi_j)\text{ and  }u_j\L(G_{J\setminus\{j\}})^{1/t_j}u_j^*=\L(\Theta_j).\end{equation} 
\end{thm}

\begin{proof} Since $G$ is as in Notation A then by Theorem \ref{intertwining} we have that for every $j\in J\setminus \{0\}$ there exists $k\in J$ such that $\Delta(\cM_{\hat j})\prec^s \cM \bar\otimes \cM_{\hat k}$. Then applying \cite[Theorem 3.6]{CU18}, whose proof uses only the previous intertwining, we get the desired conclusion. \end{proof}

\begin{thm}\label{infprod} Assume $G= \oplus_{j\in J} G_j$ where the groups $G_j$'s are as in Notation A. Assume that $H$ is an \emph{arbitrary} group satisfying $\cL(G)=\cL(H)$. 

Then $H$ admits a direct sum decomposition $H=(\oplus_{j\in J\setminus \{0\}} H_n) \oplus A$, where $H_j$ is an ICC property (T) group for all $j\in J\setminus \{0\}$ and $A$ is an ICC amenable group. Moreover, let $J_k \nearrow J\setminus \{0\}$ be an exhaustion by finite sets with $|J_k|=k$. Then  one can find a sequence $(t_j)_{j\in J\setminus \{0\}}$ of positive numbers such that for every $k\in \mathbb N$ there is a unitary $u_k\in \cL(H)$ so that 
\begin{equation}\label{intertwiningtheparts'}\begin{split}& u_k\cL(G_j)^{t_j}u_k^*=\cL(H_j)\quad \text{ for all }j\in J_k \text{; and }\\& u_k\cL((\oplus_{j\in \hat{J_k}} G_j )\oplus G_0)u_k^*=\cL((\oplus_{j \in \hat{J_k}} H_j) \oplus A).\end{split}
\end{equation} 

 Notice that when $J$ is finite, the sequence of equations above terminates at $k=|J\setminus \{0\}|$.

\end{thm}

\begin{proof} Let $J_k=\{j_1,j_2,j_3,...,j_k\}$ for all $k\in\mathbb N$.  Using Theorem \ref{TwoFoldedProduct} there exist a product decomposition $H=H_1 \oplus \Theta_1$, $v_1\in \sU(\M)$, and $t_1>0$ such that   $v_1\L(G_{j_1})^{t_1}v_1^* =\L(H_1)$ and $v_1\L(G_{\hat{J_1} \cup\{0\}})^{1/t_1}v^*_1=\L(\Theta_1)$. Since  $\L(G_{\hat{J_1} \cup\{0\}})^{1/t_1}$ is a McDuff factor, it is naturally isomorphic to $\L(G_{\hat{J_1} \cup\{0\}})$ and  applying Theorem \ref{TwoFoldedProduct} again in the prior relation for the group $G_{\hat{J_1} \cup\{0\}}$ there exist a product decomposition $\Theta_1=H_2\oplus \Theta_2$, $v_2\in\sU(v_1\L(G_{\hat{J_1} \cup\{0\}})^{1/t_1}v^*_1)$ , and $t_2>0$ such that  $v_2\L(G_{j_2})^{t_2}v_2^* =\L(H_2)$ and $v_2\L(G_{\hat{J_2} \cup\{0\}})^{1/(t_1t_2)}v^*_2=\L(\Theta_2)$.  Proceeding inductively one has $\Theta_{n-1}=H_n\oplus \Theta_n$, a unitary $v_n\in \sU(v_{n-1}\L(G_{\widehat{J_{n-1}} \cup\{0\}})^{1/(t_1t_2\cdots t_{n-1})}v^*_{n-1})$ and $t_n>0$ such that  $v_n\L(G_{j_n})^{t_n}v_n^* =\L(H_n)$ and $v_n\L(G_{\widehat{J_{n-1}} \cup\{0\}})^{1/(t_1t_2\cdots t_n)}v^*_n=\cL(\Theta_n)$. Altogether, these relations show that $\Theta_n\geqslant \Theta_{n+1}$ for all $n$ and also  $H=(\oplus_{n \in J\setminus\{0\}}H_n )\oplus A$, where $A =\cap_n \Theta_n$. In addition, for every $k\in \mathbb N$ letting $u_k:=v_1v_2\cdots v_k$ we see that 

\begin{equation}\label{intertwiningtheparts}
\begin{split}
&u_k \L(G_n)^{t_n}u^*_n=\L(H_n)\text{ for all }n\in J_k \text{ and}\\ & u_k\L((\oplus_{n\in \hat{J_k}} G_n )\oplus G_0)^{1/(t_1t_2\cdots t_k)} u^*_k=\L((\oplus_{n \in \hat{J_k}} H_n) \oplus A).
\end{split} 
\end{equation} 

Since $\L(G_k)$ is a II$_1$ factor the second relation in \eqref{intertwiningtheparts} show that for each $k\in \mathbb N$ one can find $u_k\in \sU(\cM)$ such that  $u^*_k\L(A)u_k\subseteq \L((\oplus_{n\in \hat{J_k}} G_n )\oplus G_0)^{1/(t_1t_2\cdots t_k)}$. Using \cite[Proposition 2.7]{CU18} it follows that $\L(A)$ is an amenable factor and $A$ is ICC amenable or trivial. Finally, since in all our cases, $\L((\oplus_{n\in \hat{J_k}} G_n )\oplus G_0)$ is a McDuff factor, relations \eqref{intertwiningtheparts} yield \eqref{intertwiningtheparts'}. \end{proof}

\noindent \emph{Proof of Theorem \ref{main2}.} Let $\theta\colon \cL(G) \ra \cL(H)$ be any isomorphism. Applying Theorem \ref{infprod}, there is a direct sum decomposition $H=(\oplus_{j\in J\setminus \{0\}} H_j) \oplus A$, with $H_j$ ICC property (T) group for all $j\in J\setminus \{0\}$ and $A$ ICC amenable. Moreover, let $J_k \nearrow J\setminus \{0\}$ be an exhaustion by finite sets with $|J_k|=k$, then  one can find a sequence $(t_j)_{j\in J\setminus \{0\}}$ of positive numbers such that for every $k\in \mathbb N$ there is a unitary $u_k\in \cL(H)$ so that 
\begin{equation}\label{intertwiningtheparts''}\begin{split}& u_k\theta(\cL(G_j))^{t_j}u_k^*=\cL(H_j)\quad \text{ for all }j\in J_k \text{; and }\\& u_k\theta(\cL((\oplus_{j\in \hat{J_k}} G_j )\oplus G_0))u_k^*=\cL((\oplus_{j \in \hat{J_k}} H_j) \oplus A).\end{split}
\end{equation} 
Since by \cite[Theorem 1.3]{CIOS21} every $G_j$ is $W^*$-superrigid, it follows that $t_j=1$ and $H_j \cong G_j$ for all $j\in J\setminus \{0\}$. Moreover, for every $j\in J\setminus \{0\}$ there is a group isomorphism $\delta_j \colon G_j \ra H_j$, a multiplicative character $\eta_j \colon G_j \ra \mathbb T$, and a unitary $x_j \in \cL(H)$, such that $\theta(u_g)= \eta_j(g)x_j v_{\delta(g)}x_j^*$ for all $g\in G_j$. $\hfill\qed$

\vskip 0.08in
We conclude this section by highlighting the following immediate corollary of Theorem \ref{infprod} in the case where the amenable group $G_0$ is nontrivial and the index set $J$ is finite.

\begin{cor}\label{main3} 
For every $1\leqslant j\leqslant n<\infty$ let $G_j\in \mathcal{WR}_b(A_j, B_j\ca I_j)$ be a property (T) wreath-like product group, where $A_j$ is a nontrivial abelian group, $B_j$ is an ICC subgroup of a hyperbolic group, and the action $B_j \ca I_j$ has amenable stabilizers. Let $C$ be any ICC amenable group and set $G= (\oplus^n_{j=1} G_j )\oplus C$.  

Let $H$ be any group and let $\theta\colon \L(G)\rightarrow \L(H)$ be a $\ast$-isomorphism. 

 Then $H$ admits a subgroup $K\cong \oplus^n_{j=1}G_j$ and an amenable ICC subgroup $A$ such that $H=K\oplus A$. In particular, we have $H\oplus C \cong G \oplus A$. 

 Moreover, one can find $x,y\in \mathscr{U}(\mathcal L(H))$, $\eta \in {\rm Char}(\oplus^n_{j=1} G_j)$, and $\delta \in {\rm Isom} (\oplus^n_{j=1}G_j,K)$ such that 
\begin{equation}\label{intertwiningtheparts-main2}
\begin{split}
&\theta(u_g)=\eta (g)\, x v_{\delta(g)} x^* \quad \text{for all } g\in \oplus^n_{j=1} G_j, \\
&\theta(\L(C))= y\L(A)y^*.
\end{split}
\end{equation}  
\end{cor}    

This result provides a setting in which the McDuff superrigidity phenomenon discussed in Definition ~\ref{genmdsup} holds with both ICC amenable groups $A$ and $B$ nontrivial.

In a different direction, Corollary~\ref{main3} can also be viewed as the $\mathrm{II}_1$ factor analogue of a recent rigidity result for group von Neumann algebras with diffuse center. Indeed, \cite[Theorem A]{CFQT23} shows that if $C$ is any infinite abelian group and $G\in \mathcal{WR}(A, B\ca I)$ is a property~(T) wreath-like product group (with $A$ non-trivial free abelian and $B\ca I$ as in Corollary~\ref{main3}), then for every group $H$ satisfying $\L(G\oplus C)\cong \L(H)$, one must have $H=K\oplus D$, where $K\cong G$ and $D$ is infinite abelian.

\subsection{Proof of Theorem \ref{Thm:main}}

Let $G$ be a torsion-free, non-trivial, hyperbolic group  with property (T). For example, we may start with a uniform lattice in $Sp(n, 1)$. By Selberg’s lemma, it contains a finite
index torsion-free subgroup $G$. Being a finite index subgroup of a uniform lattice, $G$ is also a uniform lattice in $Sp(n, 1)$ and, therefore, is hyperbolic and has property (T).

Fix any $g\in G\setminus \{ 1\}$. Suppose that Theorem \ref{Thm:main-gt-full} applies to every $n> N$ for some $N\in \NN$. Let $\mathcal P$ be the set of all primes $p> N$. For every infinite subset $I=\{ p_1, p_2, \ldots\} \subseteq \mathcal P$, Theorem \ref{Thm:main-gt-full} provides us with a sequence of wreath-like products $G_j=G/[\ll g^{p_j}\rr, \ll g^{p_j}\rr] $ satisfying all the assumptions of Theorem \ref{main1}; furthermore, the direct sum $G(I)=\bigoplus_{j\in \NN} G_j$ contains an element of prime order $p$ if and only if $p\in I$. In particular, $G(I_1)\not \cong G(I_2)$ whenever $I_1\ne I_2$. By Corollary \ref{main2}, this provides a continuum of pairwise non-isomorphic, $W^\ast $-McDuff groups that are McDuff superrigid.

\section{A relative solidity result for von Neumann algebras of wreath-like product groups with bounded $2$-cocycle}

We end this paper with a solidity result of independent interest which generalizes \cite[Corollary 4.7]{CIOS21} to the case of non-abelian amenable base groups.  It also generalizes the main result from \cite{CI08} for nonsplit generalized wreath product extensions.

\begin{thm} For $1\leqslant i\leqslant n<\infty$ let $G_i\in \mathcal{WR}_b( A_i, B_i\ca I_i)$ such that $A_i$ and ${\rm Stab}_{B_i}(k)$ are amenable for all $ k\in I_i$. Let $A_0$ be an amenable group. Consider $G=\times^n_{j=1} G_j\times A_0$ and denote by $\M= \cL(G)$.  Let $p \in  \cL(\times^n_{j=1} A^{(I_j)}_j \times A_0)$ be a projection.  If  $\Q\subseteq  p\cL( \times^n_{j=1} A^{(I_j)}_j \times A_0)p$ is a von Neumann algebra such that $\Q \nprec \cL(\times_{k\neq i} A^{(I_k)}_k\times A_0)$ for all $1\leqslant i\leqslant n$ and $\Q \nprec \cL(\times^n_{k=1} A^{(I_k)}_{k})$ then $\Q'\cap p \M p$ is amenable.
\end{thm}

\begin{proof} Let $A=\times^n_{j=1} A_j^{(I_j)}\times A_0$ and $B=\times^n_{j=1} B_j^{(I_j)}$. We will prove our statement by induction on $n$.

Let $z \in \mathcal{Z}(\Q' \cap p\M p)$ be the maximal central projection such that $(\Q' \cap p \M p) z$ is amenable.  We shall prove that $z=p$.  Assume by contradiction that $z\neq p$ and let $0\neq y:=p-z$.  Thus  $ (\Q' \cap  p \M p)y$ is non-amenable.  Let $\ma$ be the array on $G= \times^n_{j=1} G_j \times A_0$ in a weakly-$\ell^2$ representation constructed in Proposition \ref{relsolarray} and let $V_t^\ma$ be the corresponding deformation introduced in \cite{CS11}.   Using the spectral gap property from part \ref{spectralgap} in Proposition \ref{propo:Gaussian_properties},  as $ (\Q' \cap p\M p )y$ is non-amenable,  it follows that $\lim_{t\ra 0}(\sup_{x\in (\Q)_1}\|x y -V_t^\ma(xy )\|_2)=0$.  
Fix $\varepsilon>0$. Using the same argument to approximate $y$ as used in the proof of Theorem \ref{intertwining} (page \pageref{eq approximIntertwiner}) to approximate $vv^*$, one can find a finite set $R_\varepsilon\subset B$ and $x_h\in \cL( A)$ with 
\begin{equation}\label{approximantsize}
\|y - \sum_{h\in R_\varepsilon} x_h u_h \|_2\leqslant \varepsilon\text{ and }\|\sum_{h\in R_\varepsilon} x_h u_h\|_\infty\leqslant 1\end{equation} and a scalar $t_\varepsilon>0$ such that

\begin{equation}\label{ineq9}
    \| xy -V_t^{\ma}(xy) \|_2^2 = \sum_{h\in R_\varepsilon} \|x x_h- V_t^{\ma}(xx_h)\|_2^2\leqslant 2\varepsilon^2, 
\end{equation}
for all $x\in (\Q)_1$ and all $0<t\leqslant t_\varepsilon$.

 Then proceeding through exactly the same type of arguments as on the first half of page \pageref{fourthestimate} in the proof of Theorem \ref{intertwining}, after shrinking $t_\varepsilon$ if necessary, \eqref{ineq9} implies that for all $x\in \mathscr U(\Q)$ and  $0<t\leqslant t_\varepsilon$ we have  

 \begin{equation}\label{ineq7}  \Norm{y}_2^2 -3\varepsilon^2 - \varepsilon/4 \leqslant  \sum^n_{j=1} \langle \alpha^j_t\otimes {\rm Id} (x) ),x(\sum_{h\in R_\varepsilon}x_h x_h^*)\rangle . \end{equation}

\noindent Since $y$ is a projection, using the estimates \eqref{approximantsize} and the triangle inequality we see that  \begin{equation}\label{ineq8}\begin{split}
&\|E_{\cL( A)}(y)- \sum_h x_hx_h^*\|_2=\|E_{\cL( A)}(yy^*)- E_{\L(A)}(\sum_{h,k\in R_\varepsilon} x_h u_h u_{k^{-1}}x_k^*)\|_2\leqslant \|yy^*- \sum_{h,k\in R_\varepsilon} x_h u_h u_{k^{-1}}x_k^*\|_2\\ 
&\leqslant \|yy^*- \sum_{h,k\in R_\varepsilon} x_h u_h u_{k^{-1}}x_k^*\|_2 \leq\|y- \sum_{h\in R_\varepsilon} x_h u_h \|_2 + \| \sum_{h\in R_\varepsilon} x_hu_h\|_\infty \|y^*- \sum_{k \in R_\varepsilon} u_{k^{-1}}x_k^*\|_2\leqslant 2\varepsilon.\end{split}
\end{equation} 

\noindent Using \eqref{ineq8} and the triangle inequality in \eqref{ineq7} we further get for all $x\in \mathscr U(\Q)$ and $0<t\leqslant t_\varepsilon$ we have

 \begin{equation}\label{ineq7'}  \Norm{y}_2^2 -3\varepsilon^2 - \varepsilon/4 -2n\varepsilon\leqslant  \sum^n_{j=1}\langle \alpha^j_t\otimes {\rm Id} (x) ),xE_{\cL( A)}(y)\rangle. \end{equation} 
\noindent Consider the  support $s:= \text{sup}(E_{\L( A)}(y))\in \Q'\cap p\L( A)p$ and notice that $s\geqslant y, E_{\L( A)}(y)$. Using the facts that $s E_{\L( A)}(y)= E_{\L( A)}(y)$,  $\|\alpha^i_t\otimes {\rm Id} (s)-s\|_2\overset{t\ra 0}{\ra} 0$ and  $(\alpha^j_t\otimes {\rm Id}) (xs) =  (\alpha^j_t\otimes {\rm Id} )(x)(\alpha^j_t\otimes {\rm Id}) (s)$ in \eqref{ineq7'}, after shrinking $0<t_\varepsilon$ if necessary, we further get that for all $x\in \mathscr U(\Q)$ and $0<t\leqslant t_\varepsilon$ we have

\begin{equation}\label{ineq7''}  \Norm{y}_2^2 -3\varepsilon^2 - \varepsilon/4 -3n\varepsilon\leqslant  \sum^n_{j=1}\langle \alpha^j_t\otimes {\rm Id} (x s) ),xs E_{\cL( A)}(y)\rangle. \end{equation} 

 Letting $\varepsilon>0$ sufficiently small we can assume  $\Norm{y}_2^2 -3\varepsilon^2 - \varepsilon/4 -3n\varepsilon=:D>0$ and thus 

\begin{equation}\label{ineq7'''}  D\leqslant  \sum^n_{j=1}\langle \alpha^j_t\otimes {\rm Id} (x) ),x E_{\cL( A)}(y)\rangle, \quad \text{ for all } x\in \mathscr U(\Q s), 0<t\leqslant t_\varepsilon. \end{equation}

 Next consider the von Neumann algebra $\mathcal S = \bigoplus^n_{j=1} \left(\cL((A_j\ast \mathbb Z)^{(I_j)})\oo \cL(\times_{i\neq j} G_j \times A_0)\right )$ endowed with the trace $\tau_{\mathcal S}=\bigoplus^n_{j=1} \tau_{\cL((A_j\ast \mathbb Z)^{(I_j)})\oo \cL(\times_{i\neq j} G_j \times A_0)}$. Consider the elements $\xi =\bigoplus^n_{j=1}1$, $\eta = \bigoplus^n_{j=1}E_{\L(A)}(y)\in L^2(\mathcal S, \tau_{\mathcal S})$. Let $\phi$ be the linear functional given by the inner product with $\eta$.  Thus inequality \eqref{ineq7} can be rewritten as  $$D\leqslant \phi((\oplus^n_{j=1}x^* )\xi (\oplus^n_{j=1} (\alpha_t^j\otimes {\rm Id}(x))). $$ 
 
 Now consider the closed convex hull
 $$\mathcal K= \overline{\rm co}\{(\oplus^n_{j=1} x^*)\xi (\oplus^n_{j=1} \alpha_t^j\otimes {\rm Id}(x))\,:\, x\in \mathscr U(\Q s), 0<t\leqslant t_\varepsilon\}\subset L^2(\mathcal S, \tau_{\mathcal S}).$$

Let $m\in \mathcal K$ be the unique element of minimal $\|\cdot\|_{\tau_{\S},2}$-norm.  Uniqueness of $m$ implies that
\begin{equation}\label{conjuginvariance'}
(\oplus^n_{j=1} x^*)\, m \,(\oplus^n_{j=1} \alpha_t^j\circ {\rm Id}(x)) = m,
\end{equation}
for any $u\in \sU(\Q s)$ and $0< t \leqslant t_\varepsilon$. Moreover, from \eqref{ineq7'''} we observe that $0<D \leqslant \varphi(m \cdot \eta)$, so $m = m_1 \oplus \cdots \oplus  m_n \neq 0$ and there is $1\leqslant i\leqslant n$  such that $m_{i} \neq 0$.

 Equation (\ref{conjuginvariance'}) now implies $m_i \alpha^i_t\otimes {\rm Id} (u) = u m_i$ for all $u\in \sU(\Q s)$ and $0< t \leqslant t_\varepsilon$.  Moreover, $m_i \in \cL((A_i\ast \mathbb Z)^{(I_i)})\oo \cL(\times_{k\neq i}G_k \times A_0)$ and the partial isometry $v_i$ from its polar decomposition satisfies $v_i\alpha^i_t\otimes {\rm Id} (u) = uv_i$ for all $u\in \sU(\Q s)$.  Thus by Lemma \ref{ioanacontrol}  there is a finite subset $F \Subset I_i$ such that $\Q s \prec \cL(A_i^{F}) \oo \cL(\times_{k\neq i} G_k\times A_0)$. Notice that the hypothesis conditions also imply that $F\neq \emptyset$. We can assume $F \subset I_i$ is minimal such that the intertwining happens.  Moreover,  since $\Q s \subseteq p\cL(A)p$ we conclude that $\Q s\prec \cL(A_i^{F} \times (\times_{k\neq i} A^{(I_k)}_k) \times A_0))=:\mathcal P_i$. Moreover, since $\L(A)$ is regular in $\M$, this further implies that $\Q s\prec_{\L(A)} \mathcal P_i$.  Hence there are projections $e\in \Q, f\in \mathcal P_i$ with $es\neq 0$, a nonzero partial isometry $w\in es \L(A) f$ and a $\ast$-isomorphism onto its image $\phi\colon e\Q es\ra \mathcal R:= \phi(e\Q e s) \subseteq f\P_if$ satisfying \begin{equation}\label{intrel10}\phi (x)w=wx,\text{ for all }x\in e\Q es.\end{equation} 
 
Next observe the hypothesis conditions imply that $\mathcal R \nprec \L( \times_{k\neq i} A_k^{(I_k)}\times A_0 )
$ and thus by Theorem \ref{quasinormcontrol2} we have $ \mathcal R \vee (\mathcal R '\cap f\M f)\subset \mathscr {QN}_{f \M f}(\mathcal R)''\subseteq \L(A^{(I_i)}_{i}{\rm Norm}(F)\times (\times_{k\neq i} G_k)\times A_0 )$.

 Henceforth, denote by $C:=A^{(I_i)}_i{\rm Norm(F)} \times A_0$. Since all $A_j$'s are amenable and $B_i\ca I_i$ has amenable stabilizers it follows that $C$ is amenable. In conclusion, we obtained two commuting von Neumann subalgebras  $\mathcal R ww^*, ww^*(\mathcal R '\cap f\M f )ww^*\subseteq  \L(\times_{k\neq i} G_k \times C)$ with   $\mathcal R ww^*\subseteq \L(\times_{k\neq i} A_k^{(I_k)} \times C)$. 

 Notice that in the case $n=1$ we have $i=1$ and hence  $\times_{k\neq i} A_k^{(I_k)} \times C=C$ which is an amenable group. Hence, the prior containment already imply that  $ww^*(\mathcal R '\cap f\M f )ww^*$ is amenable. 

 Next we assume that $n\geqslant 2$. In this case the hypothesis assumptions imply that $\mathcal R ww^*\nprec_{\cL(\times_{k\neq i} G_k\times C)} \L(\times_{k\neq i,l} G_k \times C)$ for $l\in \{1,...,n\}\setminus \{i\}$ and also $\mathcal R ww^*\nprec_{\cL(\times_{k\neq i} G_k\times C)} \L(\times_{k\neq i} G_k)$.
Therefore, using the inductive assumption for $n-1$ we again have that $ww^* (\mathcal R'\cap f \M f)ww^*$ is amenable.

 However, the intertwining relation \eqref{intrel10} implies that 
\[
ww^* \, (\mathcal R' \cap f \M f) \, ww^* = w \, \bigl( (eQe s)' \cap es \M es \bigr) \, w^* .
\]
Therefore,
\[
w^*w \, \bigl( (eQe s)' \cap es \M es \bigr) \, w^*w
= w^*w \, (\Q' \cap p \M p ) \, w^*w
\]
is also amenable.  

 Now we argue that $w^*w y \neq 0$. Assume, by contradiction, that $w^* w y = 0$. Applying the conditional expectation, we obtain $w^*w E_{\L(A)}(y) = 0$. Hence $w^*w s = 0$. Since $w^*w \leqslant es$, this forces $w^*w = 0$, which would imply $w=0$, a contradiction.  

 Finally, since $w^*w (\Q \cap p \M p ) w^*w$ is amenable and $w^*w y \neq 0$, it follows that the algebra $(\Q' \cap p \M p ) y$ has a nonzero amenable corner. This contradicts the maximality of the projection $z$ chosen at the beginning of the proof. Therefore we have $z=p$, showing the inductive step. \end{proof}

\end{document}